\documentclass[12pt]{amsart}
\usepackage{amsmath,amssymb,latexsym, amsthm, amscd, mathrsfs, stmaryrd} %, txfonts}

\usepackage[mathscr]{euscript}

\usepackage{tikz}

\usepackage[Symbol]{upgreek}
\usepackage{tabularx,ragged2e,booktabs,caption}

\usepackage[all]{xy}
\usepackage{color}
\usepackage{hyperref}

\theoremstyle{definition}
\newtheorem{Def}[subsubsection]{Definition}%[section]
\newtheorem{example}[subsubsection]{Example}%[section]
\newtheorem{rem}[subsubsection]{Remark}%[section]
\newtheorem{ass}[subsubsection]{Assumption}
\theoremstyle{plain}
\newtheorem{prop}[subsubsection]{Proposition}
\newtheorem{thm}[subsubsection]{Theorem}
\newtheorem{lem}[subsubsection]{Lemma}
\newtheorem{cor}[subsubsection]{Corollary}
\newtheorem{conj}[subsubsection]{Conjecture}
\newcommand{\Hom}{\mathrm{Hom}}
\newcommand{\mbf}{\mathbf}
\newcommand{\mbb}{\mathbb}
\newcommand{\mrm}{\mathrm}

\newcommand{\F}{\mathcal F}
\newcommand{\G}{\mathrm G}
\renewcommand{\H}{\mrm H}
\newcommand{\K}{\mathscr K}
\newcommand{\M}{\mathfrak M}
\newcommand{\N}{\mathcal N}
\renewcommand{\P}{\mrm P}

\newcommand{\bF}{\mbf F}

\newcommand{\bM}{\mbf M}
\newcommand{\fP}{\mathfrak P}
\newcommand{\fS}{\mathfrak S}
\newcommand{\fY}{\mathfrak Y_{\zeta}}

\newcommand{\bv}{\mbf v}
\newcommand{\bw}{\mbf w}
\newcommand{\bx}{\mbf x}
\newcommand{\by}{\mbf y}
\renewcommand{\o}{\text{{\bf o}}}
\renewcommand{\i}{\text{{\bf i}}}
\newcommand{\ve}{\varepsilon}
\newcommand{\z}{\zeta_{\mbb C}}

\newcommand{\GL}{\mrm{GL}}

\newcommand{\IC}{\mrm{IC}}

\renewcommand{\r}{\mbf r}
\newcommand{\s}{\mbf s}

\newcommand{\SP}{\mrm{Sp}}
\renewcommand{\sp}{\mathfrak{sp}}

\newcommand{\R}{\mathscr R}
\renewcommand{\S}{\mathscr S}
\newcommand{\T}{\mrm T}
\newcommand{\bT}{\mbb T}

\newcommand{\w}{\upomega}
\newcommand{\Y}{\mathscr Y}

\newcommand{\Z}{\mathfrak Z_{\zeta}}

\newcommand{\End}{\mrm{End}}
\newcommand{\Ext}{\mrm{Ext}}

\newcommand{\Stab}{\mrm{Stab}}

\newtheorem{thrm}{Theorem}

\newtheorem*{conjecture*}{Conjecture}

\title[Quiver varieties and symmetric pairs]{Quiver varieties and symmetric pairs}

\author{Yiqiang Li}
\address{
Department of Mathematics\\
University at Buffalo\\
the State University of New York
}
\email{yiqiang@buffalo.edu}


 %Symplectic manifolds, general

\begin{document}

\begin{abstract}
We study
fixed-point loci of Nakajima varieties under symplectomorphisms and their anti-symplectic cousins, 
which are compositions of  a diagram isomorphism, a reflection functor and a transpose defined by certain  bilinear forms.
These subvarieties provide a natural home for geometric representation theory of symmetric pairs. 
In particular, the cohomology of a Steinberg-type variety of the symplectic fixed-point subvarieties 
is conjecturally related to  the universal enveloping algebra of the subalgebra in a symmetric pair.
The latter symplectic subvarieties are further 
used to construct geometrically an action of a twisted Yangian on torus equivariant cohomology of Nakajima varieties.
In type $A$ case, these subvarieties provide a quiver model
for partial Springer resolutions of nilpotent Slodowy slices of classical groups and associated symmetric spaces, which leads to 
a rectangular symmetry and a refinement of Kraft-Procesi row/column removal reductions.
\end{abstract}

\maketitle

\section{Introduction}

To a  Dynkin diagram of ADE type, one can attach a simply-laced complex simple Lie algebra, say $\mathfrak g$,  and a class of  Nakajima's quiver varieties~\cite{N94, N98}.
The latter provides a natural home for a geometric representation theory of the former. 
If  the algebra $\mathfrak g$ is further equipped with an involution, it yields a  complex Cartan decomposition of $\mathfrak g$:
\begin{align}
\label{Cartan-dec}
\mathfrak g= \mathfrak k \oplus \mathfrak p
\end{align}
where $\mathfrak k$ is the fixed-point subalgebra under involution,  and $\mathfrak p$ is the eigenspace of eigenvalue $-1$.
The pair $(\mathfrak g, \mathfrak k)$ is a so-called symmetric pair and $\mathfrak p$ is the associated  symmetric space.
The purpose of this paper 
is to develop a geometric  theory for the symmetric pair $(\mathfrak g, \mathfrak k)$ and its symmetric space $\mathfrak p$, by using
 Nakajima varieties together with
their  fixed-point loci  under certain symplectic and anti-symplectic involutions. 

Thanks to \'E. Cartan, the classification of symmetric pairs is equivalent to the classification of real simple Lie algebras, which is
given  by Satake diagrams~\cite{H, OV}. These 
are bicolor Dynkin diagrams with black or white vertices, equipped with  diagram involutions.
Representation theory of symmetric pairs was developed under the influence
of
Harish-Chandra's theory of 
$(\mathfrak g, K_{\mbb R})$-modules with  $K_{\mbb R}$ a real adjoint group of $\mathfrak k$ (\cite{D}).
A quantum version was  obtained later 
by Letzter in~\cite{Le}, 
where   a coideal subalgebra $\mbf U'_q(\mathfrak k)$ 
of the quantum algebra $\mbf U_q(\mathfrak g)$ is used as a $q$-analogue of the universal enveloping algebra of $\mathfrak k$ .

Recently, the algebra $\mbf U'_q(\mathfrak k)$ of  type  AIII/AIV without black vertices found its applications
in the study of   orthosymplectic Lie superalgebras  
%$\mathfrak{osp}_{2m+1 | 2n}$ 
by Bao and Wang~\cite{BW13}  and 
even special orthogonal Lie algebras by Ehrig and Stroppel~\cite{ES13}, independently. 
A new canonical basis was constructed for certain tensor modules of $\mbf U'_q(\mathfrak k)$ in~\cite{BW13}, 
for idempotented $\mbf U'_q(\mathfrak k)$ in ~\cite{BKLW, LW15}, and finally 
a general theory of canonical basis for $\mathfrak k$ of any type was obtained in ~\cite{BW16}. 
These works have inspired many  developments in various directions, such as categorification~\cite{BSWW} and K-matrix~\cite{BaK16}.

In the work~\cite{BKLW}, there is a geometric realization
of $\mbf U'_q(\mathfrak k)$ of type AIII/AIV without black vertices by using $n$-step isotropic flag varieties,
in the spirit of  Beilinson, Lusztig and MacPherson's influential work~\cite{BLM90}.
In light of the role of {\it loc. cit.} in the works~\cite{G91} by Ginzburg and ~\cite{N94, N98} by Nakajima,
the geometric favor in~\cite{BKLW}, as the tip of the iceberg,  strongly suggests  the existence of a new class of quiver varieties  for a general $\mathfrak k$
parallel to Nakajima varieties for $\mathfrak g$. 
Such an existence is  conjectured independently, and maybe earlier,  
by Wang through his iProgram in~\cite[Introduction]{BW13}.
This new class of quiver varieties, called $\sigma$-quiver varieties, turns out to be 
fixed-point subvarieties of Nakajima varieties under  certain symplectic involutions.
More precisely, the symplectic involution $\sigma$ is a composition of a  diagram involution, a reflection functor and a symplectic transpose induced from certain bilinear forms. 
Note that the prototype of the involution $\sigma$ has been used in~\cite[Section 9]{N03} (see also~\cite[4.6]{VV03})
for reinterpreting Lusztig's opposition~\cite{L00b}, which serves as a crucial ingredient in 
a construction of canonical bases. 
As we learned from~\cite{N18} and via private communication, it is known to Nakajima that 
in an affine analogue of ~\cite[A(iv)]{N15}, fixed-point subvarieties of $\sigma$ on the regular parts of Nakajima varieties 
provide a quiver model for $\mathrm{SO/Sp}$-instantons moduli spaces on ALE spaces, 
similar to the instantons-moduli-space origin~\cite{KN90} of Nakajima varieties; see also Remark~\ref{|a|=2} (3).

Just like Nakajima varieties, 
type $A$ $\sigma$-quiver varieties possess many desirable properties.

\begin{thrm}[{Theorems ~\ref{sigma-sigma-1},~\ref{i-Nakajima-Maffei}, Corollary ~\ref{cor:i-Nakajima-Maffei}}]
\label{a}
Nilpotent Slodowy slices of $\mathfrak k$ of type $\mrm{AI}/\mrm{AII}$ and  their partial Springer resolutions are examples of type $\mrm A$ $\sigma$-quiver varieties.
%Here `partial' refers to the parabolic analogues of Springer resolutions.
\end{thrm}

In type AI/AII, the algebra $\mathfrak k$ is an orthogonal/symplectic Lie algebra, and thus we recover the geometry used in~\cite{BKLW}. 
Theorem~\ref{a} is a classical analogue of the well-known Nakajima-Maffei theorem that Nakajima varieties of type A
are  nilpotent Slodowy slices of $\mathfrak{sl}_n$ and their partial Springer resolutions~\cite{N94, M05}.
There are two easy-but-interesting applications from Theorem~\ref{a}. 
Note that the $\mathfrak{sl}_n$-version, presented in Sections~\ref{rect} and \ref{col}, has been done by Henderson~\cite{H15}.
The first one is a  symmetry in classical groups.

\begin{thrm}[{Theorem ~\ref{Sp-O}, Remark ~\ref{rem:Sp-O}}]
\label{b}
There is a rectangular symmetry for partial Springer resolutions of nilpotent Slodowy slices of classical groups, that is
if the partitions involved can be fit into a rectangle of a certain size (see Figure~\ref{rect-sym}), then the associated  varieties  are isomorphic. 
\end{thrm}

The rectangular symmetry is further applied to prove a conjecture in Henderson and Licata's work~\cite{HL14} on 
Springer resolutions of two-row nilpotent Slodowy slices of classical groups, 
and recover relevant results in {\it loc. cit.} and~\cite{W15}; see Example~\ref{2-row}. 

The second one is an enhancement of  
Kraft-Procesi's column/row removal reductions which play critical roles in the study of minimal singularities in classical nilpotent orbits.
Kraft-Procesi~\cite{KP82} showed smooth equivalences of 
singularities between  nilpotent Slodowy slices  $S^{\mathfrak g}_{\mu', \lambda}$ and 
$S^{\mathfrak g'}_{red(\mu'), red( \lambda)}$ for certain classical Lie algebras $\mathfrak g$ and $\mathfrak g'$, where $red(\mu')$, (resp. $ red( \lambda)$) is obtained from partition $\mu'$ (resp. $\lambda$)
by removing certain rows/columns from $\mu'$ (resp. $\lambda$), see Figure~\ref{CRR}. 

\begin{thrm}[Propositions \ref{i-prop:col}, ~\ref{i-prop:row}]
\label{c}
The nilpotent Slodowy slices  in Kraft-Procesi's  column/row removal reductions in~\cite{KP82} are isomorphic. 
\end{thrm}

Nilpotent orbits and their intersections with Slodowy slices in $\mathfrak k$ have been studied 
via categorical quotients in~\cite{KP82, K90}, \cite[Remark 8.5 (4)]{N94}, and~\cite[Appendix A(i)-A(iv)]{N15}. 
The latter approach, which is quite restricted,  is closely related to the approach we take in this paper.
They represent two different orders of taking GIT quotients and taking fixed points. 
A closed immersion, which is  conjecturally isomorphic, between varieties appeared  from these two approaches 
is established in Proposition~\ref{prop:iota-1}.

In a parallel  direction, nilpotent orbits and, more generally, nilpotent Slodowy slices 
in the symmetric space $\mathfrak p$ have been studied by Kostant and Rallis~\cite{KR71}, Sekiguchi~\cite{S84} and Ohta~\cite{O86}.
They have important applications in the orbit method of real reductive groups~\cite{V, V89}.
A slight alteration of the transpose  in the involution $\sigma$ yields an anti-symplectic involution $\hat \sigma$. 
Its fixed-point subvariety, called a $\hat \sigma$-quiver variety,  can be regarded as the quiver variety for the symmetric space $\mathfrak p$,
since results similar to Theorems~\ref{a}-\ref{c} remain valid in this setting 
(see Section~\ref{sQV+QSS}). 
 To this end, the  geometries surrounding nilpotent elements in the triple $(\mathfrak g, \mathfrak k, \mathfrak p)$ in (\ref{Cartan-dec}) 
 have their quiver counterparts: 
\[
\mbox{Nakajima varieties}, \ \mbox{symplectic subvarieties}, \  \mbox{Lagrangian subvarieties}.
\]

 With Theorems~\ref{a}-\ref{c} in hand, 
 it is expected that there is a geometric representation theory for $\mathfrak k$ via general $\sigma$-quiver varieties, 
 parallel to Nakajima's original theory for $\mathfrak g$.
A further study shows that $\sigma$-quiver varieties and their Lagrangian cousins admit many favorable properties inherited from ambient Nakajima varieties.
In particular,  they are nonsingular, if the ambient Nakajima variety is so, 
and they carry a Weyl group action. The new Weyl groups   contains Weyl groups of type $\mrm B_{\ell}/\mrm C_{\ell}/\mrm F_4$. A Weyl group action of type $\mrm G_2$ is realized by using an 
automorphism
of order 6 on Nakajima varieties. 
Furthermore, a conjecture is formulated in the following, with supporting evidences given in Proposition~\ref{coideal} and (\ref{j}).

\begin{conjecture*}[{Conjecture \ref{conj-1}}]
Let $(\mathfrak g, \mathfrak k)$ be a symmetric pair listed in Table 1 in Section~\ref{sec:coideal}.
There is a nontrivial algebra homomorphism from the  enveloping algebra of $\mathfrak k$ to 
the top Borel-Moore homology of the $\sigma$-fixed-point, for a certain $\sigma$,  
of a Steinberg-type variety in the setting of Nakajima varieties.
\end{conjecture*}

From the table, one observes that Conjecture~\ref{conj-1}, if holds,  would provide a new geometric construction of 
the universal enveloping algebra of  simple Lie algebras of type B and their representations.
In addition to developing a geometric/quiver theory of $\mathfrak k$ (and $\mathfrak p$), there is a substantial interest in
making connection with the original Nakajima theory  
for $\mathfrak g$ to have a more interesting theory for  $(\mathfrak g, \mathfrak k)$-modules.
The following theorem reflects such a flavor and is obtained 
by applying the machinery of Maulik-Okounkov's R-matrix~\cite{MO12, N16} to $\sigma$-quiver varieties.

\begin{thrm}[{Theorem~\ref{GSP}}]
\label{d}
There is a $(\Y(\mathfrak g), \Y_{\sigma})$-action on the localized torus equivariant cohomology of Nakajima varieties, where
$\Y(\mathfrak g)$ is the Yangian of $\mathfrak g$ and $\Y_{\sigma}$ is a twisted Yangian constructed
in this paper via a geometric $\K$-matrix.
\end{thrm}

The twisted Yangian $\Y_{\sigma}$ should coincide with its algebraic counterpart, which can be traced back
to Cherednik's work~\cite{Ch84}; see~\cite{M07, GRWa, GRWb}.
It is our hope that Theorem~\ref{d} will serve as a small step towards a geometric theory of $(\mathfrak g, K_{\mbb R})$-modules,
which in turn will shed light on that of  unitary representations of  the associated real simple group.

Finally, we caution the reader that in the main body of the paper the automorphisms $\sigma$ and
$\hat \sigma$ do not have to be involutive and the underlying graph is not necessarily of type ADE.

\subsection{Acknowledgements}

We thank Weiqiang Wang for fruitful collaborations, especially the work ~\cite{BKLW}, and several enlightening conversations.  
We also thank Dave Hemmer, Jiuzu Hong, Jim Humphreys,   Ivan Losev, George Lusztig, Hiraku Nakajima and Catharina Stroppel  for stimulating discussions.

Results in this paper have been announced in the ICRT VII at Xiamen, China, July 2016,  the AMS special session on ``Geometric methods in representation theory,'' Charleston, NC, March 2017,  the Taipei workshop on Lie superalgebras and related topics, 
National Center for Theoretical Sciences, Taipei; 
Colloquia in Shanghai Jiaotong University and Xiamen University, July 2017 and the Algebra Seminar at University of Virginia, November 2017. 
It is a pleasure to  thank the organizers for the invitation. 

We thank the anonymous referees for their careful readings, helpful comments and insightful suggestions.
This work is partially supported by the National Science Foundation under the grant DMS 1801915.

\tableofcontents

\section{Nakajima varieties}

In this section, we recall Nakajima's quiver varieties from the works~\cite{N94, N96, N98}.

\subsection{Graph}

Let $\Gamma$ be a graph without loops, with
$I$ and $H$ being  the vertex and arrow set, respectively.
For each arrow $h$, let $\o (h)$ and $\i (h)$ be its outgoing and incoming vertex so that we can depict $h$ as
$\o(h) \overset{h}{\to} \i (h)$.
There is an involution on the arrow set $\bar \empty: H \to H$, $h \mapsto \bar h$ such that
$\o (\bar h) = \i (h) $ and $\i (\bar h) = \o (h)$.

Let $\mbf C=(c_{ij})_{i,j\in I}$ be the Cartan matrix of the graph $\Gamma$ defined by
\begin{align}
c_{ij}
= 2 \delta_{i, j}  - \# \{ h\in H | \o (h)=i, \i (h) = j\}.
\end{align} 
For each $i\in I$, we define a bijection $s_i: \mbb Z^I \to \mbb Z^I$ by
$s_i (\xi) = \xi'$ where $\xi'_j = \xi_j  - c_{ji} \xi_i $, $\xi = (\xi_j)_{j\in I}$, $\xi'=(\xi_j')_{j\in I} \in \mbb Z^I$.
Let $\mathcal W$  be the the subgroup of $\mrm{Aut} (\mbb Z^I)$ generated by $s_i$ for all $i\in I$. 
The group $\mathcal W$ is the Weyl group of $\Gamma$.
It admits a presentation with generators $s_i$ for all $i\in I$ and 
the following defining relations.
\begin{align*}
s_i^2 & =1,   && \forall \ i\in I. \\
s_i s_j &  = s_j s_i, &&  \mbox{if} \ c_{ij} =0.\\
s_i s_j s_i & = s_j s_i s_j, &&  \mbox{if} \ c_{ij} =-1.
\end{align*}

For a fixed $\bw=(\bw_i)_{i\in I} \in \mbb Z^I$, we define a second (affine) $\mathcal W$-action on $\mbb Z^I$ by
$s_i * \bv = \bv'$, where $\bv=(\bv_i)_{i\in I}, \bv' =(\bv'_i)_{i\in I} \in \mbb Z^I$ such that 
$\bv'_i = \bv_i - \sum_{j\in I} c_{ij} \bv_j + \bw_i$ and $\bv'_j = \bv_j$ if $j\neq i$.
We will put a subscript $\bw$ under $*$, that is  $s_i *_{\bw}\bv$, if needed.
If $w=s_{i_1} s_{i_2}\cdots s_{i_l}$ is a sequence of simple reflections, 
we set $w* \bv = s_{i_1} * \cdots * s_{i_l} * \bv$.
We have
\begin{align}
\label{s*v}
\mbf C ( s_i *_{\bw} \bv)  = s_i ( \mbf C \bv - \bw ) + \bw.
\end{align}

\subsection{The variety $\Lambda_{\zeta_{\mbb C}} (\bv, \bw)$}
\label{lavw}

Let $V=\oplus_{i\in I} V_i$ and $W=\oplus_{i\in I}W_i$ be two finite dimensional $I$-graded 
vector spaces over the complex field $\mbb C$
of dimension vectors $\bv = (\bv_i)_{i\in I}$ and $\bw =(\bw_i)_{i\in I}$, respectively.
We consider the vector space
\[
\bM(\bv, \bw) \equiv \bM(V, W) 
=\oplus_{h\in H} \Hom (V_{\o (h)}, V_{\i (h)}) \oplus \oplus_{i\in I} \Hom(W_i, V_i) \oplus  \Hom (V_i, W_i).
\]
A typical element in $\bM(\bv, \bw)$ will be denoted by $\bx \equiv (x, p, q) \equiv (x_h, p_i, q_i)_{h\in H, i\in I}$, where
$x_h \in \Hom (V_{\o (h)}, V_{\i(h)})$, $p_i \in \Hom (W_i, V_i)$ and $q_i \in \Hom (V_i, W_i)$.

Let 
\begin{align}
\label{Gv}
\G_{\bv} \equiv \G_V  = \prod_{i\in I} \GL (V_i), \quad 
\G_{\bw} \equiv \G_W  = \prod_{i\in I} \GL(W_i).
\end{align}
The group $\G_{\bv}$  
acts from the left on $\bM(\bv, \bw)$ by conjugation. More precisely,  
for all $g=(g_i)_{i\in I} \in \G_{\bv}$ and $\bx \in \bM(\bv, \bw)$, we define
$g. \bx = \bx' \equiv (x'_h, p'_i, q'_i)$ where
$x'_h = g_{\i (h)} x_h g^{-1}_{\o (h)}$, $p'_i = g_i p_i$ and $q'_i = q_i g^{-1}_i$ for all $h\in H$ and $i\in I$.
Similarly, let $\G_{\bw}$ acts conjugately on $\bM(\bv, \bw)$ from the left, i.e.,
for any $\mrm f=(\mrm f_i)_{i \in I} \in \G_{\bw}$ and $\bx \in \bM(\bv, \bw)$, we define
$\mrm f.\bx= \bx'\equiv (x'_h, p'_i, q'_i)$ where
$x'_h = x_h$, $p'_i = p_i \mrm f_i^{-1}$ and $q'_i = \mrm f_i q_i$ for all $h\in H$ and $i\in I$.
It is clear that the $\G_{\bv}$-action and $\G_{\bw}$-action commute.

The space $\bM(\bv, \bw)$ can be endowed with a symplectic structure, given by  
\begin{align}
\label{symplectic-form}
\omega (\bx, \bx') = \sum_{h\in H} \mrm{tr} ( \ve(h) x_h x'_{\bar h}) + \sum_{i\in I} \mrm{tr} (p_i q'_i - p'_i q_i),
\quad \forall \bx, \bx' \in \bM(\bv, \bw).
\end{align}
where $\ve : H\to \{ \pm 1\}$ is a fixed orientation function such that $\ve(h)+\ve(\bar h) =0$ for all $h\in H$.
The orientation of $H$ associated to $\ve$ is $\Omega= \ve^{-1} (1)$.
Let 
\[
\mu\equiv \mu_{\mbb C} : \bM(\bv, \bw) \to \oplus_{i\in I} \mathfrak{gl}(V_i)
\]
be the moment map associated to the $\G_{\bv}$-action on the symplectic vector space $\bM(\bv, \bw)$. 
Its projection at the $i$-th component $\mathfrak{gl}(V_i)$ is given by
\[
\mu_i : \bM (\bv, \bw) \to \mathfrak{gl}(V_i), \quad \mu_i(\bx) = \sum_{h\in H: \i (h) =i} \ve( h) x_{ h} x_{\bar h} - p_i q_i.
\]

Let $\zeta_{\mbb C} =(\zeta^{(i)}_{\mbb C})_{i\in I} \in \mbb C^I$. 
We regard $\zeta_{\mbb C}$ as an element in $\oplus_{i\in I} \mathfrak{gl}(V_i)$ via 
the imbedding 
$(\zeta^{(i)}_{\mbb C})_{i\in I}  \mapsto (\zeta^{(i)}_{\mbb C} \mrm{Id}_{V_i})_{i\in I}$.
Let
\begin{align}
\label{Lambda}
\Lambda_{\zeta_{\mbb C}} (\bv, \bw) \equiv
\mu_{\mbb C}^{-1}(\zeta_{\mbb C})= 
\{ \bx \in \bM(\bv, \bw) | \mu_i (\bx) =\zeta^{(i)}_{\mbb C}, 
\quad \forall i\in I\}.
\end{align}
We shall use the notation $\Lambda_{\zeta_{\mbb C}} (V, W) $ for $\Lambda_{\zeta_{\mbb C}} (\bv, \bw) $ if we want to emphasize 
the pair $(V, W)$.
Note that  $\Lambda_{\zeta_{\mbb C}}(\bv, \bw)$ is an affine algebraic variety. Note also that $\mu_i(g.\bx) =g_i \mu_i(\bx) g_i^{-1}=\zeta^{(i)}_{\mbb C}$ for all $g \in \G_{\bv}$ and $\bx \in \Lambda_{\zeta_{\mbb C}}(\bv, \bw)$. So the $\G_{\bv}$-action on $\bM(\bv, \bw)$ restricts to a $\G_{\bv}$-action on 
$\Lambda_{\zeta_{\mbb C}} (\bv, \bw)$. 
Similarly, for all $\mrm f\in \G_{\bw}$ and $\bx \in \bM(\bv, \bw)$, we have $\mu_i(\mrm f.\bx) = \mu_i(\bx)$. 
Hence we have a $\G_{\bw}$-action on $\Lambda_{\zeta_{\mbb C}} (\bv, \bw)$.

\subsection{Quiver varieties $\M_{\zeta}(\bv, \bw)$ and $\M_{0}(\bv, \bw)$}
\label{mvw}

Let $\xi =(\xi_i)_{i\in I} \in \mbb Z^I$.
We define a character $ \chi \equiv \chi_{\xi} : \G_{\bv} \to \mbb C^*$ by 
\[
\chi(g) \equiv \chi_{\xi} (g) = \prod_{i\in I} \det (g_i)^{- \xi_i}, \quad \forall g\in \G_{\bv}.
\]
Let $\mbb C [\Lambda_{\zeta_{\mbb C}} (\bv, \bw)]^{\G_{\bv}, \chi^n}$ 
be the space of regular functions $f$ on $\Lambda_{\zeta_{\mbb C}}(\bv, \bw)$ such that
$f(g.\bx) = \chi^n(g) f(\bx)$ for all $g\in \G_{\bv}$ and $\bx \in \Lambda_{\zeta_{\mbb C}} (\bv, \bw)$.
Then the sum 
$$
R_{\zeta}(\bv, \bw) = \oplus_{n\in \mbb N} \mbb C [\Lambda_{\zeta_{\mbb C}} (\bv, \bw)]^{\G_{\bv}, \chi^n},
\quad \zeta\equiv (\xi, \zeta_{\mbb C}),
$$
becomes an $\mbb N$-graded commutative algebra with
a subalgebra
$$
R_{0}(\bv, \bw) =\mbb C [\Lambda_{\zeta_{\mbb C}} (\bv, \bw)]^{\G_{\bv}, \chi^0}.
$$
Following Nakajima ~\cite{N94, N98}, we define the quiver varieties:
for any $\zeta\equiv (\xi, \zeta_{\mbb C}) \in \mbb Z^I \times \mbb C^I$,
\begin{align}
\label{M-M_0}
\M_{\zeta} (\bv, \bw) = \mrm{Proj} \ R_{\zeta} (\bv, \bw),\quad 
\M_0 (\bv, \bw) = \mrm{Spec} \ R_{0} (\bv, \bw).
\end{align}
The inclusion  $R_0(\bv, \bw) \to R_{\zeta}(\bv, \bw)$ of the two rings involved induces a projective morphism of algebraic varieties:
\begin{align}
\label{pi}
\pi: \M_{\zeta} (\bv, \bw) \to \M_0 (\bv, \bw).
\end{align}

The $\G_{\bw}$-action on $\Lambda_{\zeta_{\mbb C}}(\bv, \bw)$, which commutes with the $\G_{\bv}$-action, 
induces $\G_{\bw}$-actions on $\M_{\zeta}(\bv, \bw)$ and $\M_{0}(\bv, \bw)$. It is clear that the proper map $\pi$ is $\G_{\bw}$-equivariant.

\subsection{Stability condition}

Fix an element $x =(x_h)_{h\in H}$  in the first component of  $\bM(\bv, \bw)$ and an $I$-graded subspace $S=(S_i)_{i\in }$ of $V$, 
we say that $S$ is $x$-invariant if $x_h( S_{\o (h)} ) \subseteq S_{\i(h)}$ for all $h\in H$.
The standard dot product on $\mbb Z^I$ is given by $a\cdot b = \sum_{i\in I} a_i b_i$ for all $a, b \in \mbb Z^I$.
Following Nakajima, a point $\bx =(x, p, q)$ in $\bM(\bv, \bw)$ is called $\xi$-$semistable$  
if the following two  stability conditions are satisfied.
Assume  $S$ and $T$ are $I$-graded subspaces of $V$  of dimension vector $\mbf s$ and $\mbf t$, respectively, 
then the stability conditions say that 
\begin{align}
& \text{
If $S$ is $x$-invariant and $S\subseteq \ker q$, then $\xi \cdot \mbf s \leq 0$.
}
\tag{S1}
\\
& \text{
If $T$ is $x$-invariant and $T \supseteq \mrm{im} \ p$, then
$\xi \cdot \mbf t \leq \xi \cdot \bv$.
}
\tag{S2}
\end{align}

Let $\Lambda^{\xi\text{-}ss}_{\zeta_{\mbb C}} (\bv, \bw)$ be the set of all $\xi$-semistable points in 
$\Lambda_{\zeta_{\mbb C}}(\bv, \bw)$. 
We see that $\Lambda^{\xi\text{-}ss}_{\zeta_{\mbb C}}(\bv, \bw)$  is $\G_{\bv}$-invariant.
For convenience, let $[\mbf x]$ denote the $\G_{\bv}$-orbit of $\mbf x$ in $\bM(\bv, \bw)$.
From Mumford's geometric invariant theory, we have

\begin{prop}[{\cite[3.ii]{N96}}]
The geometric points in $\M_{\zeta}(\bv, \bw)$ are $\Lambda^{\xi\text{-}ss}_{\zeta_{\mbb C}}(\bv, \bw)/\sim$, 
where the GIT equivalence relation $\sim$ 
is defined as $\bf x \sim \bf y $ if and only if 
$\overline{[\bx]} \cap \overline{[\bf y]} \cap \Lambda^{\xi\text{-}ss}_{\zeta_{\mbb C}}(\bv, \bw) \neq \O$ 
where the overline denotes the Zariski closure of the underlying orbit in $\Lambda_{\zeta_{\mbb C}}(\bv, \bw)$.
\end{prop}

Recall that  $\mbf C$ is the Cartan matrix of the graph $\Gamma$. We set
\begin{align*}
\begin{split}
R_+ & =\{ \gamma \in \mbb N^I | \ ^t\gamma\mbf C \gamma \leq 2\} - \{ 0\},\\
R_+(\bv) & =\{\gamma \in R_+ | \gamma_i \leq \bv_i, \forall i\in I\},\\
D_{\gamma} & =\{ a\in \mbb C^I | a\cdot \gamma=0\}.
\end{split}
\end{align*}
So the set $R_+$ consists of  positive roots of  $\mbf C$ and the set $D_{\gamma} $ is the wall defined by $\gamma$.
Note that there is
$^t\gamma \mbf C \gamma = 2 \sum_{i\in I} \gamma_i^2  -\sum_{h\in H} \gamma_{\o(h)}\gamma_{\i (h)}$.

\begin{Def}
A parameter $\zeta=(\xi, \zeta_{\mbb C})\in \mbb Z^I \times \mbb C^I$ is called $generic$ 
if it satisfies 
\begin{align}
\label{generic}
\xi \in \mbb Z^I \backslash  \cup_{\gamma\in R_+(\bv)} D_{\gamma}
\quad \mbox{or} \quad
\zeta_{\mbb C} \in \mbb C^I \backslash  \cup_{\gamma\in R_+(\bv)} D_{\gamma}.
\end{align}
\end{Def}

\begin{prop}[{\cite[Theorem 2.8]{N94}}]
Assume that the parameter $\zeta$ is generic.  
Then the group $\G_{\bv}$ acts freely on $\Lambda^{\xi\text{-}ss}_{\zeta_{\mbb C}}(\bv, \bw)$ 
and $\M_{\zeta}(\bv, \bw) = \Lambda^{\xi\text{-}ss}_{\zeta_{\mbb C}} (\bv, \bw)/\G_{\bv}$, 
the GIT quotient of $\Lambda^{\xi\text{-}ss}_{\zeta_{\mbb C}}(\bv, \bw)$ by $\G_{\bv}$.
Moreover, $\M_{\zeta}(\bv, \bw)$ is smooth.
\end{prop}

Hence the geometric point of the quiver variety $\M_{\zeta}(\bv, \bw)$ under (\ref{generic})
are parametrized by the $\G_{\bv}$-orbits in $\Lambda^{\xi\text{-}ss}_{\zeta_{\mbb C}}(\bv, \bw)$.
We set
\begin{align}
\M_{\zeta}(\bw) = \coprod_{\bv} \M_{\zeta}(\bv, \bw).
\end{align}

\begin{ass}
The parameter $\zeta$ is assumed to be  either generic or zero, unless otherwise stated.
\end{ass}

\begin{rem}
Our $\xi$ is corresponding to the parameter $\zeta_{\mbb R}$ where $\zeta_{\mbb R} = \sqrt{-1} \xi$ in ~\cite{N03}.
\end{rem}

\section{Isomorphisms on Nakajima varieties}

In this section, we introduce three classes of isomorphisms on Nakajima varieties. 
The fixed-point loci of their compositions, when become automorphisms,  will be studied in the next section.

\subsection{Reflection functors}
\label{Weyl-action}

To each element $\w\in \mathcal W$, Nakajima ~\cite{N94, N03},  
Lusztig ~\cite{L00} 
and Maffei ~\cite{M02} define the so-called $re\! f \! lection$ $f\! unctor$ 
\begin{align}
\label{S_w}
S_{\w}: \M_{\zeta} (\bv, \bw) \to \M_{\w (\zeta)} (\w * \bv, \bw), \quad \forall \zeta \ \mbox{subject to} \ (\ref{generic}),
\end{align}
which is an isomorphism of varieties 
such that $S_{\w'} S_{\w} = S_{\w' \w}$. 
When $\w$ is a simple reflection, the definition is very much like  Bernstein, Gelfand and Ponomarev's reflection functor~\cite{BGP}, from which it is named after.

Retain the pair of vector spaces $(V, W)$ of dimension vector $(\bv, \bw)$. Fix $i\in I$ and
set $$U_i = W_i \oplus \oplus_{h\in H : \o (h) =i} V_{\i (h)}.$$
Let $V'$ be a third vector space of dimension $\bv' = s_i * \bv$ 
such that  $V_j' = V_j$ if $j\neq i$. In particular, $\dim V'_i + \dim V_i   =  \dim U_i$.
To a point $\bx \in \bM(\bv, \bw)$, we set $a_i (\bx) = (q_i, x_h)_{h: \o (h) = i} $ 
and $b_i (\bx) = (p_i, \ve(\bar h) x_h)_{h: \i (h) = i}$.
Let $F$ be the pair of points $(\bx, \bx') \in \bM(\bv, \bw) \times \bM(\bv', \bw)$ 
such that the following conditions (\ref{reflection-a})-(\ref{reflection-d}) hold.
\begin{align}
& 
\begin{CD}
0 @>>> V_i'  @>a_i(\bx')>>  U_i @> b_i(\bx) >>  V_i  @>>>  0 \quad \mbox{is exact}, 
\end{CD}
\label{reflection-a} \tag{R1}
\\  
& a_i(\bx) b_i(\bx) - a_i(\bx') b_i(\bx') = \zeta'^{(i)}_{\mbb C},  \quad \zeta'_{\mbb C} = s_i (\zeta_{\mbb C}),  
\label{reflection-b} \tag{R2}
\\
& x_{h} = x'_{h}, p_j = p'_j, q_j=q'_j, \quad  \quad \mbox{if} \ \o (h) \neq i, \i (h) \neq i, \ \mbox{and} \ j\neq i, \label{reflection-c} \tag{R3} \\
& \mu_j (\bx) = \zeta^{(j)}_{\mbb C},
% \lambda_j,
\mu_j (\bx') = \zeta'^{(j)}_{\mbb C},
% \lambda_j', 
\quad \ \mbox{if}  \ j\neq i.  \label{reflection-d} \tag{R4}
\end{align}
The $\G_{\bv}\times \G_{\bv'}$-action on $\bM(\bv, \bw) \times \bM(\bv', \bw)$
induces a $\G_{\bv \cup \bv'}=\G_{\bv}\times \mrm{GL}(V_i')$-action on $F$.  

Assume that the parameter $\xi$ satisfies $\xi_i<0$ or $\zeta^{(i)}_{\mbb C} \neq 0$. We have the following diagram.
\begin{align}
\label{reflection-raw}
\Lambda^{\xi\text{-}ss}_{\zeta_{\mbb C}} (V, W) \overset{\pi_1}{\longleftarrow} 
F^{ss}(V, V', W)
\overset{\pi_2}{\longrightarrow} \Lambda^{s_i(\xi)\text{-}ss}_{s_i(\zeta_{\mbb C})} (V', W) 
\end{align}
where 
$F^{ss}(V, V', W) 
= F \cap 
\left ( \Lambda^{\xi \text{-} ss}_{\zeta_{\mbb C}} (V, W) \times \Lambda^{s_i(\xi)\text{-}ss}_{s_i(\zeta_{\mbb C})}  (V', W) \right ) $,
$\pi_1$ and $\pi_2$ are the natural projections.
It is known that $\pi_1$ and $\pi_2$ are $\mrm{GL}(V'_i)$ and $\mrm{GL}(V_i)$ principal bundles, respectively.
This induces isomorphisms of varieties:
\begin{align}
\M_{\zeta} ( \bv, \bw) \overset{\pi_1}{\longleftarrow} 
\G_{\bv \cup \bv'} \backslash F^{ss}(V, V', W)
\overset{\pi_2}{\longrightarrow}
\M_{s_i(\zeta)} (s_i * \bv, \bw).
\end{align}
The simple reflection $S_i$ on quiver varieties is defined by
\begin{align}
\label{Si}
S_i = \pi_2 \pi_1^{-1}: \M_{\zeta} (\bv, \bw) \to \M_{s_i (\zeta)} (s_i * \bv, \bw), \quad \mbox{if} \ \xi_i < 0 \ \mbox{or}\ \zeta^{(i)}_{\mbb C} \neq 0.
\end{align}

Since $(s_i (\xi))_i >0$ if $\xi_i < 0$, we can define the reflection $S_i$ when $\xi_i >0$, by switching the roles of $\bx$ and $\bx'$.
So if $\w=s_{i_1}s_{i_2}\cdots s_{i_l} \in \mathcal W$ and $\zeta$ satisfies the condition (\ref{generic}), 
the reflection functor $S_{\w}$ in (\ref{S_w}) is defined to be
\begin{align*}
%\label{sw}
S_{\w} = S_{i_1} S_{i_2} \cdots S_{i_l} : \M_{\zeta}(\bv, \bw) \to \M_{\w(\zeta)} (\w * \bv, \bw).
\end{align*}

When $\zeta=0$,  the reflection functor $S_{\w}:  \M_0(\bv, \bw) \to \M_0(\w*\bv,\bw)$ is defined 
to be the identity morphism when $\w*\bv=\bv$, following~\cite[2.1]{L00}.

If we let $\G_{\bw}$ act diagonally on $\bM(\bv, \bw) \times \bM(\bv', \bw)$ in the above construction, we see that 
the simple reflections $S_i$ and hence the general Weyl group action $S_{\w}$ are $\G_{\bw}$-equivariant.

\subsection{The isomorphism $\tau$}
\label{involution}

A finite dimensional vector space $E$ equipped with a non-degenerate bilinear form $(-, -)_E$ is called  a $f\! ormed$ $space$.
To any linear transformation $T: E\to E'$ between two formed spaces, 
we define  its right adjoint $T^*: E' \to E$ by the rule 
\[
(T(e), e')_{E'} = (e, T^* (e'))_{E}, \quad \forall e \in E, e' \in E'. 
\]
It is clear that the map $T\mapsto T^*$ defines an isomorphism $\Hom(E, E') \cong \Hom (E', E)$ of vector spaces.
If further $E''$ is a formed space
and $T': E' \to E''$ is a linear transformation, then
$(T' T)^* = T^* T'^*$. 

Similarly, we can define the left adjoint $T^!$ of $T$ by 
$(e', T(e))_{E'} = (T^!(e'), e)_E$ for all $e\in E$ and $e'\in E'$.
We have $(T^*)^!= T$ and $(T^!)^*=T$.

Let $\delta$ be either $+1$ or $-1$. 
A formed space  $E$ is called 
a $\delta$-$f\! ormed$ space  if the associated form $(-, -)_E$ on $E$ satisfies 
that  $(e_1, e_2)_E = \delta (e_2, e_1)_E$ for all $e_1, e_2 \in E$. 
When $\delta=1$, we have a symmetric form, while when $\delta=-1$, we have a symplectic form.
In this case, the form $(-, -)_E$ is called  a $\delta$-$f\! orm$.
If  $E'$ is  a $\delta'$-formed space for some $\delta' \in \{ \pm 1\}$, then we have $(T^*)^* = \delta \delta' T$. 

Assume the vector space $E$ is a formed space and admits an $I$-grading $E = \oplus_{i\in I} E_i$, 
we call $E$ an $I$-$graded$ $f\! ormed$ $space$
if the restriction $(-, -)_{E_i}$ of the form $(-, -)_E$ to each subspace $E_i$ is a non-degenerate form and $(E_i, E_j)_E=0$ if $i\neq j$.
Let $E$ be an $I$-graded formed space  
and  fix a function $\underline \delta =(\delta_{ i})_{i\in I} \in \{\pm 1\}^I$. We call $E$
a $\underline \delta$-$f\! ormed$ $space$, or a $f\! ormed$ $space$ $with$ $sign$ $\underline \delta$,
if the restriction $(-,-)_{E_i}$ is a $\delta_{ i}$-form for all $i\in I$. 
We call $\underline \delta$ the $sign$ of $E$.

Recall the pair $(V, W)$ of vector spaces of dimension vector $(\bv, \bw)$ and $\bM(\bv, \bw)$ from Section ~\ref{lavw}. 
Assume that $V$ and $W$ are two $I$-graded formed spaces,
we define an automorphism
\begin{align}
\label{tau-raw}
\tau : \bM(\bv, \bw) \to \bM(\bv, \bw), \quad \bx=(x_h, p_i, q_i) \mapsto {}^{\tau} \bx =({}^{\tau} x_h, {}^{\tau} p_i, {}^{\tau}q_i)
\end{align}
where  ${}^\tau x_h = \ve (h)  x_{\bar h}^*$, ${}^{\tau} p_i = -  q_i^{*}$ and ${}^{\tau} q_i =  p_i^*$ 
for all $h\in H$ and $i\in I$.
Its inverse is defined by taking the left adjoints, that is $\tau^{-1} (\bx) = ({}^{\tau^{-1}} x_h, {}^{\tau^{-1}} p_i, {}^{\tau^{-1}} q_i)$
where ${}^{\tau^{-1}} x_h = \ve(\bar h) x^!_{\bar h}$, ${}^{\tau^{-1}} p_i = q_i^!$ and ${}^{\tau^{-1}} q_i = - p_i^!$.

By the properties of taking adjoints, we have  $\mu_i (^{\tau}\bx) = - \mu_i (\bx)^*$. 
So the automorphism on $\bM(\bv, \bw)$ restricts to an isomorphism still denoted by $\tau$
\begin{align}
\label{Lambda-tau-raw}
\tau: \Lambda_{\zeta_{\mbb C}} (\bv, \bw) \to \Lambda_{-\zeta_{\mbb C}}(\bv, \bw).
\end{align}
Further, for any regular function $f$ in $\mbb C[\Lambda_{\zeta_{\mbb C}} (\bv, \bw)]^{\G_{\bv}, \chi_{\xi}^n}$, we have
\begin{align*}
f  \tau (g.\bx) & = f ({}^{\tau}g. {}^{\tau} \bx)   && {}^{\tau}g_i = (g_i^{-1})^* \\
& = \chi_{\xi}^{n}({}^{\tau}g) f({}^{\tau} \bx)  && \empty \\
& = \chi_{-\xi}^n (g) f \tau (\bx),  && \forall g\in \G_{\bv}, \bx \in \Lambda_{-\zeta_{\mbb C}} (\bv, \bw).
\end{align*}
So $f \tau \in \mbb C[\Lambda_{-\zeta_{\mbb C}} (\bv, \bw)]^{\G_{\bv}, \chi_{-\xi}^n}$.
This implies that the assignment $f \mapsto f \tau$ defines an isomorphism  of graded associative algebras:
$R_{\zeta} (\bv, \bw) \to R_{- \zeta} (\bv, \bw)$, where $R_{\zeta} (\bv, \bw) $ is from Section ~\ref{mvw}.

The above isomorphism shows that the isomorphism on $\Lambda_{\zeta_{\mbb C}} (\bv, \bw)$ 
restricts to an isomorphism
$\Lambda^{\xi\text{-}ss}_{\zeta_{\mbb C}} (\bv, \bw) \to \Lambda^{(-\xi)\text{-}ss}_{-\zeta_{\mbb C}} (\bv, \bw)$.
Due to ${}^{\tau}(g.\bx) = {}^{\tau} g. {}^{\tau}\bx$, it further induces the isomorphism:
recall $[\bx]$ denotes the $\G_{\bv}$-orbit of $\bx$,
\begin{align}
\label{tau}
\begin{split}
&\tau_{\zeta} : \M_{\zeta} (\bv, \bw) \to \M_{-\zeta}(\bv, \bw), \quad [\bx] \mapsto [\ \! \! ^{\tau} \! \bx],
\end{split}
\end{align}
such that the following diagram commutes.
\begin{align}
\label{tau-comm}
\begin{CD}
\M_{\zeta} (\bv, \bw) @>\tau_{\zeta}>> \M_{-\zeta} (\bv, \bw) \\
@V\pi VV @V\pi VV \\
\M_0 (\bv, \bw) @>\tau_0 >> \M_{0} (\bv, \bw)
\end{CD}
\end{align}
Now we show that the isomorphism $\tau_{\zeta}$ depends only on the forms on $W$.

\begin{prop}
\label{tau-form}
The $\tau_{\zeta}$ in (\ref{tau}) is independent of the choices of forms on $V$.
\end{prop}

\begin{proof}
If we fix a basis for each vector space $V_i$ and $W_i$, 
then to give a form on $V_i$ or $W_i$ is the same as to give a certain invertible matrix, say $M_i$ or $N_i$.
In this way, the right adjoints are presented as $x_h^* = M^{-1}_{\o (h)} {}^t x_h M_{\i (h)}$ 
and $p_i^* = N_i^{-1} {}^t p_i M_i$ and $q_i^*=M_i^{-1} {}^t q_i N_i$ for all $h\in H$ and $i\in I$.
If we attach to each $V_i$ a new form with associated matrix $\tilde M_i$, 
we can have a new automorphism, say $\tilde \tau$, on $\bM(\bv, \bw)$, 
and  a new point ${}^{\tilde \tau}\bx$ for each $\bx \in \bM(\bv, \bw)$.
Set $g =(g_i)_{i\in I}\in \G_{\bv} $ with $g_i = \tilde M_i^{-1} M_i$, $\forall i\in I$. 
Then  the proposition follows from  $g. {}^{\tau} \bx = {}^{\tilde \tau} \bx$.
\end{proof}

For $\mrm f \in \G_{\bw}$, we set ${}^{\tau} \mrm f = (\mrm f^{-1})^*$. 
Then $\tau(\mrm f. [\bx]) =  {}^{\tau} \mrm f. \tau ([\bx])$ for all $\mrm f \in \G_{\bw}$ and $[\bx] \in \M_{\zeta}(\bv, \bw)$.

\begin{prop}
\label{order-4}
If  $W$ is a formed space with sign $\delta_{\bw}$,   then the isomorphism $\tau_{\zeta}$ on
$\M_{\zeta}(\bv, \bw)$  satisfies $\tau^{4}_{\zeta}=1$.
Moreover, if the $\delta_{\bw}$ is $\Gamma$-alternating, i.e., $\delta_{\bw, \o (h)} \delta_{\bw, \i (h)} =-1$ for all $h\in H$, then
$\tau_{\zeta}^2=1$.
\end{prop}

\begin{proof}
From the property of taking adjoints twice with respect to $\delta$-forms, it is straightforward to see the first statement in the proposition.
By Proposition~\ref{tau-form}, we can attach to each $V_i$ a symmetric form. Then for $\bx=(x_h, p_i, q_i)$, we have
$\tau^2(\bx) =(-x_h, -\delta_{\bw, i} p_i, - \delta_{\bw, i}q_i)$. Let $g=(-\delta_{\bw, i} \mrm{id}_{V_i})_{i\in I}$, then we have
$g. \tau^2(\bx) = \bx$. This implies that $\tau_{\zeta}^2=1$. 
\end{proof}

We now show that the isomorphism $\tau_{\zeta}$ commutes with the reflection functors.
Recall the setting from Section ~\ref{Weyl-action}. 
We fix a vertex $i\in I$ and a triple $(V, V', W)$ of $I$-graded vector spaces of dimension vector $(\bv, \bv', \bw)$ 
such that $V_j=V'_j$ for all $j\neq i$ and $\bv' =s_i *\bv$.
We assume that all spaces in this triple are $I$-graded formed spaces.
For simplicity, let  $B_V=(-, -)_V, B_W=(-, -)_W$ stand for the bilinear forms on $V$ and $W$, respectively.
Similar to the isomorphism $\tau_{\zeta} \equiv \tau_{\zeta}(B_V, B_W)$ as above, 
we have an isomorphism 
\[
\tau_{s_i(\zeta)}  \equiv \tau_{s_i(\zeta)} (B_{V'}, B_{W}) : \M_{s_i (\zeta)} (s_i * \bv, \bw) \to \M_{- s_i(\zeta)} (s_i * \bv, \bw).
\]

\begin{lem}
\label{comm-Si-tau}
We have $S_i \tau_{\zeta} (B_V, B_W) = \tau_{s_i(\zeta)} ( B_{V'}, B_W) S_i$ 
where $S_i$ is the reflection functor defined in Section \ref{Weyl-action}.
\end{lem}

\begin{proof}
By Proposition \ref{tau-form}, we can assume that  the forms on $V_j$ and $V'_j$ coincide    for all $j\neq i$.
We observe that $a_i ({}^{\tau} \bx) =  b_i(\bx)^*$ and $b_i ({}^{\tau} \bx) = -  a_i(\bx)^*$.
So the short exact sequence in (\ref{reflection-a}) gives rise to the following short exact sequence.
\[
\begin{CD}
0 @>>> V_i @>a_i ({}^{\tau} \bx)>>  U_i @> b_i({}^{\tau} \bx')>> V_i' @>>> 0.
\end{CD}
\]
Similarly, the equation in (\ref{reflection-b}) yields the equality
$a_i({}^{\tau}\bx') b_i({}^{\tau}\bx') 
-a_i({}^{\tau}\bx) b_i({}^{\tau}\bx)   = \zeta'^{(i)}_{\mbb C}$.
As a consequence, we have the commutative relation in the lemma.
\end{proof}

\subsection{The diagram isomorphism $a$}

\label{dia}

Let $a$ be an automorphism of $\Gamma$, that is, there are permutations of vertex and edge sets, both denoted by $a$, such that
$a(\o(h))= \o (a(h))$, $a(\i (h)) = \i (a(h))$ and $a(\bar h) = \overline{a(h)}$ for all $h\in H$. 
We further assume that $a$ is compatible with the function $\ve$ in the definition of the moment map $\mu$ in Section ~\ref{lavw}: 
there exists a constant $c\equiv c_{a, \ve} \in \{ \pm 1\}$ such that 
\begin{align}
\label{compatible-pair}
\ve(a(h)) = c \cdot \ve(h), \  \forall h\in H .
\end{align}
The automorphism $a$ on $\Gamma$ induces operations on $I$-graded vector spaces and vectors. 
If $V$ is an $I$-graded space, we denote $a(V)$ the $I$-graded vector space whose $i$-th component is $V_{a^{-1}(i)}$.
Similarly $a(\bv)$ is a vector whose $i$-entry is the $a^{-1}(i)$-th entry of $\bv$. 
Given any point $\bx=(x, p, q) \in \bM(\bv,\bw)\equiv \bM(V, W)$, 
we define a point $$a(\bx)=(a(x), a(p), a(q) )\in \bM(a(\bv), a(\bw)) \equiv\bM(a(V), a(W))$$ 
%where $a(V)_i=V_{a^{-1} (i)}$, $a(W)_i=W_{a^{-1}(i)}$,  $a(\bv)_i= \bv_{a^{-1}(i)}$ and $a(\bw)_i= \bw_{a^{-1}(i)}$, 
by
\[
a(p)_i=p_{a^{-1}(i)}, \ a(q)_i = q_{a^{-1}(i)}, \ 
a(x)_h= \ve(h)^{\frac{1-c}{2}}  x_{a^{-1}(h)}, \quad \forall i\in I, h\in H.
\]
By definition, $\mu_i (a(\bx)) = \mu_{a^{-1} (i)} (\bx)$. Thus it induces a diagram isomorphism of finite order on Nakajima's varieties:
\begin{align}
\label{dia-a}
a: \M_{\zeta} (\bv, \bw) \to \M_{a(\zeta)} (a(\bv), a(\bw)).
\end{align}
The order of this isomorphism is the same as that on the diagram.

The isomorphism $a$ is a variant of diagram automorphisms studied in ~\cite{HL14}. 
Just like {\it loc. cit.}, it can be generalized as follows.
Let us fix $(\mrm f^0 , g^0)\in \G_{\bw} \times\G_{\bv}$, we can define an isomorphism 
$ a_{\mrm f^0,g^0}
: \M_{\zeta} (\bv, \bw) \to \M_{a(\zeta)} (a(\bv), a(\bw))
$ 
to be the composition of $a$ with the action of $(\mrm f^0, g^0)$.
Specifically,    for any $[\bx]\in \M_{\zeta}(\bv,\bw)$, the element
$a_{\mrm f^0, g^0}([\bx])$, 
is represented by $a_{\mrm f^0, g^0}(\bx) = (a_{\mrm f^0,g^0}(x), a_{\mrm f^0, g^0}(p), a_{\mrm f^0,g^0}(q))$ where  for all $i\in I$, $h\in H$
\begin{align*}
a_{\mrm f^0,g^0}(p)_i = g^0_{a^{-1}(i)} p_{a^{-1} (i)} (\mrm f_{a^{-1}(i)}^0)^{-1},\\
a_{\mrm f^0,g^0}(q)_i = \mrm f^0_{a^{-1}(i)} q_{a^{-1}(i)} (g^{0}_{a^{-1}(i)})^{-1},\\
a_{\mrm f^0,g^0}(x)_h= \ve(h)^{\frac{1-c}{2}}  g^0_{\i( a^{-1}(h))} x_{a^{-1}(h)} (g^0_{\o (a^{-1}(h))})^{-1}.
\end{align*}
Similar to Proposition~\ref{tau-form}, the isomorphism 
$a_{\mrm f^0,  g^0}$ is independent of the choice of $g^0$. Hence it makes sense
to denote this isomorphism by $a_{\mrm f^0}$ and the  $a$ in (\ref{dia-a}) is $a_1$. 

There is a permutation, $\mrm f \mapsto {}^a \mrm f$, on $\G_{\bw}$ 
given by $(^a \mrm f)_i = \mrm f_{a^{-1}(i)}$ for all $i\in I$. 
It is clear that 
\begin{align}
\label{a-g}
a_{\mrm f^0}(\mrm f. [\bx]) = {}^a\mrm f^0. a_{\mrm f} ([\bx]), \quad \forall [\bx]\in \M_{\zeta}(\bv, \bw), \mrm f\in \G_{\bw}. 
\end{align}
It is also clear that the isomorphism $a$ is compatible with the reflection functor $S_i$:
\begin{align}
\label{a-S}
a_{\mrm f^0} \circ S_i = S_{a(i)} \circ a_{\mrm f^0}. 
\end{align}
Subsequently, $S_{w_0}\circ a_{\mrm f^0} = a_{\mrm f^0} \circ S_{w_0}$ 
when $\Gamma$ is Dynkin and $w_0$ is the longest Weyl group element since $a(w_0)=w_0$. 
The two isomorphisms $\tau_{\zeta}$ and $a_{\mrm f^0}$  are compatible as well. 
Precisely, 
\begin{align}
\label{a-T}
\tau_{a\zeta}(a(B_V), a(B_W)) a_{\mrm f^0} = a_{{}^{\tau}{\mrm f^0}} \tau_{\zeta} (B_V, B_W).
\end{align}

Finally, we remark that the composition $a \tau_{\zeta}$ is similar to automorphisms in~\cite{E09}
and a special case of $\tau_{\zeta}$ appeared in~\cite{KP82}, see further Section~\ref{fixed-cat}.

\section{Geometric properties of $\sigma$-quiver varieties}
\label{sigma-quiver}

In this section, we study the fixed-point subvarieties, called $\sigma$-quiver varieties, of the compositions of the three classes of 
isomorphisms of Nakajima varieties introduced in the previous section. 

\subsection{The $\sigma$-quiver varieties, I: $\zeta$ generic}
In this subsection, we assume that $\zeta$ is generic. 
We consider the following isomorphism on quiver varieties. 
\begin{align}
\label{sigma}
\sigma: =a  S_{\w} \tau_{\zeta} : \M_{\zeta}(\bv, \bw) \to \M_{- a \w(\zeta)} (a(\w * \bv), a(\bw)),
\end{align}
where $\tau_{\zeta}$, $S_{\w}$ and $a$ are defined in (\ref{tau}), (\ref{S_w}) and (\ref{dia-a}), respectively.
We shall write $\sigma_{\zeta, \w, a}$ for $\sigma$ if we want to emphasize that $\sigma$ depends on $\zeta$, $\w$ and $a$.
By the commutativity of the three isomorphisms from Lemma~\ref{comm-Si-tau}, (\ref{a-S}) and (\ref{a-T}),  we have

\begin{prop}
\label{sigma-finite}
If the forms involved are $\delta$-forms and $\w$ is of finite order, then
the order of
$\sigma$ is finite and a divisor of the least common multiple $\mbox{l.c.m.} \{ 4, |\w|, |a| \}$.
\end{prop}

By summing over all $\bv$, we have an isomorphism.
\[
\sigma: \M_{\zeta}(\bw) \to \M_{- a \w(\zeta)} (a(\bw)).
\]
If $- a\w (\zeta) = \zeta$ and $a(\bw)=\bw$, 
then $\sigma$ becomes an automorphism on $\M_{\zeta}(\bw)$. We set
$$
\fS_{\zeta}(\bw) \equiv \M_{\zeta}(\bw)^{\sigma}
$$ 
to be its fixed point subvariety. 
If further $a(\w*\bv) = \bv$, let 
$$
\fS_{\zeta}(\bv,\bw) \equiv \M_{\zeta} (\bv, \bw)^{\sigma}
$$ 
be the fixed point subvariety of $\M_{\zeta}(\bv, \bw)$ 
under the automorphism $\sigma$.
Then we have
\begin{align}
\label{M(w)}
\fS_{\zeta}(\bw) = \sqcup_{a (\w * \bv)=\bv} \fS_{\zeta} (\bv, \bw), \quad \mbox{if} \ - a \w (\zeta) = \zeta, a(\bw) =\bw.
\end{align}

\begin{Def}
The varieties $\fS_{\zeta}(\bv, \bw)$ and $\fS_{\zeta}(\bw)$ are called the $\sigma$-$quiver$ $varieties$.
\end{Def}

%Here the greek letter $\fS$ stands for  `sigma', `subvariety',   `symplectic' or `symmetric'. 
Before we proceed, we make a remark.
\begin{rem}
%The definition of $\sigma$-quiver varieties depends on the choice of isomorphisms $W_i \cong W_{a(i)}$ relevant to the automorphism $a$, 
%which is taken to be the identity $W_i = W_{a(i)}$. 
A more general isomorphism $\sigma_{\mrm f^0}$ can be defined by using  $a_{\mrm f^0}$, a generalization of $a$,  in Section~\ref{dia}. To control its order in this case, $\mrm f^0$ 
has to  satisfy a compatibility assumption in~\cite{HL14}. Specifically, 
we can identify $W_i$ with $W_{a(i)}$ for all $i\in I$ due to  $\bw_i=\bw_{a(i)}$.  
For each $i\in I$, let $m_i=\#\{ a^{n}(i)| n\in \mbb Z\}$. Fix $m$ such that $m_i|m$, $\forall i\in I$. The compatibility condition for $\mrm f^0$ reads $\mrm f^0_{i} \mrm f^0_{a(i)} \cdots \mrm f^0_{a^{m-1}(i)} = 1$, $\forall i\in I$. 
Then the order of $a_{\mrm f^0}$ is $m$.
For the sake of simplicity, we focus on the simpler version $\sigma$ instead of $\sigma_{\mrm f^0}$.
\end{rem}

The definition of $\sigma$-quiver varieties depends on the forms  on $V$ and $W$. 
But by Proposition ~\ref{tau-form},  it only depends on the form on $W$, which is recorded as follows.

\begin{prop}
\label{sigma-indep-of-V}
The variety $\fS_{\zeta}(\bv, \bw)$ is independent of the choice of  the form on $V$.
\end{prop}

By combining Proposition ~\ref{sigma-finite} and Proposition ~\ref{sigma-indep-of-V}, it yields

\begin{prop}
\label{sig-4}
If $W$ is a $\delta_{\bw}$-formed space and $\w$ is of finite order, then
the order  of $\sigma$ is  a divisor of  $\mbox{l.c.m.} \{ 4, |\w|, |a|\}$.
If further the sign $\delta_{\bw}$ is $\Gamma$-alternating
%, i.e., $\delta_{\o (h)}(\bw) \delta_{\i (h)}(\bw) =-1$ for all $h\in H$ 
and $a^2=\w^2=1$, then $\sigma^2=1$.
\end{prop}

The following example shows that  $\sigma$-quiver varieties include quiver varieties. 

\begin{example}
Let $\hat \Gamma$ be the product of four copies of $\Gamma$. Let $a$ be the obvious cyclic permutation of order $4$ on $\hat \Gamma$.
Then there is an automorphism $\sigma$ with $\w=1$ on 
$\M_{\hat\Gamma} = \M_{\zeta}(\bv, \bw)\times \M_{-\zeta}(\bv, \bw) \times \M_{\zeta}(\bv, \bw) \times \M_{-\zeta}(\bv, \bw)$. 
If the space $W$  is a formed space of sign  $\delta_{\bw}$, then we see that $\fS_{\hat\Gamma} \cong \M_{\zeta}(\bv, \bw)$.
In particular, if the $\delta_{\bw}$ is alternating, then we only need two copies of $\Gamma$ to realize $\M_{\zeta}(\bv, \bw)$ as
a $\sigma$-quiver variety. 
\end{example}

It is well-known, e.g., ~\cite[Proposition 1.3]{I72}, ~\cite[Proposition 3.4]{E92}, or
~\cite[Lemma 5.11.1]{CG}, that the fixed point subvariety of an action of a reductive group, in particular, a finite group,  
on a smooth variety is smooth.  
%A finite group is reductive.
If the automorphism $\sigma$ has a finite order $N$, then it is the same as a $\mbb Z_N$-action on quiver varieties. So it gives rise to

\begin{prop}
\label{smooth}
Assume that $\zeta$ is generic. 
The $\sigma$-quiver variety  $\fS_{\zeta}(\bv, \bw)$ is smooth, provided that it is nonempty and
the order of $\sigma$ is finite.
\end{prop}

The reflection functor $S_{\w}$ does not always exist on $\M_0(\bv,\bw)$ and, if exists, they are not  isomorphic in general. 
So to define $\sigma$-quiver varieties as a fixed-point locus on $\M_0(\bv,\bw)$  does not work in general. 
When the graph is Dynkin,  
the reflection functor does exist on the global/limit version $\M_0(\bw)$ of $\M_0(\bv,\bw)$, thanks to Lusztig's work~\cite{L00},   so in this case
it is possible to define $\sigma$-quiver variety in  $\M_0(\bw)$ as a fixed-point locus, which is treated in the following section.
The $\sigma$-quiver variety in $\M_0(\bv,\bw)$ is then obtained by taking intersection of $\M_0(\bv,\bw)$ with the $\sigma$-quiver variety in $\M_0(\bw)$.
Here, instead, we define  the following.
\begin{align}
\label{S1-def}
\begin{split}
\fS_1(\bv, \bw) & \equiv \pi ( \fS_{\zeta}(\bv,\bw)), \quad \mbox{if} \ \zeta = - a \w (\zeta), a(\w *\bv) = \bv, a\bw=\bw\\
\fS_1(\bw) & = \sqcup_{a (\w*\bv) =\bv}  \fS_1(\bv,\bw),  \quad \mbox{if} \ \zeta = - a \w (\zeta), a\bw=\bw.
\end{split}
\end{align}
In particular, the proper morphism $\pi$ in (\ref{pi}) restricts to a proper morphism:
\begin{align}
\label{pi-s-1}
\pi^{\sigma}: \fS_{\zeta} (\bv, \bw) \to \fS_{1} (\bv, \bw) \ \mbox{and} \
\pi^{\sigma}: \fS_{\zeta}(\bw) \to \fS_1(\bw).
\end{align}

Let $\G_{\bw}^{\sigma} = \{ \mrm f\in \G_{\bw} | \mrm f = {}^{a \tau} {\mrm f} \}$. 
Since $S_{\w}$ is $\G_{\bw}$-equivariant and $\tau_{\zeta}$  and $a$ satisfy
$\tau_{\zeta}(\mrm f. [\bx]) ={}^{\tau} \mrm f. \tau_{\zeta} ([\bx])$ and $a (\mrm f . [\bx]) = {}^a{\mrm f}.[\bx]$ 
for all $\mrm f \in \G_{\bw}$ 
and $[\bx] \in \M_{\zeta}(\bv, \bw)$, we see that 
the automorphism $\sigma$ satisfies the following property.
\begin{align}
\sigma(\mrm f.[\bx])= \ ^{a \tau}{\mrm f}. \sigma ([\bx]),  \quad \forall \mrm f\in \G_{\bw}, [\bx]\in \M_{\zeta}(\bv, \bw).
\end{align}
It induces $\G_{\bw}^{\sigma}$-actions on $\fS_{\zeta}(\bv, \bw)$ and $\fS_{1}(\bv, \bw)$, 
which is compatible with the proper map $\pi^{\sigma}$.
There is a natural $\mbb C^{\times}$-action on $\bM(\bv, \bw)$ given by 
$\bx=(x_h, p_i, q_i)_{h\in H, i\in I} \mapsto t. \bx=(tx_h, tp_i, tq_i)_{h\in H, i\in I}$ for all $t\in \mbb C^{\times}$.
This $\mbb C^{\times}$-action commutes with the isomorphisms $a$ and $\tau$ on $\bM(\bv, \bw)$, which
in turn induces a  $\mbb C^{\times}$-action on $\fS_1(\bv, \bw)$.
If the parameter $\zeta_{\mbb C}=0$, then the $\mbb C^{\times}$-action on $\bM(\bv, \bw)$ restricts to 
a $\mbb C^{\times}$-action on $\Lambda_{\zeta_{\mbb C}}(\bv, \bw)$, and then on $\M_{\zeta}(\bv, \bw)$.
This action clearly commutes with the $\G^{\sigma}_{\bw}$-actions on $\fS_{\zeta}(\bv, \bw)$ and $\fS_1(\bv, \bw)$.
In this case the morphism $\pi^{\sigma}$ is $\G^{\sigma}_{\bw}\times \mbb C^{\times}$-equivariant.
The above analysis yields
\begin{prop}
The map  $\pi^{\sigma}$ is $\G^{\sigma}_{\bw}$-equivariant. 
If $\zeta_{\mbb C}=0$, it is $\G^{\sigma}_{\bw}\times \mbb C^{\times}$-equivariant.
\end{prop}

\subsection{The $\sigma$-quiver varieties, II: $\zeta=(0, \z)$}

In this section, we assume that $\Gamma$ is Dynkin and  $\zeta=(0,\z)$, which is not necessarily generic.
We give a  definition of $\sigma$-quiver varieties under these assumptions, by making use of Lusztig's reflection functor and
 global versions of  the transpose $\tau$ and the diagram isomorphism $a$ defined as follows.

\subsubsection{Lusztig's variety $Z^{\z}_{\bw}$}

Let $\F$ be the space of paths in the Dynkin graph $\Gamma$.
The concatenation operation, $(\rho, \rho') \mapsto \rho\cdot \rho' = \delta_{\o (\rho), \i(\rho')} \rho \rho'$, 
of the paths defines an associative algebra structure on $\F$.
The bar involution on $\Gamma$ defines an anti-involution on $\F$, which we shall denote by the same notation.
Let $\i(f)$ and $\o(f)$ be the ending and starting vertex of the path $f$.
Let $[i]$ be the path of length zero such that $\i ([i]) =\o([i]) =i$.
For $\z \in \mbb C^I$, let 
$$
\theta_{i,\z} = \sum_{h: \i (h) =i} \ve (h) h \bar h - \z^{(i)} [i],  \ \forall i\in I.
$$ 
Recall that $W$ is an  $I$-graded vector space of dimension $\bw$.
Let $Z^{\z}_{\bw}$ be the set of linear maps $\pi'$ from $ \F$ to $\End (W)$ such that 
\begin{itemize}
\item $\pi'(f) \in \Hom (W_{\o(f)}, W_{\i(f)}) \subseteq \End(W)$ for all path $f$,
\item $\pi'(f)\pi'(f') = \pi'( f  \cdot \theta_{i, \z} \cdot f')$ for all paths $f$ and $f'$ such that 
$\i(f') =i=\o(f)$.
\end{itemize}
$Z^{\z}_\bw$ is an affine algebraic variety by~\cite{L00} and  isomorphic to $ \M_{(0,\z)}(\bv,\bw)$ for $\bv$ very large.
 As a set, $Z^{\z}_\bw$ can be identified with $\M_0(\bw)$ under a proper treatment.

 Following~\cite[2.3]{L00}, there is a $\mbb C^{\times}$-action on the totality $\sqcup_{\z\in \mbb C^I} Z^{\z}_\bw$ given by
 $t: Z^{\z}_\bw\to Z^{t^2\z}_\bw$ and $(t.\pi') (f) = t^{s+2}\pi'(f)$ if $f= h_1\cdots h_s$ for all $\pi'\in Z^{\z}_\bw$.
 Following~\cite[2.4]{L00}, there is a $\G_{\bw}$  action on $Z^{\z}_\bw$ by 
 $(g.\pi') (f) = g_{\i(f)} \pi'(f) g_{\o (f)}^{-1}$ for all path $f$ and $\pi'\in Z^{\z}_\bw$.

\subsubsection{The transpose $\tau_0$}
We extend the orientation function $\ve$ to a function on the set of paths in $\Gamma$ by defining
$\ve ([i]) =1$ and $\ve (f)=\prod_{i=1}^s \ve(h_i) $ if $f=h_1\cdots h_s$.
Now assume further  that $W$ is an $I$-graded  formed space.
We define an isomorphism 
\begin{align}
\label{sig-Z}
\tau_0 : Z^{\z}_{\bw} \to Z^{-\z}_{\bw}, \quad \pi' \mapsto \tau_0(\pi'),
\end{align}
where
$\tau_0 (\pi') (f)  = - \ve(f)  \pi' (\bar f)^*$ for any path $f$ in $\Gamma$.

We must show that $\tau_0$ on $Z^{\z}_{\bw}$ is well-defined, that is, $\tau_0 (\pi') \in Z^{-\z}_{\bw}$.
Clearly $\tau_0(\pi')(f) \in \Hom(W_{\o (f)}, W_{\i (f)})$. 
For any two paths $f, f'$ in $\Gamma$ such that $\i(f') = i=\o (f)$, we have 
\begin{align}
\begin{split}
\tau_0 (\pi') (f) \tau_0 (\pi')(f')
& =  \ve(f) \ve (f')  \pi'(\bar f)^* \pi'(\overline{f'})^* \\
& =  \ve(f) \ve(f') \left (\pi'(\overline{f'}) \pi'(\bar f) \right )^*\\
& = \ve(f) \ve(f') \left  (\pi' (\overline{f'} \theta_{i,\z} \bar f ) \right )^*\\
& = - \ve(f) \ve(f') \left  (\pi' (\overline{f \theta_{i,-\z} f'} ) \right )^* = \tau_0 (\pi') (f\theta_{i,-\z} f').\\
\end{split}
\end{align}
Therefore the well-definedness of $\tau_0$ follows.

There is a morphism of varieties $\Lambda_{\z}(\bv, \bw) \to Z^{\z}_{\bw}$ sending a point $\bx = (x_h, p_i, q_i)$ to $\pi'$ such that
$\pi'(f)  = q_{\i (h_1)} x_{h_1}\cdots x_{h_{s-1}}  x_{h_s} p_{\o (h_s)} $ 
for any path $f= h_1\cdots h_s$.
The morphism then induces an immersion $\vartheta: \M_{(0,\z)} (\bv, \bw) \to Z_{\bw}$. 
It is clear that the isomorphism $\tau$ on $\Lambda_{\z}(\bv, \bw)$ in (\ref{Lambda-tau-raw}) is compatible with the isomorphism $\tau_0$ on
$Z^{\z}_{\bw}$. This implies that the isomorphisms $\tau_0$ on $\M_{(0,\z)}(\bv, \bw)$ and $Z^{\z}_{\bw}$ 
are compatible under the immersion
$\vartheta$, that is the following diagram commutes.
\begin{align}
\label{T-compatible}
\begin{CD}
\M_{(0,\z}(\bv, \bw) @>\tau_0 >> \M_{(0, -\z)}(\bv, \bw) \\
@V\vartheta VV @V\vartheta VV \\
Z^{\z}_{\bw} @>\tau_0 >> Z^{-\z}_{\bw}.
\end{CD}
\end{align}
This indicates that the notation $\tau_0$ on $Z^{\z}_{\bw}$ and $\M_{(0,\z)} (\bv, \bw)$ will not cause any confusion.

\begin{rem}
\label{Z-sig}
We have $(Z^{0}_\bw)^{\tau_0} = \{ \pi' \in Z^0_{\bw} | \pi' (f) = - \ve(f) \pi' (\bar f)^* \ \mbox{for any path $f$ in $\Gamma$} \}$. 
This  description   is similar to the definition of classical Lie algebras in (\ref{G(W)}).
\end{rem}

\subsubsection{Diagram isomorphism $\Theta_{a,\ve}$}

Retaining the setting in Section~\ref{dia}, the automorphism $a$ on $\Gamma$
induces naturally an automorphism, still denoted by $a$,  on $\F$ such that $a: i\mapsto a(i)$ and $a: h \mapsto a(h)$. 
We define another automorphism  on $\F$ by rescaling  $a^{-1}$ on $\F$:
\begin{align}
\label{Phi}
\Phi_{a,\ve}: \F \to \F, \quad [i] \mapsto a^{-1} ([i]) ,  h\mapsto \ve(h)^{\frac{1-c}{2}} a^{-1}( h), \forall i\in I, h\in H. 
\end{align}
This is an algebra homomorphism, due to the multiplicative property of $\ve$: $\ve(ff') = \ve(f) \ve(f')$ if $f$ and $f'$ are two paths such that $\i(f')=\o (f)$.
Let $a^{-1}(\z)$ be the tuple whose $i$-th entry is $\z^{(a(i))}$.
The reason why we define $\Phi_{a,\ve}$ this way is due to the following identity.
\begin{align}
\label{Tt}
\Phi_{a,\ve}(\theta_{i,\z}) = \theta_{a^{-1}(i), a^{-1}(\z)}, \forall i\in I.
\end{align}

Indeed, we have
\begin{align*}
\Phi_{a,\ve}(\theta_{i, \z}) &= \sum_{\i (h)=i} \ve(h) (-1)^{\frac{1-c}{2}} a^{-1}(h) a^{-1}(\bar h) - \z^{(i)} [a^{-1}(i)]\\
&=  
\sum_{\i (h)=i} \ve(a^{-1}(h)) a^{-1}(h) a^{-1}(\bar h) - a^{-1}(\z)^{(a^{-1}(i))} [a^{-1}(i)] = \theta_{a^{-1}(i), a^{-1}(\z)}.
\end{align*}

Recall that  $a(W)$ is the $I$-graded space whose $i$-th component is $W_{a^{-1}(i)}$.
We then have an isomorphism of vector spaces by permutation
$
s_a: W \to a(W),
$
so that $s_a(w)_i=w_{a^{-1}(i)}$ where $w=(w_i)_{i\in I} \in W$.
This isomorphism defines an isomorphism
\[
r_a: \End(W) \to \End(a(W)), \phi \mapsto r_a(\phi): = s_a \circ \phi \circ s^{-1}_a.
\]
Let 
\begin{align}
\label{Theta}
\Theta_{a,\ve}: Z^{\z}_\bw \to Z^{a^{-1}(\z)}_{a^{-1} (\bw)}
\end{align}
be the isomorphism defined by
$
\Theta_{a, \ve} (\pi') = r_{a^{-1}} \circ \pi' \circ  \Phi^{-1}_{a, \ve}, \quad \forall \pi'\in Z^{\z}_\bw.
$
Due to (\ref{Tt}), $\Theta_{a,\ve}$ is well-defined. 
$\Theta_{a,\ve}$ is compatible with the diagram isomorphism $a^{-1}$ on $\M_{(0,\z)}(\bv,\bw)$ in (\ref{dia-a}).

\begin{lem}
\label{M-Z}
There is a commutative diagram.
\[
\begin{CD}
\M_{(0,\z)}(\bv, \bw) @>a^{-1}>> \M_{(0, a^{-1}(\z))} (a^{-1}(\bv), a^{-1}(\bw))\\
@V\vartheta VV @VV\vartheta V\\
Z^{\z}_\bw @> \Theta_{a,\ve} >> Z^{a^{-1}(\z)}_{a^{-1} (\bw)}
\end{CD} 
\]
\end{lem}

\begin{proof} 
Let $[\bx]=[x, p, q]\in \M_{(0,\z)} (\bv, \bw)$. It suffices to show that 
$\vartheta\circ a^{-1} ( \bx) = \Theta_{a, \ve} \circ \vartheta (\bx)$.
Given any arrow $h$, the evaluations of the left and right hand sides on $h$ are 
equal to $q_{a(\i (h))} x_{a(h)} p_{a(\o (h))}$. 
So the equality must hold, and the lemma follows.
\end{proof}

\subsubsection{Lusztig's reflection functor  on $Z^{\z}_\bw$}

Lusztig~\cite{L00} defined a reflection functor
\begin{align}
\label{L-reflection}
S_i : Z^{\z}_\bw \to Z^{s_i(\z)}_{\bw}, \ \pi' \mapsto S_i(\pi'),
\end{align}
where the evaluation of $S_i(\pi')$ on a given path $f$ is defined to be
\[
S_i(\pi') (f) =
\begin{cases}
\pi'([j])+\delta_{j, i} \z^{(i)} \mrm{id}_{W_i}, &\mbox{if}\ f=[j],\\
\sum_{J: J\subset J_0} \left (  \prod_{t\in J}  - \ve(h_t) \z^{(i)} \right ) \pi' (( h_1  \cdots h_s)^{ \vee J}), 
& \mbox{if} \ f= h_1\cdots h_s, s\geq 1,
\end{cases}
\]
here $J_0=\{ t\in [2, r] | \i (h_{t-1}) = i=\o (h_{t})\}$ and the superscript ${\vee J}\empty$ is the operation of removing the arrows $h_{t-1}$, $h_t$ for all $t\in J$.
(Note that $J$ can be an empty set.)
Since the isomorphism $S_i$ satisfies the Weyl group relations, we define $S_\w=S_{i_1} \cdots S_{i_s}$ for any 
$\w=s_{i_1}\cdots s_{i_s}\in \mathcal W$.

\subsubsection{The $\sigma$-quiver variety $\fS_{(0,\z)}(\bw)$}

Let 
\[
\sigma_0= S_\w \circ \Theta_{a,\ve} \circ \tau_0: Z^{\z}_\bw\to Z^{-\w a^{-1}(\z)}_{a^{-1} (\bw)}.
\]
When the isomorphism $\sigma_0$ becomes an automorphism, we can take its fixed-point.

\begin{Def}
\label{S0-def}
$\fS_{(0,\z)} (\bw) = \left (Z^{\z}_{\bw} \right )^{\sigma_0}$, if $\bw=a^{-1}(\bw), \z = - \w a^{-1} (\z).$
\end{Def}

When there is no danger of confusion, we use $\fS_0(\bw)$ for $\fS_{(0,\zeta)}(\bw)$.
By (\ref{T-compatible}), Lemma (\ref{M-Z}) and~\cite{L00},  the definition is compatible with the varieties $\fS_{\zeta}(\bv, \bw)$ with $a$ replaced by $a^{-1}$, 
and so we have  proper morphisms
\begin{align}
\label{pi-s}
\pi^{\sigma}: \fS_{\zeta}(\bv,\bw) \to \fS_{0} (\bw), \quad
\pi^{\sigma}: \fS_{\zeta}(\bw) \to \fS_0(\bw).
\end{align}
There is a $\G^{\sigma}_{\bw}$-action on $\fS_0(\bw)$ induced from $Z^{\z}_\bw$ and further a $\G^{\sigma}_{\bw}\times \mbb C^{\times}$-action  on $\fS_0(\bw)$ if $\z=0$.
It is clear that the morphisms in $(\ref{pi-s})$ are $\G^{\sigma}_{\bw}$-equivariant  (resp. $\G^{\sigma}_{\bw}\times \mbb C^{\times}$-equivariant if $\z=0$.)
It is also clear that $\pi^{\sigma}$ factors through the map under the same notation $\pi^{\sigma}$ in (\ref{pi-s-1}) and $\fS_1(\bv,\bw)$ is a closed subvariety of  $\fS_{(0,\z)}(\bw)$. 

\begin{rem}
\label{S0}
We can define $\fS_0(\bv,\bw) = \M_0(\bv,\bw) \cap \fS_0(\bw)$ in corresponding to $\fS_\zeta(\bv,\bw)$. 
This definition makes sense even when $\Gamma$ is not Dynkin, but in this generality we are not sure if 
$\fS_0(\bv,\bw)$ is an algebraic variety. 
On the other hand, we can always define the fixed-point locus $\M_0(\bv,\bw)^{a\tau}$ as long as $a(\bv) = \bv$ and $a(\bw) =\bw$.
This fixed-point locus does not have to assume $\Gamma$ being Dynkin either. 
\end{rem}
 
\begin{lem}
\label{S0-ind}
When  $\zeta_{\mbb C}=0$ and $ \w*\bv=\bv$,  
$\M_0(\bv,\bw)^{a\tau} = \fS_0(\bv,\bw)$. 
\end{lem}

\begin{proof}
In this case the reflection functor is the identity morphism by~\cite{L00}. 
\end{proof}

\subsection{Weyl group action on $\sigma$-quiver varieties}

Let $\zeta$ be generic in this section. 
The diagram automorphism $a$ induces an automorphism on the Weyl group $\mathcal W$.
Let $\mathcal W^{\w, a}=\{ x\in \mathcal W| x \w=\w x, a (x) = x\}$. 
This implies that the action $S_{x}$ for $x\in \mathcal W^{\w, a}$ on quiver varieties commutes with the action $S_{\w}$ and $a$.
Further, thanks to Lemma  ~\ref{comm-Si-tau}, it commutes with the isomorphism $\sigma$.
Hence we have 

\begin{prop}
\label{s-action}
The action $S_x$ for $x\in \mathcal W^{\w, a}$ restricts to an action on $\sigma$-quiver varieties:
\begin{align}
S^{\sigma}_{x}: \fS_{\zeta}(\bv, \bw) \to \fS_{x(\zeta)}(x* \bv, \bw), \quad \forall x\in \mathcal W^{\w, a}.
\end{align}
\end{prop}

As a consequence, we obtain

\begin{cor}
The group $\mathcal W^{\w, a}$ acts on the cohomology  group
$\H^*(\fS_{\zeta}(\bv, \bw),\mbb Z)$ when $\bw - \mbf C\bv=0$.
\end{cor}

\begin{rem}
When $\Gamma$ is of Dynkin type, $a=1$ and $\w=w_0$, 
the group $\mathcal W^{\w, a}$ is a Weyl group of type $B_{\ell}$ if $\Gamma$ is of type $A_{2\ell}$,
$C_{\ell}$ if $\Gamma$ is of type $A_{2\ell-1}$, $B_{\ell}$ if $\Gamma$ is of type $D_{\ell +1}$, $\ell$ even, 
of type $F_4$ if $\Gamma$ is of type $E_6$.
If $\Gamma$ is of type $D_4$, $\w=w_0$ and $a$ is the unique automorphism of order $3$, then the group $\mathcal W^{\w, a}$
is the Weyl group $G_2$. 
\end{rem}

\subsection{Symplectic structure on $\fS_{\zeta} (\bv, \bw)$}

In this section, we assume that the parameter $\zeta$ is generic.
Recall the symplectic vector space $\bM(\bv, \bw)$ from Section ~\ref{lavw}.
It is straightforward to check that the isomorphisms $a$ and $\tau$ on $\bM(\bv, \bw)$ in (\ref{dia-a}) and (\ref{tau-raw}), respectively, are  symplectomorphisms.
The varieties $\M_{\zeta}(\bv, \bw)$ inherit from $\bM(\bv, \bw)$ a symplectic structure. 
In turn, the fact that $a$ and $\tau$ being  symplectomorphisms implies that  
the induced isomorphisms $a$ and  $\tau_{\zeta}$ on $\M_{\zeta}(\bv, \bw)$ are also symplectomorphisms.
By the analysis in ~\cite[Theorem 6.1]{N03}, the reflection functor $S_{\w}$ is a hyper-K\"{a}hler isometry and in particular  a symplectomorphism.
Altogether, we see that the isomorphism $\sigma$ on $\M_{\zeta}(\bv, \bw)$ is a symplectomorphism.

\begin{prop}
\label{Ms-symplectic}
Assume that  $W$ is a $\delta_{\bw}$-formed space and the order of $\w$ is finite.
Then the $\sigma$-quiver variety $\fS_{\zeta}(\bv, \bw)$ is a  symplectic submanifold of  $\M_{\zeta}(\bv, \bw)$. 
\end{prop}

\begin{proof}
We only need to show that the restriction of the form $\omega$ to $\fS_{\zeta}(\bv, \bw)$ is non-degenerate.
Fix a point $[\bx] \in \fS_{\zeta}(\bv, \bw)$, the differential $d\sigma_{[\bx]}$ of the automorphism $\sigma$ 
at $[\bx]$ is an automorphism on the tangent space $T_{[\bx]} \M_{\zeta}(\bv, \bw)$.
By the assumption, we see from Proposition ~\ref{sig-4} that $\sigma^N=1$ for some $N$, and hence $(d\sigma_{[\bx]})^N=1$.
By a result of Edixhoven ~\cite{E92}, the fixed points of  $d\sigma_{[\bx]}$, i.e., the eigenspace of eigenvalue $1$, is exactly the tangent space
$T_{[\bx]} \fS_{\zeta}(\bv, \bw)$ of $\fS_{\zeta}(\bv, \bw)$ at $[x]$.
It thus yields the following eigenspace decomposition:
\[
T_{[\bx]} \M_{\zeta}(\bv, \bw) = T_{[\bx]} \fS_{\zeta}(\bv, \bw) \oplus C,
\]
where $C$ consists of linear combinations of eigenvectors of eigenvalues other than $1$.
By the above analysis,  the automorphism $\sigma$ is a symplectomorphism, 
and so this implies that $T_{[\bx]} \fS_{\zeta}(\bv, \bw)$ and $ C$ are orthogonal with each other. 
Hence the restriction of the symplectic form on them are non-degenerate. We are done.
\end{proof}

Since the zero fiber $\pi^{-1}(0)$ is Lagrangian, we see that the fiber $(\pi^{\sigma})^{-1}(0)=\pi^{-1}(0)^{\sigma}$  is isotropic. 
But the fiber $(\pi^{\sigma})^{-1}(0)$ is not coisotropic and hence not Lagrangian in general, 
see Remark~\ref{rem:abnormal} (3).
A nice consequence of Proposition~\ref{Ms-symplectic} is the semismallness of $\pi^{\sigma}$.

\begin{cor}
\label{semismall}
The map $\pi^{\sigma} : \fS_{\zeta}(\bv, \bw) \to \fS_1(\bv,\bw)$ is semismall.
\end{cor}

\begin{proof}
By~\cite[Theorem 7.2]{N98}, the fiber product $\M_{\zeta} (\bv, \bw) \times_{\M_0(\bv, \bw)} \M_{\zeta}(\bv, \bw)$ is a lagrangian subvariety
in $\M_{\zeta} (\bv, \bw) \times \M_{\zeta}(\bv, \bw)$. So its $\sigma$-analogue 
$\fS_{\zeta} (\bv, \bw) \times_{\fS_1(\bv, \bw)} \fS_{\zeta}(\bv, \bw)$ is isotropic in 
$\fS_{\zeta} (\bv, \bw) \times \fS_{\zeta}(\bv, \bw)$, and thus has at most half of the dimension of the latter manifold. 
So we have
\[
\dim \fS_{\zeta} (\bv, \bw) \times_{\fS_1(\bv, \bw)} \fS_{\zeta}(\bv, \bw)= \dim \fS_{\zeta}(\bv, \bw).
\]
According to~\cite[8.9.2]{CG}, it implies the corollary.
\end{proof}

\begin{rem}
(1). We refer to Remark~\ref{rem:can} for an alternative proof of Corollary~\ref{semismall}.

(2) By Corollary~\ref{semismall}, to show that $\fS_{\zeta}(\bv, \bw)$ is equidimensional, a.k.a., of pure dimension, it is enough to show
that the images of all connected components under $\pi^{\sigma}$ coincide.
\end{rem}

%%%%%%%%

\section{Quiver varieties and symmetric pairs}

In this section, we assume that the graph $\Gamma$ is a Dynkin diagram, $|a|=1$ or $2$, 
and the Weyl group element $\w = w_0$ is the longest element
in the Weyl group of $\Gamma$ and the sign function $\delta_{\bw}$ is $\Gamma$-alternating in the definition of the automorphism $\sigma$.
In this case,  we have $\sigma^2=1$ by Proposition~\ref{sig-4}.

\subsection{Restriction diagram}

Now we assume that $\zeta_{\mbb C} =0$ and $\xi_i=1$ for all $i\in I$.

Let  
$\T$ be a torus in $\G_{\bw}$.
Let $\M_{\zeta}(\bv, \bw)^{\T}$ be the $\T$-fixed point subvariety of $\M_{\zeta}(\bv, \bw)$.
For each homomorphism $\rho: \T \to \G_{\bv}$, let 
$$
\M_{\zeta} (\rho) =\{ [\bx] \in \M_{\zeta} (\bv, \bw) | t. \bx = \rho(t)^{-1}. \bx, \forall t\in \T\}.
$$
Nakajima ~\cite{N00} showed that
the $\M_{\zeta} (\rho)$ depends on the $\G_{\bv}$-conjugacy class of $\rho$ and there is  a partition of $\M_{\zeta} (\bv, \bw)^{\T}$ into connected components:
$
\M_{\zeta}(\bv, \bw)^\T = \coprod \M_{\zeta}(\rho),
$
where the union is over the set of all $\G_{\bv}$-conjugacy classes, say $\langle \rho\rangle$, of homomorphisms $\rho: \T \to \G_{\bv}$.
We  define 
\begin{align}
\fS_{\zeta} (\bv, \bw)^{\T} =\fS_{\zeta}(\bv, \bw) \cap \M_{\zeta}(\bv, \bw)^\T. 
\end{align}
Thus there is a decomposition:
\[
\fS_{\zeta} (\bv, \bw)^{\T} = \coprod_{\langle \rho \rangle} \fS_{\zeta}(\rho),
\quad
\fS_{\zeta}(\rho) = \fS_{\zeta} (\bv, \bw)  \cap \M_{\zeta}(\rho).
\]

In the special case $\T = \mbb C^{\times} \subseteq \G^{\sigma}_{\bw} $, then we can further define
\begin{align}
\fS_{\zeta} (\bv, \bw)^{+\mbb C^{\times}} =\{[\bx]\in \fS_{\zeta} (\bv, \bw) | \lim_{t \to 0} t. [\bx] \ \mbox{exists}\},\\
\fS_{\zeta} (\bv, \bw)^{- \mbb C^{\times}} =\{[\bx]\in \fS_{\zeta} (\bv, \bw) | \lim_{t \to \infty } t. [\bx] \ \mbox{exists}\}.
\end{align}
Similarly, there are varieties $\M_{\zeta}(\bv, \bw)^{\pm \mbb C^{\times}}$. Since the chosen $\mbb C^{\times}$ is in $\G^{\sigma}_{\bw} $, we have
\[
\fS_{\zeta} (\bv, \bw)^{ \pm \mbb C^{\times}} = \fS_{\zeta}(\bv, \bw) \cap \M_{\zeta}(\bv, \bw)^{\pm \mbb C^{\times}}.
\]
Thus there is the following hyperbolic localization/restriction diagram.
\begin{align}
\label{restriction-diag}
\xymatrix{
& \fS_{\zeta} (\bv, \bw)^{+ \mbb C^{\times}} \ar@{->}[dr]^{\kappa^+}  \ar@{->}[dl]_{\iota^+} & \\
\fS_{\zeta} (\bv, \bw) && \fS_{\zeta} (\bv, \bw)^{ \mbb C^{\times}} \\ 
& \fS_{\zeta} (\bv, \bw)^{- \mbb C^{\times}} \ar@{->}[ul]^{\iota^-}  \ar@{->}[ur]_{\kappa^-} 
}
\end{align}
where $\iota^{\pm}$ and $\kappa^{\pm}$ are natural embeddings and projections.

\subsection{A characterization of $\fS_{\zeta} (\bv, \bw)^{ \T} $}

Assume now that we have a decomposition 
$W=W^1 \oplus W^2 \oplus W^3 $ of the formed space $W$ such that  the following conditions hold.
\begin{itemize}
\item  For all $i\in I$,
the restrictions of the form to $W^1_i$ and $W^2_i\oplus W^3_i$ are non-degenerate.  

\item For all $i\in I$, $W^1_i$ and $W^2_i\oplus W^3_i$ are orthogonal to each other.
 
\item   For all $i\in I$, we have 
$(W^2_i)^{\perp}_{W^2_i\oplus W^3_i} = W^3_i$, $(W^3_i)^{\perp}_{W^2_i\oplus W^3_i}=W^2_i$, where $(-)^{\perp}_{W^2_i\oplus W^3_i}$ is taken in $W^2_i\oplus W^3_i$,   and hence $W^2_i$ and $W^3_i$ are maximal isotropic  in $W^2_i\oplus W^3_i$ of the same dimension. 
  
\item $a(W^1) = W^1$, $a(W^2) = W^2$ and $a(W^3)=W^3$.  
\end{itemize}
Set $\dim W^1 = \bw^1$, $\dim W^2 = \bw^2$ and $\dim W^3 = \bw^3$ so that they satisfy the  condition as follows from the above assumption.
 $$
 \bw^2=\bw^3, \bw^2=a \bw^2, a\bw^1=\bw^1\ \mrm{and}\ \bw=\bw^1 + 2\bw^2.
 $$ 
Consider the following 1-parameter subgroup in $\G_{\bw}^{\sigma}$.
\begin{align}
\label{A}
\lambda: \mbb C^{\times} \to \G^{\sigma}_{\bw} , \quad t\mapsto 
\mrm{id}_{W^1} \oplus t\cdot \mrm{id}_{W^2} \oplus t^{-1} \cdot  \mrm{id}_{W^3}.
\end{align}
By a result ~\cite[Lemma 4.4]{VV00} of Varagnolo and Vasserot, we have that 
$\M_{\zeta}(\rho)$ is empty unless $\rho$ is $\G_{\bv}$-conjugate to 
the group homomorphism
\begin{align}
\label{rho}
\mbb C^{\times} \to \G_{\bv}, \quad t\mapsto \mrm{id}_{V^1} \oplus t \cdot \mrm{id}_{V^2} \oplus t^{-1} \cdot \mrm{id}_{V^3}, 
\end{align}
for some decomposition $V = V^1 \oplus V^2 \oplus V^3$.
Moreover, if $\rho$ is of the latter form with the dimension vector of $V^1, V^2$ and $V^3$ being $\bv^1$, $\bv^2$ and $\bv^3$, respectively, then we have 
\begin{align}
\label{rho-VV}
\M_{\zeta} (\rho)  =  \M_{\zeta} (\bv^1, \bw^1) \times \M_{\zeta} (\bv^2, \bw^2) \times \M_{\zeta} (\bv^3, \bw^3).
\end{align}

Since $\lambda(\mbb C^{\times}) \leq \G_{\bw}^{\tau} \cap\G_{\bw}^a \leq \G_{\bw}^{\sigma}$ and $\sigma$ is $\G^{\sigma}_{\bw}$-equivariant, we see that
\[
\sigma(\M_{\zeta} (\bv, \bw)^{\lambda(\mbb C^{\times}) }) \subseteq \M_{\zeta} (a w_0 * \bv, a \bw)^{\lambda(\mbb C^{\times})}.
\]
Recall that the automorphism $\sigma$ is a composition $a \tau S_{w_0}$. 
If $\rho$ is of the form (\ref{rho}), we write $w_0 (\rho)$ be the group homomorphism
\begin{align}
\label{w-rho}
\mbb C^{\times} \to \G_{\bv}, \quad t\mapsto \mrm{id}_{w_0 *_{\bw^1} V^1} \oplus t \cdot \mrm{id}_{w_0 *_{\bw^2} V^2} \oplus t^{-1} \cdot \mrm{id}_{w_0 *_{\bw^2}V^3}, 
\end{align}
where $w_0*_{\bw^1}V^1$, $w_0*_{\bw^2} V^2$ and $w_0*_{\bw^2}V^3$ 
are vector spaces of dimension vectors $w_0 *_{\bw^1} \bv^1$, $w_0 *_{\bw^2} \bv^2$ and $w_0 *_{\bw^2} \bv^3$, respectively.
Since the construction of the automorphism $\sigma$ is independent of the choice of forms on $V$, we can, and shall, assume that
the non-degenerate symmetric form on $V$ has its restriction to $V^1$, $V^2$ and $V^3$ non-degenerate and that the latter spaces are orthogonal with each other.
Note that $a(\lambda(\mbb C^{\times})) =\lambda(\mbb C^{\times})$ and $\tau (\lambda(\mbb C^{\times})) =\lambda (\mbb C^{\times})$.
We observe that 
\[
a(\M_{\zeta}(\rho)) \subseteq \M_{a \zeta} (a \circ \rho \circ a^{-1}), 
\tau (\M_{\zeta} (\rho) ) \subseteq \M_{-\zeta} (\tau \circ \rho \circ \tau^{-1}),
S_{w_0} (\M_{\zeta} (\rho)) \subseteq \M_{w_0 \zeta} (w_0 (\rho)).
\]
Thus we have $\sigma (\M_{\zeta} (\rho)) \subseteq \M_{-a w_0 \zeta} ( w_0({}^{a \tau}\rho))$ 
where ${}^{a\tau}\rho$ is the compositions of $\rho$ with the automorphism $a\tau$ on $\G_{\bv}$.
This implies that $\fS_{\zeta} (\rho)$ is empty unless 
\begin{align}
\label{rho-w-rho}
\rho = w_0 ({}^{a \tau} \rho), \quad \mbox{up to a $\G_{\bv}$-conjugate.}
\end{align}
By comparing  (\ref{w-rho}) for $w_0 ({}^{a \tau} \rho)$ and (\ref{rho}), we see that $\fS_{\zeta} (\rho)$ is empty unless 
\begin{align}
\label{rho-cond-1}
\bv^1 =a ( w_0 *_{\bw^1} \bv^1) \quad \mbox{and} \quad 
\bv^2 =a ( w_0 *_{\bw^3} \bv^3).
\end{align}

Assume now that  the condition (\ref{rho-cond-1}) holds.
If $[\bx^2] \in \M_{\zeta} (\bv^2, \bw^2)$, then a slight generalization of the operation $\sigma$
yields an element in $\M_{\zeta} (\bv^3, \bw^3)$, denoted abusively by $\sigma([\bx^2])$.
(The involution $\tau$ in the definition is changed to be an isomorphism $\M_{\zeta}(\bv^2, \bw^2) \to \M_{-\zeta} (\bv^2, \bw^3)$ with respect to the above decomposition.)
Similarly, we can define $\sigma([\bx^3])$.
By definition, if $([\bx^1], [\bx^2], [\bx^3]) \in \M_{\zeta} (\rho)$ under the identification (\ref{rho-VV}), then
\[
\sigma ([\bx^1], [ \bx^2],  [ \bx^3]) = (\sigma ([\bx^1]),  \sigma ([\bx^3]),  \sigma ([\bx^2])).
\]
Thus, in light of the fact that $\sigma^2=1$, that  $([\bx^1], [\bx^2], [\bx^3]) \in \fS_{\zeta} (\rho)$ if and only if
$[\bx^1] = \sigma ([\bx^1])$ and $[\bx^2] = \sigma ([\bx^3])$.
Therefore, under the assumption (\ref{rho-cond-1}), 
%and (\ref{rho-w-rho}), 
it yields
\begin{align}
\fS_{\zeta} (\rho) \cong \fS_{\zeta} (\bv^1, \bw^1) \times \M_{\zeta} (\bv^2, \bw^2).  
\end{align}

Summing up the above analysis, there is

\begin{prop}
\label{A-fixed-sub}
Assume that $\T=\lambda(\mbb C^{\times})$ in (\ref{A}). Then there is an isomorphism:
\begin{align}
\fS_{\zeta} (\bv, \bw)^{\lambda(\mbb C^{\times})} \cong \coprod_{(\bv^1, \bv^2)}  \fS_{\zeta} (\bv^1, \bw^1) \times \M_{\zeta} (\bv^2, \bw^2),
\end{align}
where $\bw= \bw^1 + 2\bw^2$, $a \bw = \bw$, $a \bw^1 = \bw^1$ and  the union is over $(\bv^1, \bv^2)$ such that 
\begin{align}
\label{models}
\bv^1 = a (w_0 *_{\bw^1} \bv^1), \quad \bv^1 + \bv^2 + a(w_0 *_{\bw^2} \bv^2) = \bv.
\end{align}
\end{prop}

We shall write ``$\bv^1+\bv^2 \models \bv $'' if the condition (\ref{models}) is satisfied.
In general, we can consider a $1$-parameter subgroup in $\G^{\sigma}_{\bw}$ defined by
\begin{align}
\label{T-general}
\lambda: \mbb C^\times \to \G^{\sigma}_{\bw},  t\mapsto 
\mrm{id}_{W^1} \oplus \oplus_{i=2}^m (t^{\lambda_i} \mrm{id}_{W^i} \oplus t^{-\lambda_i} \mrm{id}_{W^{i,-}}),
\end{align}
where the pair $(W^i, W^{i,-})$ play a similar role as $(W^2, W^3)$ in (\ref{A}) and $0 < \lambda_1< \lambda_2 < \cdots <\lambda_m$. 
By applying the same argument, it gives rise to the following decomposition.
\[
\fS_{\zeta} (\bv, \bw)^{ \lambda(\mbb C^{\times})} \cong \coprod_{\bv^1 + (\sum_{i=2}^m \bv^i) \models \bv}  
\fS_{\zeta} (\bv^1, \bw^1) \times \prod_{i=2}^m \M_{\zeta} (\bv^i, \bw^i),
\]
where the product is taken in the natural order.
By summing up all $\bv$, the above decomposition gives rise to the following.

\begin{prop}
\label{A-fixed-sub-b}
Assume that $\lambda$ is given by (\ref{T-general}), there is an isomorphism:
\begin{align}
\fS_{\zeta}(\bw)^{ \lambda(\mbb C^{\times})} \cong 
\fS_{\zeta} (\bw^1) \times \prod_{i=2}^m \M_{\zeta} (\bw^i), \quad  \bw= \bw^1 + 2\sum_{i=2}^m \bw^i, a(\bw^i)= \bw^i.
\end{align}
\end{prop}

Now consider an arbitrary torus $\T \in \G^{\sigma}_{\bw}$ and the space $\Hom(\mbb C^{\times}, \T)$ of $1$-parameter subgroups in $\T$.
We form the real form $\Hom(\mbb C^{\times}, \T) \otimes_{\mbb Z} \mbb R$. 
There are $generic$  $1$-parameter subgroups in $\T$, i.e., those $\lambda$ such that 
$\fS_{\zeta}(\bv, \bw)^{ \lambda(\mbb C^{\times})}= \fS_{\zeta}(\bv, \bw)^{\T}$. 
The remaining ones are called $special$, giving rise to larger fixed-point subvarieties. 
The special $1$-parameter subgroups form unions of hyperplanes, i.e., $walls$,  in $\Hom(\mbb C^{\times}, \T) \otimes_{\mbb Z} \mbb R$, 
separating generic $1$-parameter subgroups into $chambers$, i.e., connected components of the complements of the unions of walls. 
From our analysis above, we see that if $\T$ is a maximal torus in $\G^{\sigma}_{\bw}$, then the chamber structure can be identified with
the usual Weyl chambers of type $B/C$.

\subsection{Coideal structure}

\label{sec:coideal}

We shall write $\IC_X$ the intersection cohomology complex attached to an algebraic variety $X$ (see~\cite{BBD82}). 
In particular, if $X = \sqcup_{i=1}^n X_i$ is a disjoint union of irreducible smooth varieties, then $\IC_X=\oplus_{i=1}^n \mbb C_{X_i}[\dim X_i]$, where $\mbb C_{X_i}$ is the constant sheaf on $X_i$ with coefficients in $\mbb C$.

Recall from  (\ref{pi-s}) that there is a proper map
\[
\pi^{\sigma} : \fS_{\zeta} (\bw) \to \fS_0(\bw).
\]
So one can consider the following complex:
\begin{align}
\P_{\fS_{0}(\bw)} = (\pi^{\sigma})_! \IC_{\fS_\zeta(\bw)}.
\end{align}
Similarly, we define the complexes $\P_{\M_0(\bw)}$ and $\P_{\fS_0(\bw)^{T}}$. 
The complexes
$\P_{\fS_{0}(\bw)}$, $\P_{\M_0(\bw)}$ and $\P_{\fS_0(\bw)^{ T}}$ are semisimple perverse sheaves, since the map
 $\pi^{\sigma}$ is semismall by Corollary~\ref{semismall}. 
Now we study the hyperbolic localization/restriction functor of Braden~\cite{B03} and Drinfeld-Gaitsgory~\cite{DG14} on the level of $\sigma$-quiver varieties.

\begin{thm}
\label{restriction-functor}
There exists a canonical isomorphism
\begin{align}
\label{can-iso}
\mrm{can}_{\mathscr C}: \P_{\fS_0(\bw)^{\T}} \overset{\cong}{\longrightarrow}
\kappa^+_* (\iota^+)^! \P_{\fS_0 (\bw)},
\end{align}
where $\kappa^+$ and $\iota^+$ are in (\ref{restriction-diag}) with the $1$-parameter subgroup of $\T$ in the chamber $\mathscr C$.
\end{thm}

\begin{proof}
For each $z \in \mbb C$, let $\zeta_{\mbb C} (z) \in \mbb C^I$  be the element whose $i$-th component is $z$. 
Let $\zeta(z) = (\xi, \zeta_{\mbb C}(z))$ where $\xi \in \mbb C^I$ be the element whose $i$-th component is $1$. 
We consider 
\[
\fS^{\clubsuit}(\bv, \bw) = \sqcup_{z\in \mbb C} \fS_{\zeta(z)} (\bv, \bw)
\quad \mbox{and} \quad
\fS^{\clubsuit}_1(\bv, \bw) = \sqcup_{z\in \mbb C} \fS_{\zeta (z), 1} ( \bw),
\]
where $ \fS_{\zeta(z), 1} (\bv, \bw)$ is the  $\fS_{1} (\bv, \bw)$ with $\zeta(z)$ emphasized.
Similarly, one can consider $\M^{\clubsuit} (\bv, \bw)$ and $\M^{\clubsuit}_1(\bv, \bw)$.
These are algebraic varieties defined in a similar way as $\M_{\zeta}(\bv, \bw)$ and $\M_{1}(\bv, \bw)$, and so are 
$\fS^{\clubsuit}(\bv, \bw) $ and $\fS^{\clubsuit}_1(\bv, \bw)$ 
as fixed-point subvarieties of automorphisms on the former algebraic varieties.
Similarly, there is a proper morphism over $\mbb C$:
\begin{align}
\label{Diag-deform}
\xymatrix{
\fS^{\clubsuit}(\bv, \bw) \ar[rr]^{\Pi^{\sigma}}  \ar[dr]_{f} & &  \fS^{\clubsuit}_1(\bv, \bw) \ar[dl]^{f_0}\\
& \mbb C & 
}
\end{align}
where the morphisms to $\mbb C$ are defined by sending a point in  $ \fS_{\zeta(z)} (\bv, \bw)$
and $ \fS_{\zeta(z), 1} (\bv, \bw)$ to  $z$.
The $\Pi^{\sigma}$  is a one-parameter deformation of $\pi^{\sigma}$.  In particular, the fiber of $\Pi^{\sigma}$ over $0 \in \mbb C$ is exactly $\pi^{\sigma}$. Moreover, the fiber over $\mbb C-\{0\}$ is an isomorphism 
\begin{align}
\label{Pi-a}
\fS^{\clubsuit} (\bv, \bw) \backslash \fS_{\zeta}(\bv, \bw) \cong
\fS^{\clubsuit}_1(\bv, \bw) \backslash \fS_1(\bv, \bw).
\end{align}
Now we apply the argument in ~\cite[5.4]{N16}.
By  (\ref{Pi-a}), there is a canonical isomorphism
\begin{align}
\label{can_a}
\uppsi_{f_0}[-1] \IC_{\fS^{\clubsuit}_1(\bv,\bw)}|_{ \fS^{\clubsuit}_1(\bv,\bw) \backslash \fS_1(\bv,\bw)}
\cong \pi^{\sigma}_! \IC_{\fS(\bv,\bw)},
\end{align}
where $\uppsi_{f_0}$ is the nearby cycle functor with respect to $f_0$.
Similarly, there is an isomorphism
\begin{align}
\label{can_aT}
\uppsi_{f^\T_0}[-1] \IC_{\fS^{\clubsuit}_1(\bv,\bw)^{\T}}|_{ \fS^{\clubsuit}_1(\bv,\bw)^{ \T} \backslash \fS_1(\bv,\bw)^{\T}} 
\cong \pi^{\sigma,\T}_! \IC_{\fS(\bv,\bw)^{\T}},
\end{align}
where $f^\T_0$ is an analogue of $f_0$ in (\ref{Diag-deform}) and $\pi^{\sigma, \T}$ is the restriction of $\pi^{\sigma}$ to its $\T$-fixed point part. By (\ref{Pi-a})
and the relative symplectic form on $\M^{\clubsuit}(\bv, \bw)$ induced from (\ref{symplectic-form}), it yields a canonical isomorphism
\begin{align}
\label{can_b}
\IC_{\fS^{\clubsuit}_1(\bv,\bw)^{\T}}|_{\fS^{\clubsuit}_1(\bv,\bw)^{\T}   \backslash \fS_1(\bv,\bw)^{\T}}
\overset{\cong}{\longrightarrow}
\tilde \kappa_* \tilde \iota^! \IC_{\fS_1(\bv,\bw)} |_{\fS^{\clubsuit}_1(\bv,\bw)^{\T}  \backslash \fS_1(\bv,\bw)^{\T}}.
\end{align}
where $\tilde \kappa$ and $\tilde \iota$ are the counterparts of $\kappa^+$ and $\iota^+$ respectively 
on $\fS^{\clubsuit}_0(\bv,\bw)$.
Therefore, there is a canonical  isomorphism
\begin{align}
\label{can_c}
\begin{split}
\pi^{\sigma,\T}_! \IC_{\fS(\bv,\bw)^{\T}} 
&\overset{(\ref{can_aT})}{\cong} \pi^{\sigma,\T}_! \uppsi_{f^\T} [-1] \IC_{\fS^{\clubsuit} (\bv,\bw)^{\T}}|_{ \fS^{\clubsuit} (\bv,\bw)^{\T} \backslash \fS(\bv,\bw)^{\T}} \\
&\overset{(\star)}{\cong} \uppsi_{f^\T_0}[-1] \Pi^{\sigma,\T}_! \IC_{\fS^{\clubsuit} (\bv,\bw)^{\T}}|_{ \fS^{\clubsuit} (\bv,\bw)^{\T} \backslash \fS(\bv,\bw)^{\T}} \\
&\overset{(\dagger)}{\cong}  \uppsi_{f^\T_0} [-1]\IC_{\fS^{\clubsuit}_1(\bv,\bw)^{\T}}|_{ \fS^{\clubsuit}_1(\bv,\bw)^{\T} \backslash \fS_1(\bv,\bw)^{\T}} \\
&\overset{(\ref{can_b})}{\cong}  \uppsi_{f^\T_0} [-1]\tilde \kappa_* \tilde \iota^! \IC_{\fS^{\clubsuit}_1(\bv,\bw)}|_{ \fS^{\clubsuit}_1(\bv,\bw)^{\T} \backslash \fS_1(\bv,\bw)^{\T}}  \\
& \overset{(\flat)}{\cong} (\kappa^+)_* (\iota^+)^! \uppsi_{f_0}[-1] \IC_{\fS^{\clubsuit}_1(\bv,\bw)}|_{ \fS^{\clubsuit}_1(\bv,\bw) \backslash \fS_1(\bv,\bw)}   \\
& \overset{(\ref{can_a})}{\cong} (\kappa^+)_* (\iota^+)^!  \pi^{\sigma}_! \IC_{\fS(\bv,\bw)},
\end{split}
\end{align}
where  ($\star$) is due to the fact that a nearby cycle functor commutes with proper maps,
($\dagger$) is due to $\pi^\T$ is an isomorphism when restricts to $\fS^{\clubsuit}(\bv,\bw)^{\T} \backslash \fS(\bv,\bw)^{\T}$,
and ($\flat$) is due to the fact that a nearby cycle functor commutes with hyperbolic restrictions.
The theorem follows by summing the above (\ref{can_c})  over all $\bv$.
\end{proof}

\begin{rem}
\label{rem:can}
(1). We refer the reader to ~\cite[3(iv)]{N13} for the subtleties in choosing an isomorphism in (\ref{can-iso}).

(2) Since a nearby cycle functor, shifted by $[-1]$, sends perverse sheaves to perverse sheaves, the isomorphism (\ref{can_a}) implies that 
the complex $ \pi_* \IC_{\fS(\bv,\bw)}$ is a semisimple perverse sheaf. This in turn implies that the map $\pi^{\sigma}$ is
semismall onto its image (see Corollary~\ref{semismall}). 
\end{rem}

By Proposition~\ref{A-fixed-sub-b}, we see that there is a canonical isomorphism
\[
\P_{\fS_0(\bw)^{\T}} \cong \oplus_{ \bv^2 \models \bv} \pi^{\sigma}_! \IC_{\fS_{\zeta}( \bw^1)} 
\boxtimes \pi_!\IC_{\M_{\zeta}(\bw^2)}.
\]
Thus the complex $\P_{\fS_0(\bw)^{ \T}} $ is a direct summand of the complex
$\P_{\fS_0(\bw^1)} \boxtimes \P_{\M_0(\bw^2)} $.
So the restriction functor $ \kappa^+_* (\iota^+)^! $ induces an algebra homomorphism
\begin{align}
\label{D-End}
\Delta^{\sigma}_{\bw^1, \bw^2}: \End ( \P_{\fS_0 (\bw)}) \to \End( \P_{\fS_0(\bw^1)}) \otimes \End ( \P_{\M_0(\bw^2)}), 
\end{align}
where $ \bw^1 + 2\bw^2 = \bw$ and $a\bw^i=\bw^i$ for $i=1,2$.
(Here the endomorphisms are taken inside  abelian categories of perverse sheaves.)

Now we consider the following Steinberg-like varieties.
\begin{align}
\fY (\bw) = \sqcup_{\bv^1, \bv^2} \fY (\bv^1, \bv^2, \bw), \quad
\fY (\bv^1, \bv^2, \bw) = \fS_{\zeta} (\bv^1, \bw) \times_{\fS_0(\bw) }  \fS_{\zeta} (\bv^2, \bw).
\end{align}
Similarly, the notation $\Z(\bw)$ is defined with respect to Nakajima varieties $\M_{\zeta}(\bv,\bw)$.
Let $\H_{top}(X)$ denote the top Borel-Moore homology of $X$, see~\cite{CG}. 
From ~\cite[8.9.7]{CG}, there is an algebra isomorphism 
\[
\End ( \P_{\fS_0 (\bw)}) \cong \H_{top} (\fY(\bw))
\quad \mbox{and} \quad
\End ( \P_{\M_0 (\bw)}) \cong \H_{top} (\Z(\bw)).
\]
In terms of top Borel-Moore homology, the algebra homomorphism in (\ref{D-End}) becomes the following algebra homomorphism, 
denoted by the same notation. 
\begin{align}
\label{D-BM}
\Delta^{\sigma}_{\bw^1, \bw^2} : \H_{top} (\fY (\bw)) \to \H_{top} (\fY (\bw^1)) \otimes \H_{top} (\Z(\bw^2)),
\ \mbox{if} \ \bw^1 + 2\bw^2 = \bw, a\bw^i=\bw^i.
\end{align}
In the same vein, we have an algebra homomorphism:
\begin{align}
\Delta_{\bw^1, \bw^2} : \H_{top}(\Z(\bw)) \to \H_{top} (\Z(\bw^1) ) \otimes \H_{top}(\Z(\bw^2)) , \quad \mbox{if}\  \bw^1 + \bw^2 =\bw.
\end{align}
By the canonical choice of isomorphism in Theorem~\ref{restriction-functor}, we obtain

\begin{prop}
\label{coideal}
The algebra homomorphism $\Delta^{\sigma}_{\bw^1, \bw^2}$ satisfies the coassociativity, that is
\begin{align}
(\Delta^{\sigma}_{\bw^1, \bw^2} \otimes 1) \circ \Delta^{\sigma}_{\bw^1+\bw^2+
\bw^3, \bw^3} =
(1\otimes \Delta_{\bw^2,\bw^3} ) \circ  \Delta^{\sigma}_{\bw^1,\bw^2+\bw^3} , 
\end{align}
for all $\bw^1+2(\bw^2 +\bw^3)  =\bw$ and $a\bw^i=\bw^i$ for $i=1,2,3$.
\end{prop}

To $\bw$, we define $\r_{\bw}$ by  $(\r_{\bw})_i = \frac{\bw_i}{2} - \frac{1- (-1)^{\bw_i}}{4}$, that is,
$(\r_{\bw})_i$ is the rank of the $i$-th isometry group with respect to the $i$-th $\delta_{\bw,i}$-form.
Let $\s_{\bw} = \bw - 2 \r_{\bw}$. 
If $\fS_{\zeta}(\s_{\bw}) = \{ \mrm{pt}\}$, then the coproduct (\ref{D-BM}) becomes
the following algebra homomorphism.
\begin{align}
\label{j}
\jmath: \H_{top} (\fY (\bw)) \to \H_{top} (\Z(\r_{\bw})), \quad \mbox{if} \ \fS_{\zeta}(\s_{\bw}) = \{ \mrm{pt}\}.
\end{align}

Recall that $\Gamma$ is a Dynkin diagram
and $w_0$ is the longest element in the associated Weyl group.
There is an involution $\theta$ on $I$ such that $w_0 (\alpha_i) =  - \alpha_{\theta(i)}$ 
where $\alpha_i$ is the $i$-th simple root of $\Gamma$. 
Let $\mathfrak g_{\Gamma}$ be the simple Lie algebra associated to $\Gamma$ with Chevalley generators $\{ e_i, f_i, h_i| i\in I\}$.
Then the assignment $e_i \mapsto f_{a \theta(i)}$, $f_i \mapsto e_{a\theta (i)}$ and $h_i \mapsto - h_{a\theta(i)}$ for all $i\in I$ defines
an involution, denoted by $\sigma$, on $\mathfrak g_{\Gamma}$. 
It is known that the fixed-point Lie subalgebra  $\mathfrak g_{\Gamma}^{\sigma}$ 
is generated by $e_i + f_{a\theta(i)}$ and $h_i - h_{a\theta(i)}$ for all $i\in I$. 
The algebra $\mathfrak g_{\Gamma}^{\sigma}$ is usually denoted by $\mathfrak k$ in the introduction. 
The pair $(\mathfrak g_{\Gamma}, \mathfrak g_{\Gamma}^{\sigma})$ then forms a so-called symmetric pair. 
The Lie algebra $\mathfrak g_{\Gamma}^{\sigma}$ is classified by the Satake diagrams without black vertices (i.e., $X=\O $ in~\cite{K14}).
Specifically, they are listed in the Table 1.
\begin{table}
\label{Fixed-point-subalgebra}
\begin{center}
\begin{tabular}{| l | l | l | }
\hline
$(\Gamma, |a|)$ & $\mathfrak g^{\sigma} \equiv \mathfrak k$ & Satake type\\
\hline 
($\mrm A_\ell, 1): \ell=2p$ & $\mathfrak{sl}_p\oplus \mathfrak{gl}_{p+1}$ & AIII\\
\hline
($\mrm A_\ell, 1): \ell = 2p-1$ & $\mathfrak{sl}_p \oplus \mathfrak{gl}_p$ & AIII\\
\hline
$(\mrm A_\ell, 2)$ & $\mathfrak{so}_{\ell+1}$ & AI\\
\hline
$(\mrm D_\ell, 1): \ell \ \mbox{odd}$ & $\mathfrak{so}_{\ell -1} \oplus \mathfrak{so}_{\ell+1} $ & DI\\
\hline
$(\mrm D_\ell, 1): \ell \ \mbox{even}$ & $\mathfrak{so}_\ell \oplus \mathfrak{so}_{\ell}$ & DI\\
\hline
$(\mrm D_\ell, 2): \ell \ \mbox{odd}$ & $\mathfrak{so}_\ell \oplus \mathfrak{so}_\ell$ & DI\\
\hline
$(\mrm D_\ell, 2): \ell \ \mbox{even}$ & $\mathfrak{so}_{\ell -1}\oplus \mathfrak{so}_{\ell +1}$ & DI\\
\hline
$(\mrm E_6, 1)$ & $\mathfrak{sl}_2 \oplus \mathfrak{sl}_6$ & EII\\
\hline 
$(\mrm E_6, 2)$ & $\mathfrak{sp}_4$ & EI\\
\hline 
$(\mrm E_7, 1)$ & $\mathfrak{sl}_8$ & EV\\
\hline
$(\mrm E_8, 1)$ & $\mathfrak{so}_{16}$ & EVIII\\
\hline
\end{tabular}
\end{center}
\caption{}
\end{table}

Let $U(\mathfrak g_{\Gamma}^{\sigma})$ be the universal enveloping algebra of $\mathfrak g^{\sigma}_{\Gamma}$. 
With the coassociativity and (\ref{j}) in hand, we make the following conjecture.

\begin{conj}
\label{conj-1}
There is a nontrivial algebra homomorphism
$
U(\mathfrak g^{\sigma}_{\Gamma}) \to \H_{top}(\fY(\bw)). 
$
\end{conj}

When the Dynkin diagram $\Gamma$ is of type $A$ and $a=1$, 
this conjecture can be shown by the results in ~\cite{BKLW} and an argument similar to
~\cite{BG99}.

\subsection{The stable map $\mathrm{Stab}_{\mathscr C}$}
Recall that $\fS_{\zeta}(\bw)$ has a $\mbb C^{\times}$-action by scaling. 
Since the maps $\pi^{\sigma}$ and $\pi^{\sigma, \T}$ are $\bT \equiv \T \times \mbb C^{\times}$ -equivariant. 
The isomorphism $\mrm{can}_{\mathscr C}$ in (\ref{can-iso}) also holds in the derived category of 
$\bT$-equivariant $\mbb C$-constructible sheaves.
This $\bT$-equivariant version of the canonical isomorphism (\ref{can-iso}) 
is the same as the one given by  Maulik-Okounkov's stable envelope~\cite{MO12}, 
as explained in ~\cite{N16} (the statement after Corollary 5.4.2 therein).
With the help of $\mrm{can}_{\mathscr C}$ in (\ref{can-iso}), one obtains the stable map  on the torus-equivariant cohomologies.
\begin{align}
\Stab_{\mathscr C}:  \H^{[*]}_{\bT} (\fS_{\zeta}(\bw)^{\T}) \to \H^{[*]}_{\bT} (\fS_{\zeta}(\bw)),
\end{align}
where $[*]$ is the shifted degree defined by $\H^{[*]}_{\bT}(?) = \H^{* + \dim ?}_{\bT}(?)$.
This sheaf-theoretic definition of $\Stab_{\mathscr C}$ is formal and contained in~\cite{N16}. To the convenience of the reader, we reproduce it here. 
For simplicity, we write $\mathfrak X \equiv \fS_{\zeta}(\bw)$ and $\mathfrak X_0 \equiv \fS_{0}(\bw)$  in this process.
There are  canonical isomorphisms/identifications:
\begin{align}
\label{stable-def-1}
\H^{[*]}_{\bT}(\mathfrak X^{\T}) \cong \Ext^*_{\bT}(\mbb C_{\mathfrak X^{\T}}, \IC_{\mathfrak X^{\T}})
\cong \Ext^*_{\bT} ( (\pi^{\sigma, \T})^* \mbb C_{\mathfrak X_0^\T}, \IC_{\mathfrak X^{\T}}) 
\cong  \Ext^*_{\bT} (\mbb C_{\mathfrak X_0^\T}, (\pi^{\sigma, \T})_* \IC_{\mathfrak X^{\T}} ) .
\end{align}
On the other hand, there are canonical isomorphisms
\begin{align}
\label{stable-def-2}
\begin{split}
\Ext^*_{\bT} *(\mbb C_{\mathfrak X_0^{\T}}, \kappa^+_* (\iota^+)^! \pi^{\sigma}_* \IC_{\mathfrak X}) 
 \cong \Ext^*_{\bT} ( \mbb C_{\mathfrak X_0^{+\mbb C^{\times}}}, (\iota^+)^! \pi^{\sigma}_* \IC_{\mathfrak X})
& \cong \Ext^*_{\bT} ( \mbb C_{\mathfrak X_0^{+\mbb C^{\times}}}, (\tilde \pi^{\sigma})_* (\tilde \iota)^! \IC_{\mathfrak X})\\
& \cong \Ext^*_{\bT}( \mbb C_{\mathfrak X^{+\mbb C^{\times}}}, (\tilde \iota)^! \IC_{\mathfrak X})\\
& \cong \Ext^*_{\bT}(\mbb C_{\mathfrak X} , (\tilde \iota)_! (\tilde \iota)^! \IC_{\mathfrak X}), 
\end{split}
\end{align}
where  the 1-parameter subgroup $\mbb C^{\times}$ is chosen from the chamber $\mathscr C$, 
$\tilde  \pi^{\sigma}$ and $\tilde \iota$ are given in the following cartesian diagram.
\[
\begin{CD}
\mathfrak X @<\tilde \iota << \mathfrak X^{+\mbb C^{\times}} \\
@V\pi^{\sigma} VV @VV\tilde \pi^{\sigma} V\\
\mathfrak X_0 @<\iota^+<< \mathfrak X_0^{+\mbb C^{\times}}
\end{CD}
\]
Note that there is an adjunction $\mrm{adj}: (\tilde \iota)_! (\tilde \iota)^! \to \mrm{id}$, which induces a morphism
\begin{align}
\mrm{adj}: \Ext^*_{\bT}(\mbb C_{\mathfrak X} , (\tilde \iota)_! (\tilde \iota)^! \IC_{\mathfrak X}) 
\to \Ext^*_{\bT}(\mbb C_{\mathfrak X}, \IC_{\mathfrak X} ) = \H^{[*]}_{\bT}(\mathfrak X).
\end{align}
The stable map is thus defined to be the following composition.

\begin{align}
\begin{split}
\H^{[*]}_{\bT}(\mathfrak X^{\T})   \overset{(\ref{stable-def-1})}{=}
\Ext^*_{\bT}(\mbb C_{\mathfrak X} , (\tilde \iota)_! (\tilde \iota)^! \IC_{\mathfrak X})
& \overset{\mrm{can}_{\mathscr C}}{\longrightarrow}
\Ext^*_{\bT} *(\mbb C_{\mathfrak X_0^{\T}}, \kappa^+_* (\iota^+)^! \pi^{\sigma}_* \IC_{\mathfrak X}) \\
&\overset{(\ref{stable-def-2})}{=}
\Ext^*_{\bT}(\mbb C_{\mathfrak X} , (\tilde \iota)_! (\tilde \iota)^! \IC_{\mathfrak X})
\overset{\mrm{adj}}{\longrightarrow}
\H^{[*]}_{\bT}(\mathfrak X).
\end{split}
\end{align}
So, modulo  the canonical identifications in (\ref{stable-def-1}) and (\ref{stable-def-2}), we have
$
\Stab_{\mathscr C} = \mrm{adj} \circ \mrm{can}_{\mathscr C}. 
$

\subsection{Universal $\K$-matrix}
Let $\mbb C[\mrm{Lie} (\bT)]$ be the coordinate ring of the Lie algebra $\mrm{Lie}(\bT)$, as an affine space. 
Let $\bF_{\T}$ be its rational field. 
The cohomologies $\H^{[*]}_{\bT}( \fS_{\zeta}(\bw)^{\T}) $ and $\H^{[*]}_{\bT}( \fS_{\zeta}(\bw))$ 
are $\mbb C[\mrm{Lie}(\bT)]$-modules (since $\H^{*}_{\bT}(\{\mrm{pt}\} ) = \mbb C[\mrm{Lie}(\bT)]$)
and the map $\Stab$ is compatible with the $\mbb C[\mrm{Lie}(\bT)]$-module structures. 
It is known that after change of coefficients from $\mbb C[\mrm{Lie}(\bT)]$ to $\bF_{\T}$, the stable map is invertible. 
%\red{Say more here?}
Following Maulik-Okounkov~\cite{MO12},  the $\K$-matrix is defined by   
\begin{align}
\K_{\mathscr C', \mathscr C} =  \Stab_{\mathscr C'}^{-1 } \circ \Stab_{\mathscr C} \in 
\End_{\mbb C[ \mrm{Lie} (\bT)]} ( \H^*_{\bT} (\fS_{\zeta}(\bw)^{\T}))
\otimes_{\mbb C[\mrm{Lie}(\bT)]} \bF_{\T}.
\end{align}
Clearly, one has 
\begin{align}
\label{KCC}
\K_{\mathscr C'', \mathscr C'} \K_{\mathscr C', \mathscr C} = \K_{\mathscr C'', \mathscr C}
\quad \mbox{and}\quad
\K_{\mathscr C, \mathscr C} =1.
\end{align}

If the chamber $\mathscr C$ is determined by the inequalities $a_m>\cdots>a_1>0$  and 
$\fS(\bw)^{ \T} \cong \prod_{i=1}^m \M(\bw^i)$ we write $\K_{- \mathscr C, \mathscr C}$ by
$\K_{\underline \bw} (\underline a)$ where $\underline \bw = (\bw_1, \cdots, \bw_m)$ and $\underline a=(a_1, \cdots, a_m)$. 
For $\underline \bw=(\bw^1)$ and $\underline a=(a_1)$, we write $\K_{\bw^1}(a_1)$ for $\K_{\underline \bw}(\underline a)$.
Let us list an example of the $\K$-matrix.
\begin{example}
Let $\Gamma=A_1$ and $\bw=2$ with $\delta_{\bw}=-1$ so that 
$\fS_{\zeta}(\bw)= T^* \mathscr B^{\sp_2} \cong T^* \mbb P^1$ the cotangent bundle of complete flag variety of 
$\SP_2 \cong \mrm{SL}_2$.  So by ~\cite[4.1.2]{MO12}, we have
\[
\K_1(a) = \frac{1 - \frac{\hbar}{a} \begin{bmatrix}0 & 1\\ 1 & 0\end{bmatrix}}{1- \frac{\hbar}{a}}
\in \End (\mbb C^2(\hbar, a)).
\]
\end{example}
 
Since $\K_{\mathscr C, \mathscr C} =1$, the identity in (\ref{KCC})  yields

\begin{prop}
The $\K$-matrix is unitary, i.e.,
$\K_{\bw} (-a) = \K_{\bw} (a)^{-1}$.
\end{prop}

Suppose that $\T$ is a two-dimensional torus in $\G^{\sigma}_{\bw}$. 
A typical $1$-parameter subgroup in $\T$ is given similar to (\ref{T-general}) for various $(a_1, a_2)\in \mbb Z^2$:
\begin{align}
\label{reflection-para}
\lambda_{a_1, a_2}: 
\mbb C^{\times} \to \G^{\sigma}_{\bw}, t\mapsto \mrm{id}_{W^0} \oplus (t^{a_1}\mrm{id}_{W^1} \oplus t^{-a_1} \mrm{id}_{W^{1,-}})
\oplus (t^{a_2}\mrm{id}_{W^2} \oplus t^{-a_2} \mrm{id}_{W^{2,-}}).
\end{align}
The real form of $\mrm{Lie}(\T)$ is thus a plane $\mbb R^2$, whose walls are the lines $a_1=0$, $a_2=0$, $a_1-a_2=0$ and $a_1+a_2=0$.
In particular, there are $8$ chambers in $\mrm{Lie}(\T)$, which is exactly the Weyl chambers of type $B_2/C_2$ (see Figure~\ref{Reflection}).
Let $\R_{\mathscr C', \mathscr C}$ denote Maulik-Okounkov's $\R$-matrix on 
the torus equivariant cohomology of Nakajima variety. 
Under this setting, the $\K$-matrix satisfies the reflection equation, instead of the Yang-Baxter equation for $\R$-matrix.

\begin{prop}
The $\K$-matrix satisfies the following reflection equation. 
\begin{align}
\label{Ref-equation}
%\K(a) &= \K(-a)^{-1}\\
%
\K_{\bw^2}(a_2) \R(a_1 + a_2) \K_{\bw^1}(a_1) \R(a_1-a_2) & = 
\R(a_1-a_2) \K_{\bw^1}(a_1) \R(a_1+a_2) \K_{\bw^2}(a_2),
\end{align}
where the $\K_{\bw^i}(a_i)$'s are  understood as $\K_{\bw^1} (a_1) \otimes 1$ and $1\otimes \K_{\bw^2} (a_2)$, respectively. 
\end{prop}

\begin{proof}
If the $\R$s in the equation are replaced by the $\K$s, then this holds because both sides are the same as 
$\K_{-\mathscr C, \mathscr C}$ where $\mathscr C$ is the chamber in  Figure~\ref{Reflection}.
It remains to show that $\R(a_1\pm a_2)=\K(a_1\pm a_2)$. 
Under the setting (\ref{reflection-para}), 
the condition $a_1-a_2=0$ defines a subtorus $\T'$ in $\T$  so that 
$\fS_{\zeta}(\bw)^{ \T'} \cong \fS_{\zeta}(\bw^0)\times \M_{\zeta}(\bw^1+\bw^2)$ 
by Proposition \ref{A-fixed-sub}. 
By general properties of $\K$/$\R$-matrices, the $\K$-matrix $\K_{\mathscr C', \mathscr C}$ of crossing the wall $a_1-a_2=0$ 
in Figure~\ref{Reflection} is the same as the $\K$-matrix for the torus $\T/\T'$ on $\fS_{\zeta}(\bw)^{ \T'} $. 
Note that $\T/\T'$ acts trivially on the component  $\fS_{\zeta}(\bw^0)$,
so the latter $\K$-matrix is Maulik-Okounkov's original $\R$-matrix $\R(a_1-a_2)$ on quiver varieties. 
This shows that $\K(a_1-a_2)=\R(a_1-a_2)$. The other equality can be obtained by the same argument. The proposition is thus proved.
\end{proof}

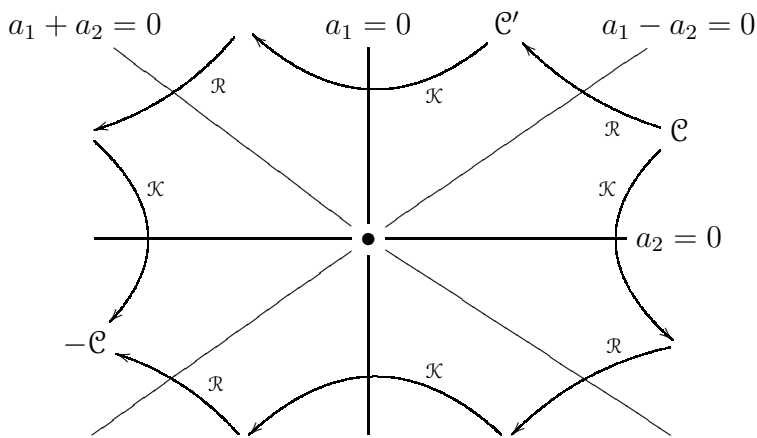
\begin{figure}
\begin{center}
\xymatrix{
&&& a_1+ a_2=0&  \ar@/^/[dl]^(.3) \R &a_1=0 &\ar@/^2pc/[ll]^(.3) \K  \mathscr C' & a_1-a_2=0\\
&&& \ar@/^2pc/[dd]^(.3) \K & && &  \mathscr C \ar@/^/[ul]^(.3) \R \ar@/_2pc/[dd]_(.3)\K \\
&&& &&\bullet \ar@{-}[rr] \ar@{-}[ll]  \ar@{-}[uu] \ar@{-}[dd] \ar@{-}[uurr]  \ar@{-}[uull] \ar@{-}[ddll] \ar@{-}[ddrr]&& a_2=0\\
&&& - \mathscr C &&&& \ar@/_/[dl]_(.3) \R \\
&&&  &\ar@/_/[ul]_(.3) \R&& \ar@/_2pc/[ll]_(.3) \K &
}
\end{center}
\caption{Reflection Equation}
\label{Reflection}
\end{figure}

In general, the $\K$-matrix $\K_{\underline \bw}(\underline a)$ can be obtained  from the $\K_{\bw^i}(a_i)$'s 
via  the so-called fusion procedure.
In particular, when $\underline \bw$ contains two components, it reads as follows. 

\begin{prop}
\label{fusion}
One has
\begin{align}
\begin{split}
\K_{\bw^1, \bw^2} (a_1, a_2)& =  \R(a_2-a_1) \K_{\bw^2} (a_2)  \R (a_1+a_2)   \K_{\bw^1}(a_1) \\
&= \K_{\bw^1}(a_1) \R(a_1+a_2) \K_{\bw^2}(a_2) \R(a_2-a_1).
\end{split}
\end{align}
\end{prop}

\begin{proof}
We have $\K_{\bw^1, \bw^2} = \K_{- \mathscr C', \mathscr C'}(a_1, a_2)$ where the chamber $\mathscr C'$ is given in Figure~\ref{Reflection}.
The proposition follows by multiplying $\R(a_2 - a_1)$ on both sides of Equation (\ref{Ref-equation}) and using the unitary property of $\K$-matrix. 
\end{proof}

\begin{rem}
As we learnt from  Weiqiang Wang, algebraic $K$-matrix for quantum symmetric pairs of type AIII/IV first appeared in~\cite{BW13}.
The relationship between the $\K$-matrix in this section and the algebraic ones in~\cite{BaK16} is not clear.   
\end{rem}

\subsection{Twisted Yangian via the FRT formalism}
Let $\Y$ be Maulik-Okounkov's Yangian, which is formulated in the framework of Faddeev-Reshetikhin-Takhtajan~\cite{FRT}. 
In particular, the algebra $\Y$  is a subalgebra  in the product $\prod_{\bw, \T} \H^*_{\bT}(\fS_{\zeta}(\bw)^{\T})\otimes \bF_{\bT}$ generated by
the matrix coefficients in the $\R$-matrix $\R_{0,1}(a_0-a_1) \cdots  \R_{0,m}(a_0-a_m)$ with respect to $a_0$ (see~\cite[6.2.6]{MO12}). 
Let $\Y_{\sigma}$ be the subalgebra of $\Y$   
generated by the matrix coefficients with respect to $a_0$ of the operators
\begin{align}
\label{K-tensor}
\begin{split}
\R_{0,m}(a_0-a_m)  \cdots \R_{0, 1}(a_0 -a_1) \cdot  \K_0(a_0) \cdot  \R_{0,1}(a_0-a_1) \cdots  \R_{0,m}(a_0-a_m).
\end{split}
\end{align}

Note that operators of the above form satisfy the reflection equation, which 
can be shown by induction in the following.
In light of this property, we shall call $\Y_{\sigma}$ a $twisted$ $Yangian$.

\begin{prop}
Let $\R_{ij}(a_i-a_j)$ be an $\R$-matrix at $(i,j)$-component on the tensor $F_0(a_0) \otimes F_1(a_1) \otimes \cdots \otimes F_m(a_m)$
Let $\K_0(a_0)$ be a $\K$-matrix at the $0$-component.
Then  the operator, say $\S(a_0)$, in (\ref{K-tensor}) satisfies the reflection equation
\[
\R_{0, 1} (a_0 - b_0)  \S_0 (a_0)  \R_{0, 1} (a_0 + b_0) \S_1 (b_0)  
= 
\S_1 (b_0) \R_{0, 1} (a_0+ b_0) \S_0 (a_0) \R_{0, 1} (a_0- b_0),
\]
in the tensor $F_0(a_0) \otimes F_0(b_0) \otimes F_1(a_1)\otimes\cdots \otimes F_m(a_m)$.
\end{prop}

\begin{proof}
We shall prove the proposition by induction.
When $m=1$, we shift the subindex by $1$ and set $(a_0, b_0, a_1)=(u, v, w)$.
Then we have
\begin{align*}
& \R_{12}(u-v) \S_1(u) \R_{12}(u+v) \S_2(v) = \\
& =\R_{12}(u -v) \left ( \R_{13}(u-w) \K_1(u) \R_{13}(u-w) \right ) \R_{12}(u+v) \left ( \R_{23}(v-w) \K_2(v) \R_{23}(v-w) \right ) \\
& =  \R_{12}(u -v) \R_{13}(u-w) \K_1(u)  \R_{23}(v-w)  \R_{12}(u+v) \R_{13}(u-w) \K_2(v) \R_{23}(v-w)\\
& = \R_{12}(u -v) \R_{13}(u-w)  \R_{23}(v-w)  \K_1(u)  \R_{12}(u+v)  \K_2(v)  \R_{13}(u-w) \R_{23}(v-w)\\
& = \R_{23}(v-w) \R_{13}(u-w) \R_{12}(u -v)  \K_1(u)  \R_{12}(u+v)  \K_2(v)  \R_{13}(u-w) \R_{23}(v-w)\\
& = \R_{23}(v-w) \R_{13}(u-w)  \K_2(v) \R_{12}(u+v)  \K_1(u)  \R_{12}(u -v) \R_{13}(u-w) \R_{23}(v-w)\\
& = \R_{23}(v-w) \R_{13}(u-w)  \K_2(v) \R_{12}(u+v)  \K_1(u) \R_{23}(v-w) \R_{13}(u-w) \R_{12}(u -v)\\
& = \R_{23}(v-w)  \K_2(v)   \R_{13}(u-w) \R_{12}(u+v) \R_{23}(v-w) \K_1(u)  \R_{13}(u-w) \R_{12}(u -v)\\
&=  \R_{23}(v-w)  \K_2(v) \R_{23}(v-w) \R_{12}(u+v)  \R_{13}(u-w) \K_1(u)  \R_{13}(u-w) \R_{12}(u -v)\\
& = \S_2(v) \R_{12}(u+v) \S_1(u) \R_{12} (u-v),
\end{align*} 
where the second equality is due to the modified Yang-Baxter equation
\[
\R_{13}(u-w) \R_{12}(u+v) \R_{23}(v-w)  = \R_{23}(v-w)  \R_{12}(u+v) \R_{13}(u-w),
\]
via the unitary property of $\R$, the third equality is due to the  commutativity of $\K_i(a)$ with $\R_{j, k}$ if $i\neq j, k$, and
the fifth one is due to the reflection equation of the $\K$-matrices. 

In general, we write $\S^{(m)}(u)$ for the $\S$ on $F_0(a_0)\otimes \cdots  \otimes F_m(a_m)$. Then we have
$$\S^{(m)}(u) = \R_{0,m} (a_0-a_m) \S^{(m-1)}(u) \R_{0,m}(a_0-a_m).$$
In particular, there is  the following with $(a_0, b_0)=(u,v)$.
\begin{align*}
& \R_{0 1} (u-v) \S^{(m)}_1(u)  \R_{01} (u+v) \S^{(m)}_2(v) =\\
& = \R_{01}(u-v) \R_{0, m+1} (u- a_m)  \S^{(m-1)}_1(u) \R_{0, m+1}(u- a_m) R_{01} (u+v) \R_{1,m+1}(v- a_m) \\
& \hspace{10cm} \S^{(m-1)}_2(v)  R_{1,m+1}(v-a_m) \\
& =  \R_{01}(u-v) \R_{0, m+1} (u- a_m)   \R_{1,m+1}(v- a_m) \S^{(m-1)}_1(u) R_{01} (u+v)  \S^{(m-1)}_2(v)   \\
&\hspace{10cm} \R_{0, m+1}(u- a_m)  R_{1,m+1}(v-a_m) \\
& = \R_{1,m+1}(v- a_m) \R_{0, m+1} (u- a_m)  \R_{01}(u-v)  \S^{(m-1)}_1(u) R_{01} (u+v)  \S^{(m-1)}_2(v)   \\
& \hspace{10cm} \R_{0, m+1}(u- a_m)  R_{1,m+1}(v-a_m) \\
&= \R_{1,m+1}(v- a_m) \R_{0, m+1} (u- a_m) \S^{(m-1)}_2(v) R_{01} (u+v)  \S^{(m-1)}_1(u)  \R_{01}(u-v) \\
&\hspace{10cm} \R_{0, m+1}(u- a_m)  R_{1,m+1}(v-a_m) \\
&= \R_{1,m+1}(v- a_m) \R_{0, m+1} (u- a_m) \S^{(m-1)}_2(v) R_{01} (u+v)  \S^{(m-1)}_1(u)   R_{1,m+1}(v-a_m)  \\
&\hspace{10cm} \R_{0, m+1}(u- a_m)  \R_{01}(u-v) \\
&= \R_{1,m+1}(v- a_m) \S^{(m-1)}_2(v)   \R_{0, m+1} (u- a_m) R_{01} (u+v)  R_{1,m+1}(v-a_m)  \S^{(m-1)}_1(u)  \\
&\hspace{10cm} \R_{0, m+1}(u- a_m)  \R_{01}(u-v) \\
& =  \R_{1,m+1}(v- a_m) \S^{(m-1)}_2(v) R_{1,m+1}(v-a_m) R_{01} (u+v)   \R_{0, m+1} (u- a_m) \S^{(m-1)}_1(u)  \\
& \hspace{10cm} \R_{0, m+1}(u- a_m)  \R_{01}(u-v) \\
& = \S^{(m)}_2(v) \R_{01}(u+v) \S^{(m)}_1(u) \R_{01}(u-v).
\end{align*}
Proposition is thus proved.
\end{proof}

From the definition, both algebras $\Y$ and $\Y_{\sigma}$ act on 
$\H^*_{\bT}(\M_{\zeta}(\bw)) \otimes \bF_{\bT}$ and their tensor products.
Summing up the above analysis, it yields

\begin{thm}
\label{GSP}
There is a $(\Y, \Y_{\sigma})$-action on $\H^*_{\bT}(\M_{\zeta}(\bw)) \otimes \bF_{\bT}$ and their tensor products.
\end{thm}

%%%%%%%%

\section{Example I: cotangent bundles of  isotropic flag varieties}

In this section, we show that a natural involution on the cotangent bundle of the $n$-step partial flag variety of type $\mrm A_n$ is a special case of the automorphism $\sigma$. As a consequence,
we show that cotangent bundles of partial flag varieties of classical type are examples of the quiver varieties 
$\fS_{\zeta}(\bv, \bw)$.

\subsection{Notations}

In this section, we assume that  the graph $\Gamma$ is a Dynkin diagram  of type $A_n$. 
\begin{align}
\label{Dynkin-A}
\begin{split}
%&
\xymatrix{
\mrm A_n (n\geq 1):   \quad 1 \ar@<1ex>[r] & 2 \ar@<1ex>[l] \ar@<1ex>[r]  & \cdots \ar@<1ex>[l] \ar@<1ex>[r]  & n \ar@<1ex>[l] 
}\\
\end{split}
\end{align}
Assume further
that the dimension vectors $\bv$ and $\bw$ of the pair of vector spaces  $V$ and $W$ satisfy that
$\bw_i=0$ for $i\geq 2$ and $\bw_1\geq \bv_1\geq \bv_2\geq \cdots \geq \bv_n$. 
Let $ \mathscr F_{\bv, \bw}$ be the variety of $n$-step partial flags,
$F = ( W\equiv F_0 \supseteq F_1 \supseteq \cdots \supseteq F_n \supseteq F_{n+1} \equiv 0)$,
such that $\dim F_i =\bv_i$ for all $1\leq i\leq n$.
The cotangent bundle $T^* \mathscr F_{\bv, \bw}$ of $\mathscr F_{\bv, \bw}$ can be defined as follows.
\begin{align}
T^* \mathscr F_{\bv, \bw} =\{ (x, F) \in \End (W) \times \mathscr F_{\bv, \bw} | x( F_i) \subseteq F_{i+1}, \forall 0\leq i\leq n\}.
\end{align}

From Section ~\ref{involution}, we assume that 
$W\equiv W_1$ is a formed space with the bilinear-form $(-,-)_W$.  
Let $G(W)$ be the subgroup of $\mrm{GL}(W)$ leaving the form invariant. 
In particular, if the form is a $\delta$-form, 
$G(W) = \mrm O_{\bw_1}$ the orthogonal group if  $\delta=1$ 
and $G(W) = \mrm{Sp}_{\bw_1}$ the symplectic group if $\delta=-1$.
Let $\mathfrak g(W)$ be the Lie algebra of $G(W)$. Then we have
\begin{align}
\label{G(W)}
\begin{split}
G(W) & = \{ g\in \mrm{GL}(W) | gg^*=1\}, \\
\mathfrak g(W) & =\{x\in \End(W) | x = - x^*\}.
\end{split}
\end{align}

For each subspace $F_i \subseteq W$, we can define its orthogonal complement $F_i^{\perp} =\{ w\in W| (x, F_i)_W=0\}$.
We set $F^{\perp} = ( F_{n+1}^{\perp} \supseteq F_n^{\perp} \supseteq \cdots \supseteq F_1^{\perp} \supseteq F_0^{\perp})$. 
Note that $w_0 * \bv = ( \bw_1 - \bv_n, \bw_1 - \bv_{n-1} ,\cdots, \bw_1 - \bv_1)$.
%%%: The computations are as follows:
%\begin{align}
%w_0 * \bv = w_0 (\sum_{i\in I} \bv_i \alpha_i -  \bw_1 \Lambda_1) + \bw_1 \Lambda_1\\
%
%=\sum_{i\in I} - \bv_{\theta(i)} \alpha_i + \bw_1 (\Lambda_1 + \Lambda_n)\\
%
%= \sum_{i\in I} - \bv_{\theta(i)} \alpha_i + \bw_1 \sum_{i\in I} \alpha_i \\
%
%= \sum_{i\in I} (\bw_1 - \bv_{\theta(i)}) \alpha_i.  
%\end{align}
%Here  $\Lambda_1 + \Lambda_n = \sum_{i\in I} \alpha_i$, since
%in the standard presentation of root systems $\Lambda_i = \sum_{j\leq i} e_i$ and $\alpha_i = e_i - e_{i+1}$.
%%%
So if $F\in \mathscr F_{\bv, \bw}$, then  $F^{\perp} \in \mathscr F_{w_0*\bv, \bw}$,
thus taking $\perp$ defines an involution $\sigma_1: \mathscr F_{\bv, \bw} \to \mathscr F_{w_0 * \bv, \bw}$.
In the case when $w_0 * \bv = \bv$, that is $\bv_i + \bv_{n+1-i} = \bw_1$ for all $1\leq i\leq n$, 
the fixed point subvariety $\mathscr F^{\sigma_1}_{\bv, \bw}$ under
$\sigma_1$ is a partial flag variety of the classical group $G(W)$.
Its cotangent bundle is given by
\[
T^* \mathscr F^{\sigma_1}_{\bv, \bw} =\{ (x, F) \in \mathfrak g(W) \times \mathscr F^{\sigma_1}_{\bv, \bw} | x(F_i) \subseteq F_{i+1},
\ \forall 0\leq i\leq n\}.
\]
 More generally, the assignment $(x, F) \mapsto (- x^*, F^{\perp})$ defines an isomorphism
\begin{align}
\label{sigma1}
\sigma_1: T^* \mathscr F_{\bv, \bw} \to T^* \mathscr F_{w_0 * \bv, \bw}.
\end{align}

We must prove the well-definedness of $\sigma_1$. 
We only need to show $-x^* (F^{\perp}_{i} ) \subseteq   F^{\perp}_{i-1}$ for all $1\leq i\leq n+1$.
For  any $u \in F_{i-1}$ and $u' \in F_{i}^{\perp}$,  one has
\begin{align}
(u, - x^* (u'))_{W} = - (x(u), u')_W =0, 
\end{align}
since $x(u) \in F_i$. This implies that $-x^*(u') \in F_{i-1}^{\perp}$ as required. 

From the above analysis, one has
\begin{align}
T^* \mathscr F^{\sigma_1}_{\bv, \bw} = (T^* \mathscr F_{\bv, \bw})^{\sigma_1}, \quad \mbox{if} \ w_0 * \bv= \bv.
\end{align}

Note that $\mrm f\mapsto \! \ ^{\tau}\mrm f = (\mrm f^*)^{-1}$ defines an automorphism $\tau: \G_{\bw} \to \G_{\bw}$. 
The isomorphism $\sigma_1$ is $\tau$-equivariant, i.e., 
$\sigma( \mrm f. (x, F)) =  \ ^{\tau} \! \mrm f.  \sigma_1(x, F)$ for all $\mrm f\in \G_{\bw}$ and $(x, F) \in T^* \mathscr F_{\bv, \bw}$.
In turn, this induces a $G(W)$-action on the fixed point variety  $T^* \mathscr F^{\sigma_1}_{\bv, \bw}$. 

\subsection{Identification with $\sigma$-quiver varieties}
\label{quiver-cotangent}

In this section, we assume that $a=1$ and  $\w=w_0$ the longest element in the Weyl group $\mathcal W$. 
Let $\theta: I\to I$ be the involution defined by $\theta(i) = n+1-i$ for all $1\leq i \leq n$.
Assume that the parameter $\zeta=(\xi, \zeta_{\mbb C})$ satisfies that $\zeta_{\mbb C} =0$,  $\theta(\xi) =\xi$ 
and $\xi_i>0$ for all $i\in I$.
We choose the function $\ve: H\to \{\pm 1\}$ to be $\ve(h) = \o (h) - \i (h)$ with the label in (\ref{Dynkin-A}).
Recall from ~\cite[Theorem 7.3]{N94}, there is an isomorphism $\phi: \M_{\zeta}(\bv, \bw) \to T^* \mathscr F_{\bv, \bw}$ of varieties
given by 
\begin{align}
\label{M-TF}
[\bx]=[x_h, p_i, q_i] \mapsto 
(q_1 p_1, 
W_1\supseteq \mbox{im} \ q_1 \supseteq  \mbox{im} \ q_1 y_{1}  \supseteq  \cdots \supseteq 
\mbox{im} \ q_1 y_{1} \cdots y_{n-1} \supseteq 0),
\end{align}
where $y_{i}$ is the $x_h$ such that $\o (h) = i+1$ and $\i (h) = i$.

\begin{thm}
\label{sigma-sigma-1}
Under Nakajima's isomorphism, the isomorphism  $\sigma=\sigma_{\zeta, w_0}$ for $\w=w_0$ (\ref{sigma})  on quiver varieties 
gets identified with the isomorphism $\sigma_1$ (\ref{sigma1}) on the cotangent bundle of flag varieties.
In particular, if $w_0 * \bv = \bv$, the quiver variety $\fS_{\zeta} (\bv, \bw)$ is 
the cotangent bundle $T^* \mathscr F^{\sigma_1}_{\bv, \bw}$.
\end{thm}

\begin{proof}
Similar to ~\cite[Theorem 7.3]{N94}, there is an isomorphism $\psi: \M_{-\zeta}(\bv, \bw) \to T^* \mathscr F_{w_0 * \bv, \bw}$ of varieties
given by
\begin{align}
\label{M-TF-dual}
[\bx] \mapsto 
(q_1 p_1, W_1 \supseteq \ker  x_{ n-1}\cdots x_{1}  p_1  \supseteq \cdots \supseteq \ker x_{1} p_1  \supseteq \ker p_1 \supseteq 0),
\end{align}
where $x_{i} $ stands for the $x_h$ with $\o (h) =i$ and $\i (h) = i+1$. 
The proof consists of two steps.
The first step of the proof is to show that the following diagram is commutative.
\[
\begin{CD}
\M_{\zeta} (\bv, \bw) @>\tau_{\zeta}>> \M_{-\zeta} ( \bv, \bw) \\
@V\phi VV @VV\psi V \\
T^* \mathscr F_{\bv, \bw} @> \sigma_1 >> T^* \mathscr F_{w_0 * \bv, \bw}
\end{CD}
\]
where the isomorphism $\tau_{\zeta}$ is from (\ref{tau}).
Note that $^{\tau} q_1 \ ^{\tau} p_1 = - (q_1p_1)^*$. 
So it suffices to show that the associated flags to $[\bx]$ and  $[^{\tau}\bx]$  can be obtained from each other via the operator $\perp$.
More precisely, setting $x_{0}= y_{0} =1$, we need to show
\begin{align}
\label{im-ker}
(\mrm{im} \ q_1 y_{1} \cdots y_{i-1})^{\perp} = \ker \ \!  ^{\tau} x_{i-1}\cdots \ \!  ^{\tau} x_{1}  \ \! ^{\tau}p_1  , 
\quad \forall 1\leq i\leq n.
\end{align}
Let $f_i = q_1 y_{1} \cdots y_{i-1}$. 
Then by definition, $^{\tau} x_{i-1}\cdots \ \!  ^{\tau} x_{1}  \ \! ^{\tau} p_1  = (- 1)^i f_i^*$.
So for all $u\in \mrm{im} \ f_i$, $u' \in  \ker \ (-1)^i  f_i^* = \ker f^*_i$, there is $v_i\in V_i$ such that $f_i (v_i) = u$ and
$(u, u')_{W_1} = (f_i(v_i), u')_{W_1} = (v_i, f_i^* (u'))_{V_i} =0$. 
Hence we obtain that  $\ker f^*_i \subseteq (\mrm{im}  f_i)^{\perp}$.
Since the linear maps  $q_1$, $y_{1}, \cdots, y_{i-1}$ are injective,  
we have that $\dim \mrm{im} f_i + \dim \ker f_i^* = \dim V_i + \dim W_1 - \dim V_i = \bw_1$, which implies  the equality (\ref{im-ker}).
This proves that the above diagram is commutative.

The second step is to show that the involution $S_{w_0}: \M_{-\zeta} (\bv, \bw) \to \M_{\zeta} (w_0* \bv, \bw)$ commutes with the maps in (\ref{M-TF}) and (\ref{M-TF-dual}), that is $\psi = \phi S_{w_0}$. Here we use $w_0(-\xi) = \theta(\xi) = \xi$.
Let us fix a reduced expression of the longest element $w_0 = s_n (s_{n-1} s_n) \cdots (s_1\cdots s_n)$, 
so that 
$S_{w_0} = S_n (S_{n-1} S_n) \cdots (S_1\cdots S_n)$.
Observe that each time we apply $S_i$, the parameter on the affected quiver varieties always has a negative value at $i$.
This allows us to use the definition (\ref{Si}).
Now fix a point $[\bx] \in \M_{-\zeta} (\bv, \bw)$, with $\bx$ given by
\[
\xymatrix{
W_1 \ar@<1ex>@{->>}[r]^{p} & 
V_1 \ar@<1ex>[l]^{q} \ar@<1ex>@{->>}[r]^{x_{1}} 
& V_2 \ar@<1ex>[l]^{y_{1}} \ar@<1ex>@{->>}[r]^{x_{2}}  
& \cdots \cdots \ar@<1ex>[l]^{y_{2}}  \ar@<1ex>@{->>}[r]^{x_{n-2}} 
&  V_{n-1} \ar@<1ex>[l]^{y_{n-2}} \ar@<1ex>@{->>}[r]^{x_{n-1}} 
& V_{n} \ar@<1ex>[l]^{y_{n-1}}
}
\] 
By applying $S_1 \cdots S_n$ to $\bx$, the point $[\bx]$ gets sent to a point represented by 
\[
\xymatrix{
W_1 \ar@<1ex>[r]^{qp} &
\ker x_{n-1} \cdots x_{1} p \ar@<1ex>@{^{(}->}[l] \ar@<1ex>@{->>}[r]^{ p} 
&  \ker x_{n-1} \cdots x_{1}  \ar@<1ex>[l]^{q} \ar@<1ex>@{->>}[r]^{x_{1}}  
& \cdots \cdots \ar@<1ex>[l]^{y_{1}}  \ar@<1ex>@{->>}[r]^{ x_{n-3}} 
&  \ker x_{n-1} x_{n-2}   \ar@<1ex>[l]^{y_{n-3}} \ar@<1ex>@{->>}[r]^{x_{n-2}} 
& \ker x_{n-1}  \ar@<1ex>[l]^{y_{n-2}}
}
\]
where the arrow without name is the natural inclusion.
By applying $S_i\cdots S_n$ for $i=2, \cdots, n$ consecutively, we see that the point $S_{w_0} ([\bx])$ is represented by
\[
\xymatrix{
W_1 \ar@<1ex>[r]^{qp} &
\ker x_{n-1} \cdots x_{1} p \ar@<1ex>@{^{(}->}[l] \ar@<1ex>[r]^{qp} 
&  \ker x_{n-2} \cdots x_{1}p  \ar@<1ex>@{^{(}->}[l] \ar@<1ex>[r]^{qp}  
& \cdots \cdots \ar@<1ex>@{^{(}->}[l]  \ar@<1ex>[r]^{qp} 
&  \ker  x_{1} p  \ar@<1ex>@{^{(}->}[l] \ar@<1ex>[r]^{qp} 
& \ker p  \ar@<1ex>@{^{(}->}[l] 
}
\]
By (\ref{M-TF}) and (\ref{M-TF-dual}), 
it implies immediately that $\phi S_{w_0} ([\bx]) = \psi ([\bx])$, completing the proof.
\end{proof}

\begin{rem}
\label{rem:abnormal}
The identification in Theorem~\ref{sigma-sigma-1} indicates that the geometry of general $\sigma$-quiver varieties
is quite complicated as we shall see from the following remarks.
\begin{enumerate}

\item 
In general, the quiver variety $\fS_{\zeta}(\bv, \bw)$  is not connected. An example is as follows. 
When $\delta=1$ and $\bw_1$ is even, $G(W)= \mrm O_{\bw_1}$ is an even orthogonal group, 
hence the cotangent bundle $T^* \mathscr F^{\sigma_1}_{\bv, \bw}$ has two connected/irreducible components if $n$ is odd.

\item 
The morphism $\pi^{\sigma}$ is not a resolution of singularities in general.
Indeed, we have a commutative diagram
\[
\begin{CD}
\fS_{\zeta}(\bv, \bw) @>\phi>\cong> T^*\mathscr F^{\sigma_1}_{\bv, \bw} \\
@V\pi^{\sigma}VV @VV\Pi V\\
\fS_1(\bv, \bw) @>\phi_0>> \mathfrak g(W)
\end{CD}
\]
where $\Pi$ is the first projection. 
So we only need to know that $\Pi$ is not a resolution of singularities. 
For example, the morphism 
$\Pi: T^* \mathscr F^{\mathfrak{sp}_6}_{\bv, \bw} \to \overline{\mathcal O}_{(2^2, 1^2)}^{\mathfrak{sp}_6}$,
with $\bv = ( 5, 1)$ and $\bw=(6,0)$, is not a resolution of singularities since the fiber at any point of Jordan type $2^2, 1^2$ contains two points.
Note that  
$\Pi$ is generically finite to its image.

Note that $\Pi$ factors through the affinization map of $ T^*\mathscr F^{\sigma_1}_{\bv, \bw}$ 
which is a resolution of singularities.
It is not clear if the same holds for the affinization map of $\fS_{\zeta}(\bv, \bw)$.

\item
The fiber $(\pi^{\sigma})^{-1}(0)$ is not equidimensional/lagrangian in general. 
Indeed, by Corollary~\ref{cor:i-Nakajima-Maffei},  it
corresponds to the  fiber of a partial resolution of nilpotent Slodowy slices at $e_0$, 
which  is  not necessarily equidimensional/lagrangian. 
%Move this claim.  Refer to Spaltenstein's book, 11.6. 

\item
In general, the variety $\fS_1(\bv, \bw)$ is non-normal. 
In ~\cite{KP82}, the orbit closure $\overline{\mathcal O}_{3^2, 1^2}$ in $\mathfrak{sp}_8$ is connected non-normal, 
which is a special case of $\fS_1(\bv, \bw)$. 
\end{enumerate}
\end{rem}

\section{Nakajima-Maffei isomorphism and  symmetry}

In this section, we assume again that the Dynkin diagram is of type $\mrm A_n$.
We recall Nakajima-Maffei's isomorphism of  the quiver varieties and partial Springer resolutions of nilpotent Slodowy slices of type $\mrm A$.
We deduce, as preliminary results for later study,  a rectangular symmetry and the column-removal and row-removal reductions in~\cite[Proposition 5.4]{KP81}. 
(During the preparation of the paper, we noticed that these applications have been appeared in ~\cite[Section 9]{H15}.)

\subsection{Nakajima-Maffei theorem}
\label{Maffei}

Recall from Section ~\ref{quiver-cotangent}, we define $\ve (h) = \o (h) - \i (h)$ for all arrow $h$.
For any pair $(\bv, \bw)$,
we define a new pair  $(\widetilde \bv=(\widetilde \bv_i)_{1\leq i\leq n}, \widetilde \bw=(\widetilde \bw_i)_{1\leq i\leq n})$
where
\begin{align}
\label{tild-bw}
\widetilde \bv_i = \bv_i + \sum_{j\geq i+1} (j-i) \bw_j, \quad 
\widetilde \bw_i = \delta_{i, 1} \sum_{1\leq j \leq n} j \bw_j,\quad \forall 1\leq i\leq n.
\end{align}
To a pair $(V, W)$ of $I$-graded vector spaces of dimension vector $(\bv, \bw)$,
we associate  a new pair $(\widetilde V, \widetilde W)$ of dimensional vector $(\widetilde \bv, \widetilde \bw)$ whose $i$-th component
is given by
\begin{align}
\label{t-V-W}
\widetilde  V_i = V_i \oplus \oplus_{1\leq h\leq j-i} W_j^{(h)}, \quad
\widetilde W_1 = \oplus_{1\leq h\leq j} W_j^{(h)},
\end{align}
where $W_j^{(h)}$ is an identical copy of $W_j$ for all $h$.
For convenience, we set $V_0=0$ and $\widetilde V_0 =\widetilde W_1$.
With respect to the decomposition of $\widetilde V_i$, 
a linear map $\widetilde x_i: \widetilde V_i \to \widetilde V_{i+1}$ is  a collection of the following four types of linear maps.
\begin{align}
\label{xi-dec}
\begin{split}
X_i : V_i \to V_{i+1}, \
T^V_{i, j, h}: V_i \to W_j^{(h)},\
T^{j', h'}_{i, V}: W_{j'}^{(h')} \to V_{i+1},\
T^{j', h'}_{i, j, h} : W_{j'}^{(h')} \to W_{j}^{(h)}, 
\end{split}
\end{align}
for all $j\geq i+1$, $1\leq h\leq j-i$.
Similarly, to give a linear map $\widetilde y_i: \widetilde V_{i+1} \to \widetilde V_i$ is the same as to give a collection of the following four types of linear maps for all $j\geq i+1$, $1\leq h\leq j-i$.
\begin{align}
\label{yi-dec}
\begin{split}
Y_i : V_{i+1} \to V_{i},  \
S^V_{i, j, h}: V_{i+1} \to W_j^{(h)},\
S^{j', h'}_{i, V}: W_{j'}^{(h')} \to V_{i}, \
S^{j', h'}_{i, j, h} : W_{j'}^{(h')} \to W_{j}^{(h)}.
\end{split}
\end{align}

Following Maffei,  we define the following numerical data.
\begin{align}
\label{grad}
\begin{split}
\mrm{grad} (T^{j', h'}_{i, j, h} ) & = \min ( h-h'+1, h-h'+1+j'-j), \\
%
%\mrm{deg} (T^{j', h'}_{i, j, h} ) & = 2(h-h'+1) + j'-j, \\
%
\mrm{grad} (S^{j', h'}_{i, j, h}) & = \min (h-h', h-h'+j'-j).\\
%
%\mrm{deg} (S^{j', h'}_{i, j, h}) & = 2(h - h') + j' -j.
\end{split}
\end{align}

Let $W_i' = \oplus_{1\leq h\leq j-i} W_j^{(h)}$ so that $\widetilde V_i = V_i \oplus W_i'$ for all $0\leq i\leq n$.
Let $e_i: W_i'\to W_i'$ be a linear map 
whose component  $W_j^{(h)} \overset{e_i}{\to} W_j^{(h-1)}$ is equal to   $\mrm{id}_{W_j}$ for $2\leq h\leq j-i$ and $0$ otherwise.
Let $f_i: W_i' \to W_i'$ be a linear map 
whose component  $W_j^{(h)} \overset{f_i}{\to} W_j^{(h+1)}$ is equal to $h(j-i-h) \mrm{id}_{W_j}$ for $1\leq h\leq j-i -1$ 
and $0$ otherwise. 
The triple $(e_i, f_i, [e_i, f_i])$ in $\mathfrak{sl}(W'_i)$
is a Maffei $\mathfrak{sl}_2$-triple .

Now assume that $\zeta_{\mbb C}=0$ and we write $\Lambda (V, W)$ instead of  $\Lambda_{\zeta_{\mbb C}}(V, W)$.
We preserve the convention used in previous sections: $x_i/y_i$ stands for the map associated to the arrow $i\to i+1/ i+1\to i$.
Following Maffei, an element $\widetilde \bx =(\widetilde x_i, \widetilde y_i, \widetilde p_i, \widetilde q_i) \in \Lambda(\widetilde V, \widetilde W)$, represented in the form of (\ref{xi-dec})-(\ref{yi-dec}), is $transversal$ if it satisfies the following conditions.
\begin{align}
T^{V}_{i, j, h} & =0,  \tag{t1} \label{t1}\\
T^{j', h'}_{i, V} & =0, &&  \mbox{if} \ h' \neq 1, \tag{t2} \label{t2}\\
T^{j', h'}_{i, j, h} & =0, && \mbox{if} \ \mrm{grad} (T^{j', h'}_{i, j, h}) < 0, \tag{t3} \label{t3} \\
T^{j',h'}_{i, j, h} & = 0, && \mbox{if} \  \mrm{grad} (T^{j', h'}_{i, j, h}) =0, (j', h') \neq (j, h+1), \tag{t4} \label{t4}\\
T^{j', h'}_{i, j, h} & = \mrm{id}_{W_j}, && \mbox{if} \  \mrm{grad} (T^{j', h'}_{i, j, h}) =0, (j', h') = (j, h+1), \tag{t5} \label{t5} \\
S^{j', h'}_{i, j, h} & = \mrm{id}_{W_j}, && \mbox{if} \  \mrm{grad} ( S^{j', h'}_{i, j, h} ) =0, (j', h') = (j, h), \tag{s5} \label{s5}\\
S^{j', h'}_{i, j, h} & =0, && \mbox{if}\ \mrm{grad} ( S^{j', h'}_{i, j, h}) =0, (j',h') \neq (j, h), \tag{s4} \label{s4} \\
S^{j', h'}_{i, j, h} & =0,  && \mbox{if} \ \mrm{grad} ( S^{j', h'}_{i, j, h}) < 0, \tag{s3} \label{s3}\\
S^{V}_{i, j, h} & = 0, && \mbox{if} \ h \neq j- i, \tag{s2} \label{s2}\\
S^{j', h'}_{i, V} & =0, &&  \tag{s1} \label{s1}  \\
[ \pi_{W'_i} \widetilde y_i \widetilde x_i|_{W'_i}  - e_i, f_i]  & =0.  \tag{r1} \label{sl2}
\end{align}

\begin{prop}[{\cite[Lemma 19]{M05}}] 
\label{Maffei-Phi}
There is an injective morphism $\Phi: \Lambda(V, W) \to \Lambda(\widetilde V, \widetilde W)$  of varieties defined by the following rules. 
For all $\bx  =( x_i, y_i, p_i, q_i) \in \Lambda(V, W)$,  
the element $\Phi(\bx)$ is the unique transversal element in $\Lambda(\widetilde V, \widetilde W)$ 
satisfies that 
\begin{align}
X_i & = x_i, & Y_i &= y_i, \\
T^{i+1, 1}_{i, V} &  = p_{i+1}, & S^V_{i, i+1,1} &= q_{i+1}.\\
%and the  transversal conditions (\ref{t1})-(\ref{t5}), (\ref{s1})-(\ref{s5}) and (\ref{sl2}).
\intertext{Moreover $T^{j, h}_{i, V} $ and $S^V_{i, j, h}$ are zero unless $j>i$, and in this case they are}
T^{j, h}_{i, V} & = \delta_{h, 1} y_{i+1} \cdots y_{j-1} p_j, 
&
S^V_{i, j, h} & = \delta_{h, j-i} q_j x_{j-1} \cdots x_{i+1}.
\end{align} 
\end{prop}

Assume the parameter $\xi$ satisfies that $\xi_i >0$ for all $i\in I$. 
The homomorphism $\Phi$ restricts to an injective homomorphism
$\Lambda^{\xi\text{-}ss}(\bv,  \bw) \to \Lambda^{\xi\text{-}ss}(\widetilde \bv, \widetilde \bw)$
which is compatible with the $\G_{\bv}$ and $\G_{\bw}$ actions on the respective varieties.
Hence it   induces closed immersions with $\zeta_{\mbb C}=0$
\begin{align}
\label{varphi}
\varphi: \M_{\zeta}(\bv, \bw) \to \M_{\zeta}(\widetilde \bv, \widetilde \bw), \quad
\varphi_0: \M_0(\bv, \bw) \to \M_0(\widetilde \bv, \widetilde \bw),
\end{align}
such that we have the following diagram.
\begin{align}
\label{diag-varphi}
\begin{CD}
\M_{\zeta}(\bv, \bw) @>\varphi>> \M_{\zeta}(\widetilde \bv, \widetilde \bw)\\
@V\pi VV @VV\pi V\\
\M_0(\bv, \bw) @>\varphi_0>> \M_0(\widetilde \bv, \widetilde \bw)
\end{CD}
\end{align}
Moreover, $\varphi_0(0) = e_0$ where $e_0$ is in  
the $\mathfrak{sl}_2$ triple $(e_0, f_0, [e_0, f_0])$ from the paragraph below (\ref{grad}).
Now set
\[
\mu = ( \widetilde \bv_0 - \widetilde \bv_1, \widetilde \bv_1 - \widetilde \bv_2, \cdots, 
\widetilde \bv_{n-1} - \widetilde \bv_n, \widetilde \bv_n).
\]
We have $\mu_i = \bw_i + \cdots + \bw_n - \bv_i + \bv_{i-1}$.
Reorder the entries in $\mu$ in decreasing order: $\rho_1\geq \rho_2 \geq \rho_3 \geq \cdots \geq \rho_{n+1}$ and set
\[
\mu' = 1^{\rho_1 -\rho_2} 2^{\rho_2-\rho_3} \cdots n^{\rho_n -\rho_{n+1}} (n+1)^{\rho_{n+1}}.
\]
Let $\mathcal O_{\mu'}$ be the nilpotent $\mrm{GL}(\widetilde W_1)$-orbit in $\mathfrak{gl}(\widetilde W_1)$
whose Jordan blocks have size $\mu'$.
It is known that the closure $\overline{\mathcal O}_{\mu'}$ of the orbit $\mathcal O_{\mu'}$ is  the image of 
the first projection $\Pi$ from the cotangent bundle $T^* \mathscr F_{\widetilde \bv, \widetilde \bw}$ to 
$\mathfrak{gl}(\widetilde W_1)$ so that via (\ref{M-TF}) we have a commutative diagram:
\begin{align}
\label{phi-comm}
\begin{CD}
\M_{\zeta}(\widetilde \bv, \widetilde \bw) @>\phi >\cong> T^*\mathscr F_{\widetilde \bv, \widetilde \bw} \\
@V\pi VV @VV \Pi V\\
\M_{1} (\widetilde \bv, \widetilde \bw) @>\phi_0 >\cong > \overline{\mathcal O }_{\mu'}
\end{CD}
\end{align}
where $\M_{1} (\widetilde \bv, \widetilde \bw)$ is the image of $\pi$.
Recall again Maffei's $\mathfrak{sl}_2$-triple $(e_0, f_0, [e_0, f_0])$ from the paragraph below (\ref{grad}).
The Slodowy transversal slice of the orbit of $e_0$ at $e_0$ is defined to be 
\[
\mathcal S_{e_0} =\{ x\in \mathfrak{gl}(\widetilde W_1) | x\ \mbox{is nilpotent}, [x - e_0, f_0] =0\}.
\]
Note that $e_0 \in \mathcal O_{\lambda} $ where $\lambda =  1^{\bw_1} 2^{\bw_2} \cdots n^{\bw_n}$.
For simplicity, we also say the trivial triple $(0, 0, 0)$ is an $\mathfrak{sl}_2$-triple, and in this case
the Slodowy slice is the whole nilpotent cone.
For convenience, we set
\begin{align}
\label{slodowy}
\mathcal S_{\mu', \lambda} = \overline{\mathcal O}_{\mu'} \cap  \mathcal S_{e_0} \quad \mbox{and}\quad
\widetilde{\mathcal S}_{\mu', \lambda} = \Pi^{-1}(\mathcal S_{\mu', \lambda}),
\end{align}
where $\Pi$ is from (\ref{phi-comm}).
The following theorem is conjectured by Nakajima ~\cite[Conjecture 8.6]{N94} and proved by Maffei ~\cite[Theorem 8]{M05}.

\begin{thm}[{\cite[Theorem 8]{M05}, \cite[Conjecture 8.6]{N94}}]
\label{Nakajima-Maffei}
The compositions $\phi \varphi$ and $\phi_0 \varphi_0$ of morphisms from (\ref{varphi}) and (\ref{phi-comm})  
yield isomorphisms  $\M_{\zeta}(\bv, \bw) \simeq \widetilde{\mathcal S}_{\mu', \lambda}$ 
and  
$\M_{1}(\bv, \bw) \simeq \mathcal S_{\mu', \lambda}$, respectively.
In particular, we have the following commutative diagram
\begin{align}
\label{phi-varphi}
\begin{CD}
\M_{\zeta}(\bv, \bw) @>\phi\varphi >> 
\widetilde{\mathcal S}_{\mu', \lambda} \\
@V\pi VV @VV\Pi V\\
\M_1(\bv, \bw) @>\phi_0 \varphi_0 >> \mathcal S_{\mu', \lambda}
\end{CD}
\end{align}
\end{thm}

In what follows, we discuss two applications of the above remarkable theorem, which could have been stated in~\cite{M05} 
and has been discussed in~\cite[Section 9]{H15}.
We present the two applications in the following as a preparation  to the analogous results in classical groups and their associated symmetric spaces.

\subsection{Rectangular symmetry}
\label{rect}

In this section, we deduce a rectangular symmetry from Theorem~\ref{Nakajima-Maffei}.
If we relabel the vertex set $I$ by $i \to \theta(i)$, where $\theta(i)= n+1-i$, 
we can repeat the process in Section~\ref{Maffei} again. In particular, we obtain an immersion
\[
\M_{\zeta}(\bv, \bw) \to T^*\mathscr F_{\widehat \bv, \widehat \bw},
\]
where the pair $(\widehat \bv, \widehat \bw)$ 
is given by
\begin{align}
\label{hat-bw}
\widehat \bw_i = \delta_{i, 1} \sum_{1\leq j \leq n} (n+1- j) \bw_{j}, \quad  \widehat \bv_i 
= \bv_{n+1-i} + \sum_{j: j \geq  i+1} (j - i ) \bw_{n+1-j}, \quad \forall 1\leq i\leq n.
\end{align}
Similar to the $\mathfrak{sl}_2$-triple $(e_0, f_0, [e_0, f_0])$, 
we have an $\mathfrak{sl}_2$-triple $(\widehat e_0, \widehat f_0, [\widehat e_0, \widehat f_0])$
where the  nilpotent element $\widehat e_0$ has the  Jordan type 
$ 1^{\bw_{\theta(1)}} 2^{\bw_{\theta(2)}}\cdots i^{\bw_{\theta(i)}} \cdots n^{\bw_{\theta(n)}}$.
Similar to the partition $\mu$, we also have a partition $\widehat \mu \equiv (\widehat \mu_i)_{1\leq i\leq n+1}
=(\widehat \bv_{i-1} - \widehat \bv_{i})_{1\leq i \leq n+1}$ 
determined by $\widehat \bv$ 
and let $\widehat \mu'$ be its transpose.
Observe that 
\[
\widehat \bv_{n-i+1} - \widehat \bv_{n-i+2} = \bw_1+\cdots + \bw_{i-1} + \bv_i - \bv_{i-1}.
\]
This implies that 
\begin{align}
\label{mu-hat}
\mu_i + \widehat \mu_{n-i+2}  = \bw_1+\bw_2 + \cdots + \bw_n, \quad \forall 1\leq i\leq n+1.
\end{align}
In particular, if the transpose $\mu'$ is $\mu'= 1^{\mu'_1} 2^{\mu'_2} \cdots (n+1)^{\mu'_{n+1}}$, then the transpose $\widehat \mu'$ of $\widehat \mu$ is
\begin{align}
\label{mu'-hat}
\widehat \mu' = 1^{\mu'_n} 2^{\mu'_{n-1}} \cdots n^{\mu'_1} (n+1)^{\mu'_0} 
=(i^{\mu'_{\theta(i)}})_{1\leq i\leq n+1}, \quad \mu'_0 = \sum_{1\leq i\leq n} \bw_i - \sum_{1\leq i\leq n+1} \mu'_i .
\end{align}
Hence we have a similar result as Theorem ~\ref{Nakajima-Maffei} in describing the new immersion via the intersection
$\mathcal S_{\widehat \mu', \widehat{\lambda}} =
\overline{\mathcal O}_{\widehat \mu'} \cap \mathcal S_{\widehat e_0}$ and its 
partial Springer resolution 
$\widetilde{\mathcal S}_{\widehat \mu', \widehat{\lambda}}:=\Pi_{\widehat \bv, \widehat \bw}^{-1}(\mathcal S_{\widehat e_0})$
where $\Pi_{\widehat \bv, \widehat \bw}$ is the  natural projection similar to $\Pi$.
Altogether, we have the following result.

\begin{prop}
\label{prop:rect}
Let $(\widetilde \bv, \widetilde \bw)$ and $(\widehat \bv, \widehat \bw)$ be the pairs defined by
(\ref{tild-bw}) and (\ref{hat-bw}) such that associated compositions $\mu$ and $\widehat \mu$ satisfy (\ref{mu-hat}) (see also (\ref{mu'-hat})).
Then the following diagram is commutative with  isomorphic horizontal maps, 
which sends $e_0$ of Jordan type $(i^{\bw_i})_{1\leq i\leq n}$ to $\widehat e_0$  of Jordan type $(i^{\bw_{\theta(i)}})_{1\leq i\leq n}$ in the base.  
\begin{align}
\begin{CD}
\widetilde{\mathcal S}_{\mu', \lambda}
@>\cong>> 
\widetilde{\mathcal S}_{\widehat \mu', \widehat \lambda}\\
@V \Pi_{\widetilde \bv, \widetilde \bw} VV @VV\Pi_{\widehat \bv, \widehat \bw}V\\
\mathcal S_{\mu', \lambda}
@>\cong >>
\mathcal S_{\widehat \mu', \widehat \lambda}
\end{CD}
\end{align}
\end{prop}

In light of (\ref{mu'-hat}), the partitions $\mu'$ and $\widehat \mu'$ fit into a rectangle of size $(n+1)\times \sum_{1\leq i\leq n} \bw_i$. Similarly, the Jordan types $\lambda=(1^{\bw_i})_{1\leq i\leq n}$ and $\widehat \lambda= (1^{\bw_{\theta(i)}})_{1\leq i\leq n}$ of $e_0$ and $\widehat e_0$ fit into a rectangle of the same size, depicted in the following.
%(\ref{rect-sym}). 
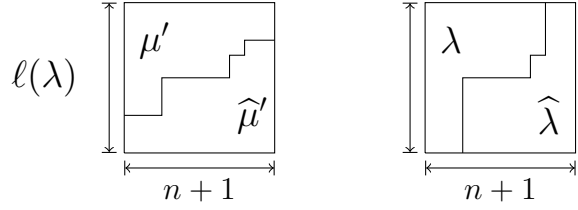
\begin{figure}
\begin{tikzpicture}
\tikzset{rec/.style={thick,text width=2cm,font=\bfseries\large, align=left, text opacity=1},
line/.style={dashed, line width=2pt},
sharea/.style={shade, left color=gray!50, right color=white} % B/W shading
}
\draw (0,0) -- (2,0) -- (2,2) -- (0,2) -- (0,0);

\draw (0, 0.5)  -- (0.5, 0.5) -- (0.5, 1) -- (1.4, 1) -- (1.4, 1.3) -- (1.6, 1.3) -- (1.6, 1.5) -- (2, 1.5);

\node[rec] at (1.2, 1.5)  {$\mu'$};
\node[rec] at (2.5, 0.5) {$\widehat \mu'$};

\draw[|<-> |] (0, -0.2) -- (2, -0.2) node[pos=.5,sloped,below] {$n+1$};

\draw[|<-> |] (-0.2, 0) -- (-0.2, 2);

\node[rec] at (-0.5, 1) {$\ell(\lambda)$};
\draw (4,0) -- (6,0) -- (6,2) -- (4,2) -- (4,0);

\draw (4.5, 0)  -- (4.5, 0.5) -- (4.5, 1) -- (5.4, 1) -- (5.4, 1.3) -- (5.6, 1.3) -- (5.6, 1.5) -- (5.6, 2);

\node[rec] at (5.2, 1.5)  {$\lambda$};
\node[rec] at (6.5, 0.5) {$\widehat \lambda$};

\draw[|<-> |] (4, -0.2) -- (6, -0.2) node[pos=.5,sloped,below] {$n+1$};

\draw[|<-> |] (3.8, 0) -- (3.8, 2);

%\node[rec] at (3.7, 1) {$\ell(\lambda)$};
\end{tikzpicture}
\caption{Rectangular Symmetry}
\label{rect-sym}
\end{figure} 
This explains the name of the section.

\begin{rem}
Proposition~\ref{prop:rect} yields an identity on the Kostka numbers:
$$
K_{\lambda, \mu'} = 
K_{\widehat \lambda, \widehat \mu'},
\quad \mbox{with}\ \lambda = (i^{\bw_i})_{1\leq i\leq n},
\quad \widehat \lambda = (i^{\bw_{\theta(i)}})_{1\leq i\leq n}
$$
first proved by Briand-Orellana-Rosas in~\cite{BOR15}.
\end{rem}

\subsection{Column/Row removal reductions}
\label{col}

Now we discuss the second application of the Theorem~\ref{Nakajima-Maffei}.
It is clearly that the quiver varieties $\M_{\zeta}(\bv, \bw)$ is isomorphic to  the quiver varieties
$\M_{\zeta}(\breve \bv, \breve \bw)'$ of Dynkin diagram $A_{n+1}$ with the dimension vectors $\breve \bv, \breve \bw$ given by
\[
\breve \bv_0=0,  \breve \bv_i=\bv_{i-1}; \breve \bw_0=0, \breve \bw_i=\bw_{i-1}, \quad \forall 2\leq i\leq n+1. 
\]
By Theorem~\ref{Nakajima-Maffei}, we have
$
\M_{\zeta}(\breve \bv, \breve \bw)' \cong  \widetilde{\mathcal S}_{\breve \mu', \breve\lambda},
$
where 
$\breve \mu$ has an extra entry $\sum_{1\leq i\leq n} \bw_i$ than $\mu$ and 
$\breve \lambda= ((i+1)^{\bw_i})_{1\leq i\leq n}$.
It yields 

\begin{prop}
\label{prop:column}
The  following diagram is commutative with  isomorphic horizontal maps, 
which sends $e_0$ of Jordan type $(i^{\bw_i})_{1\leq i\leq n}$ to $\breve e_0$  of Jordan type $((i+1)^{\bw_{i}})_{1\leq i\leq n}$ in the base.  
\begin{align}
\begin{CD}
\widetilde{\mathcal S}_{\mu', \lambda}
@>\cong>> 
\widetilde{\mathcal S}_{\breve \mu', \breve\lambda}\\
@VVV @VVV\\
\mathcal S_{\mu', \lambda}
@>\cong >>
\mathcal S_{\breve \mu', \breve \lambda}
\end{CD}
\end{align}
\end{prop}

If we write the partitions involved as Young diagrams, then the partition $\lambda$ can be  obtained 
from $\breve \lambda$ by removing the left-most column of $\breve \lambda$ via Figure~\ref{CRR}.
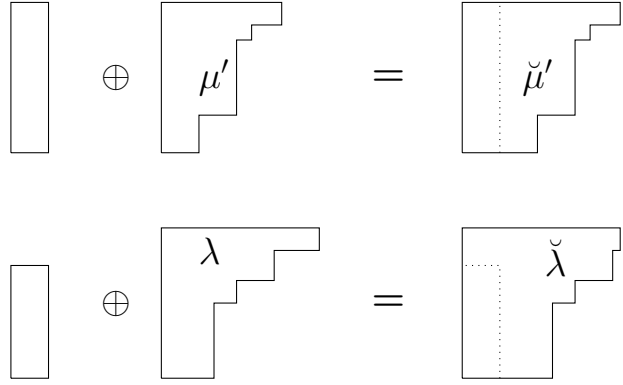
\begin{figure}
\begin{tikzpicture}
\tikzset{rec/.style={thick,text width=2cm,font=\bfseries\large,align=left, text opacity=1},
line/.style={dashed, line width=2pt},
sharea/.style={shade, left color=gray!50, right color=white} % B/W shading
}
\draw (0, 3) -- (0, 5) -- (.5, 5) -- (.5,3)--(0,3);

\node[rec] at (2.2, 4)  {$\oplus$};

\draw (2, 3) -- (2, 5) -- (3.6, 5) -- (3.6, 4.7) -- (3.2, 4.7) -- (3.2, 4.5) -- (3, 4.5) -- (3, 3.5) -- (2.5, 3.5) -- (2.5, 3) -- (2,3);

\node[rec] at (3.5, 4) {$\mu'$};

\node[rec] at (5.8,  4) {=};

\draw (6, 3) -- (6, 5) -- (8.1, 5) -- (8.1, 4.7) -- (7.7, 4.7) -- (7.7, 4.5) -- (7.5, 4.5) -- (7.5, 3.5) -- (7, 3.5) -- (7, 3) -- (6,3);

%\draw (6, 3)  -- (6, 5) -- (7.5, 5) -- (7.5, 4.7) -- (7.7, 4.7) -- (7.7, 4.5) -- (8.1, 4.5) -- (8.1, 3) -- (6, 3);

\draw [dotted] (6.5, 3) --  (6.5, 5);

\node[rec] at (7.8, 4) {$\breve \mu'$};

\draw (0, 0) -- (0, 1.5) -- (.5, 1.5) -- (.5,0)--(0,0);

\node[rec] at (2.2, 1)  {$\oplus$};

\draw (2,0) -- (2, 2) -- (4.1, 2) -- (4.1, 1.7) -- (3.5, 1.7) -- (3.5, 1.3) -- (3, 1.3) -- (3, 1) -- (2.7, 1) -- (2.7, 0) -- (2, 0);

\node[rec] at (3.5, 1.7) {$\lambda$};

\node[rec] at (5.8,  1) {=};

\draw (6, 0)  -- (6, 2) -- (8.1, 2) -- (8.1, 1.7) -- (8, 1.7) -- (8, 1.3) -- (7.5, 1.3) -- (7.5, 1) --(7.2, 1)-- (7.2,0) -- (6, 0);

\draw [dotted] (6.5, 0) --  (6.5, 1.5) -- (6, 1.5);

\node[rec] at (8.1, 1.6) {$\breve \lambda$};

%\node[rec] at (5.8, 2) {$\sum_{1\leq i\leq n} \bw_i$};
\end{tikzpicture}
\caption{Column-Removal Reduction}
\label{CRR}
\end{figure}
The partition $\mu'$ can be obtained from $\breve \mu'$ by removing part of the left-most column $\breve \mu'$.
Thus, Proposition~\ref{prop:column} is  a geometric version of Kraft-Procesi's column removal reduction in~\cite[Proposition 5.4]{KP81}.

Now we turn to provide a geometric version of Kraft-Procesi's row-removal reduction in ~\cite[Proposition 4.4]{KP81}.
This is obtained in exactly the same spirit as that of column-removal reduction. 
Precisely, we can identify the variety the variety $\M_{\zeta}(\bv, \bw)$ with another quiver variety $\M_{\zeta}(\ddot{\bv}, \ddot{\bw})'$ of type $A_{n+1}$
where the vectors $\ddot{\bv}$ and $\ddot{\bw}$ are given as follows.
\[
\ddot{\bv}_i = \bv_i, \ddot{\bv}_{n+1} =0, \ddot{\bw}_i = \bw_i, \ddot{\bw}_{n+1}=a, \quad \forall 1\leq i\leq n.
\]
If the associated pair of partitions to $\bw, \bw$ is $(\mu'=(i^{\mu'_i})_{1\leq i\leq n+1}, \lambda=(1^{\bw_i})_{1\leq i\leq n}$, then
the similar one for $\ddot{\bv}, \ddot{\bw}$ will be given as
\begin{align}
\label{row-addition}
\ddot{\mu}' = 1^{\mu'_1} \cdots n^{\mu'_n} (n+1)^{\mu'_{n+1} +a}, \ddot\lambda= 1^{\bw_1}\cdots n^{\bw_n} (n+1)^a. 
\end{align}
(Note that $\widetilde{\ddot\bv}_i = \bv_i + \sum_{j>i} (j-i) \bw_i + (n+1-i) a$.)
Therefore, by Theorem~\ref{Nakajima-Maffei} and the identification $\M_{\zeta}(\bv, \bw) \cong \M_{\zeta}(\ddot \bv, \ddot \bw)'$, we obtain

\begin{prop}
\label{prop:row}
The following diagram is commutative with  isomorphic horizontal maps, 
which sends $e_0$ of Jordan type $(i^{\bw_i})_{1\leq i\leq n}$ to $\ddot e_0$  of  type $((1^{\bw_1} \cdots n^{\bw_n} (n+1)^a)$ in the base.  
\begin{align}
\begin{CD}
\widetilde{\mathcal S}_{\mu', \lambda}
@>\cong>> 
\widetilde{\mathcal S}_{\ddot \mu', \ddot\lambda}\\
@VVV @VVV\\
\mathcal S_{\mu', \lambda}
@>\cong >>
\mathcal S_{\ddot \mu', \ddot \lambda}
\end{CD}
\end{align}
\end{prop}

Note that the partitions $\mu'$ and $\lambda$ can be obtained from $\ddot\mu'$ and $\ddot \lambda$ by removing the respective first $a$-th rows in Figure~\ref{RRR}. 
\begin{figure}
\begin{tikzpicture}
\tikzset{rec/.style={thick,text width=2cm,font=\bfseries\large,align=left, text opacity=1},
line/.style={dashed, line width=2pt},
sharea/.style={shade, left color=gray!50, right color=white} % B/W shading
}
\draw (0.5, 4) -- (0.5, 4.2) -- (-1.5, 4.2) -- (-1.5,4)--(0.5,4);

\node[rec] at (2.2, 4)  {$\oplus$};

\draw (2, 3) -- (2, 5) -- (3.6, 5) -- (3.6, 4.7) -- (3.2, 4.7) -- (3.2, 4.5) -- (3, 4.5) -- (3, 3.5) -- (2.5, 3.5) -- (2.5, 3) -- (2,3);

\node[rec] at (3.5, 4) {$\mu'$};

\node[rec] at (5.8,  4) {=};

\draw (6, 3) -- (6, 5) -- (8.1, 5) -- (8.1, 4.7) -- (7.7, 4.7) -- (7.7, 4.5) -- (7.5, 4.5) -- (7.5, 3.5) -- (7, 3.5) -- (7, 3) -- (6,3);

\draw (6, 5) -- (6, 5.2) -- (8.1, 5.2) -- (8.1, 5);

%\draw [dotted] (6.5, 3) --  (6.5, 5);

\node[rec] at (7.8, 4) {$\ddot \mu'$};

\draw (0.5, 1) -- (0.5, 1.2) -- (-1.5, 1.2) -- (-1.5,1)--(0.5,1);

\node[rec] at (2.2, 1)  {$\oplus$};

\draw (2,0) -- (2, 2) -- (4.1, 2) -- (4.1, 1.7) -- (3.5, 1.7) -- (3.5, 1.3) -- (3, 1.3) -- (3, 1) -- (2.7, 1) -- (2.7, 0) -- (2, 0);

\node[rec] at (3.5, 1.7) {$\lambda$};

\node[rec] at (6.2,  1) {=};

\draw (6, 0)  -- (6, 2) -- (8.1, 2) -- (8.1, 1.7) -- (7.5, 1.7) -- (7.5, 1.3) -- (7, 1.3) -- (7, 1) --(6.7, 1)-- (6.7,0) -- (6, 0);

\draw (6, 2) -- (6, 2.2) -- (8.1, 2.2)--(8.1, 1.7);

\node[rec] at (7.5, 1.7) {$\ddot \lambda$};

%\node[rec] at (5.8, 2) {$\sum_{1\leq i\leq n} \bw_i$};
\end{tikzpicture}
\caption{Row-Removal Reduction}
\label{RRR}
\end{figure}
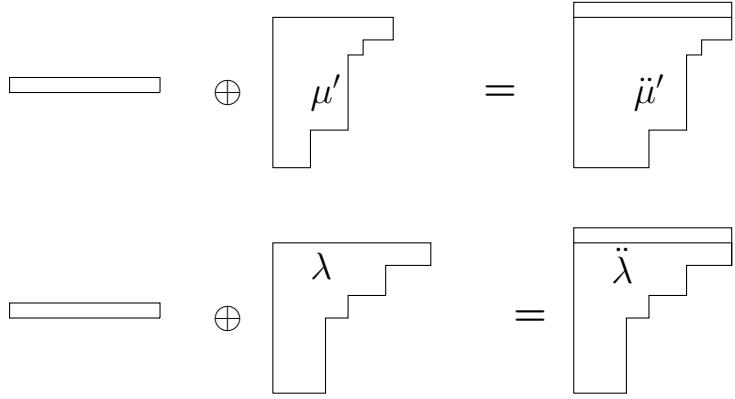
So Proposition~\ref{prop:row} is a geometric version of Kraft-Procesi's row-removal reduction. 
Of course, by combining Propositions~\ref{prop:column} and~\ref{prop:row} we obtain a geometric version of the general reduction in 
~\cite[Proposition 3.1]{KP81}, which plays a critical role in the study of minimal singularities in $\GL_n$. 

\section{Example II: partial resolutions of  nilpotent Slodowy slices}

This section is devoted to the compatibility of Maffei's morphism and the isomorphism $\sigma$.
From which, we see that the $\sigma$-quiver varieties encompass partial Springer resolutions of
nilpotent Slodowy slices of classical groups. 
As a consequence, we deduce the rectangular symmetry for classical groups. 
The symmetry provides a natural home for the recent results of~\cite{HL14, W15} 
on the interactions of two-row Slodowy slices
of symplectic and orthogonal groups.
We also briefly discuss a geometric version of  
Kraft-Procesi's column-removal and row-removal reductions for classical groups in~\cite[Proposition 13.5]{KP82}.

We again assume that $\w=w_0$ and $a=1$. 

\subsection{Maffei's morphism and the bilinear forms on $\widetilde V$ and $\widetilde W$}

Assume that $V$ and $W$ are formed spaces  with signs  $\tilde \delta_{\bv}$ and  $\delta_{\bw}$, respectively.
%The most crucial ingredient in the connection is the following bilinear form on $\widetilde V$.
Recall that  we set $V_0=0$ and $\widetilde V_0 = \widetilde W_1$. 
We define a non-degenerate bilinear form  on $\widetilde V_i$, for all  $0\leq i\leq n$, by
\begin{align}
\label{tild-V-form}
\left \langle (v_i, w_j^{(h)})_{j \geq i+ h}  
\middle|\:  (v'_i, u_j^{(h)})_{j \geq i + h} \right \rangle_{\widetilde V_i}
=\left (v_i, v'_i \right )_{V_i} + \sum_{j \geq i + h} (-1)^{j-i+h} \left (w_j^{(h)}, u_j^{(j-i+1-h)} \right )_{W_j},
\end{align}
where $ v_i, v_i' \in V_i$ and  $w_j^{(h)}, u_j^{(h)} \in W_j^{(h)}$ such that $1\leq h \leq j-i$.
The form $\langle - |- \rangle_{\widetilde V_i}$ on $\widetilde V_i$ may not be a $\delta$-form.
However, if the signs $\tilde \delta_{\bv}$ and $\delta_{\bw}$ alternate, it turns out to be  the case.

\begin{lem}
\label{tV-alternate}
If the sign $\delta_{\bw}$ is $\Gamma$-alternating as in Proposition~\ref{order-4}, then
$\widetilde W_1$ is a $(-1)^{i+1}\delta_{\bw,i}$-form (for each $i$).
If further $\tilde \delta_{\bv}$ is $\Gamma$-alternating and $\tilde \delta_{\bv,i} \delta_{\bw,i}=-1$ for all $i$, then 
the form on $\widetilde V_i$ is a  $\tilde \delta_{\bv,i}$-form for all $1\leq i\leq n$.
\end{lem}

\begin{proof}
We only need to observe that for a fixed $j_0$ such that  $j_0-i$ is even (resp. odd), 
then the restriction of $\langle - | - \rangle_{\widetilde V_i}$ to $\oplus_{1\leq h\leq j_0-i} W_{j_0}^{(h)}$ is a 
$-\delta_{\bw, j_0}$-form (resp. $\delta_{\bw,j_0}$-form).
Specifically, we have
$$
\langle (u_{j_0}^{(h)})_{1\leq h\leq j_0-i} | (w_{j_0}^{(h)})_{1\leq h\leq j_0 -i} \rangle_{\widetilde V_i} = (-1)^{j_0 -i +1} \delta_{\bw,j_0}
\langle (w_{j_0}^{(h)})_{1\leq h\leq j_0 -i} | (u_{j_0}^{(h)})_{1\leq h\leq j_0-i} \rangle_{\widetilde V_i},
$$
for all elements $ (w_{j_0}^{(h)})_{1\leq h\leq j_0 -i}$, $(u_{j_0}^{(h)})_{1\leq h\leq j_0-i}$ in $\oplus_{1\leq h\leq j_0-i} W_{j_0}^{(h)}$.
\end{proof}

Let $T^{\natural}: \widetilde V_j \to \widetilde V_i$ denote the right adjoint of 
the linear map $T: \widetilde V_i \to\widetilde V_j$, 
with respect to the forms on $\widetilde V_i$ and $\widetilde V_j$.
In particular, the right adjoint of  a point in $\Lambda(\widetilde V, \widetilde W)$ represented as the collection of linear maps in (\ref{xi-dec}) and (\ref{yi-dec}) satisfies the following.
\begin{align}
X_i^{\natural} & = X_i^*, & Y_i^{\natural} & = Y_i^*, \nonumber \\
(T^{i+1, 1}_{i, V})^{\natural} &  =(T^{i+1, 1}_{i, V})^*, & (S^{V}_{i, i+1, 1})^{\natural} & = (S^{V}_{i, i+1, 1})^*,  \label{x-natural}\\
(T^{j', h'}_{i, j, h})^{\natural} & = (-1)^{j-j'+ h - h' - 1} (T^{j', h'}_{i, j, h})^*, &
(S^{j', h'}_{i, j,h})^{\natural} & = (-1)^{j-j' + h - h' +1} (S^{j', h'}_{i, j, h})^*. \nonumber
\end{align}

Indeed, the identities in the first two rows are due to the fact that the signs of the forms on $V_i$ and $W_{i+1}^{(1)}$ at  vertex $i$ remain unchanged.
For simplicity, we write $t = T^{j', h'}_{i, j, h}$ and then
$$
(-1)^{j-(i+1)+h}(t(u), v)_{W_j} = \langle t(u) | v\rangle_{\widetilde V_{i+1}} = \langle u | t^{\natural} (v) \rangle_{\widetilde V_i} 
=(-1)^{j' - i + h'} (u, t^{\natural}(v))_{W_{j'}},
$$
for all $u\in W_{j'}^{(h')}$ and $v\in W_{j}^{(h)}$.
So we have $t^{\natural} = (-1)^{j-j' + h-h'-1} t^*$, which is the first identity in the last row of (\ref{x-natural}). 
Identities in the second column of (\ref{x-natural})  are proved in a similar way.

By the definition (\ref{tau-raw}), we have an automorphism
$\widetilde \tau: \Lambda(\widetilde V, \widetilde W) \to \Lambda(\widetilde V, \widetilde W)$ 
with respect to the forms on $\widetilde V$ and $\widetilde W$.
We have the following compatibility result of Maffei's morphism and the automorphisms $\tau$ in (\ref{tau-raw})  
and its analog $\widetilde \tau$ on $\Lambda(\widetilde V, \widetilde W)$.

\begin{prop}
\label{tau-Phi}
The following diagram commutes.
\[
\begin{CD}
\Lambda(V, W) @>\Phi>> \Lambda(\widetilde V, \widetilde W) \\
@V\tau VV @VV \widetilde \tau V \\
\Lambda(V, W) @>\Phi >> \Lambda(\widetilde V, \widetilde W)
\end{CD}
\]
\end{prop}

\begin{proof}
Fix a point $\bx =(x_i, y_i, p_i, q_i) \in \Lambda(V, W)$. 
By definition, the point $\Phi \tau(\bx)$ is a point in $\Lambda(\widetilde V, \tilde W)$ determined by 
the transversal conditions (\ref{s1})-(\ref{s5}), (\ref{t1})-(\ref{t5}), (\ref{sl2}) and the following.
\begin{align}
\label{x-dual}
X_i & = - y_i^*, & Y_i & = x_i^* ,\\
T^{i+1, 1}_{i, V} & = - q^*_{i+1}, & S^{V}_{i, i+1, 1} & = p^*_{i+1}.
\end{align}
So it suffices to show that the point $\widetilde \tau \Phi(\bx)$ satisfies the  transversal conditions and the above. 
 
We put a superscript $\widetilde \tau$ on the upper left for the decomposition (\ref{xi-dec})-(\ref{yi-dec}) of $\widetilde \tau \Phi(\bx)$ 
with respect to the fixed decompositions of $\widetilde V$ and $\widetilde W$.
By  (\ref{x-natural}),  the point $\widetilde \tau \Phi(\bx)$ satisfies 
\begin{align}
^{\widetilde \tau} X_i & = - Y_i^{\natural} = - y^*_i , & ^{\widetilde \tau} Y_i & = X_i^{\natural} = x_i^*, \\
^{\widetilde \tau} T^{i+1, 1}_{i, V} & = - (S^V_{i, i+1, 1})^{\natural}  = - q_{i+1}^*, 
& ^{\widetilde \tau} S^V_{i, i+1, 1} & = (T^{i+1,1}_{i, V})^{\natural} = p_{i+1}^*.
\end{align}
Hence it remains to show that the point $\widetilde \tau \Phi(\bx)$ satisfies the transversal conditions.
The conditions (ti) and (si) for $1\leq i\leq 5$ for $\widetilde \tau \Phi(\bx)$ 
follow from the conditions (si) and (ti) for $\Phi(\bx)$, respectively.
More precisely, for (\ref{t1}), we have
\[
^{\widetilde \tau} T^V_{i, j, h} = - (S_{i, V}^{j, j-i+1 -h})^{\natural} =0.
\] 
For (\ref{t2}), we notice that 
$h' \neq 1$ if and only if $(j' - i + 1 - h') \neq j' -i$. 
So by the (\ref{s2}) of $\Phi(\bx)$,  
\[
^{\widetilde \tau} T^{j', h'}_{i, V} = - (S^V_{i, j', j'-i+1-h'})^{\natural} =0.
\]
This shows that $\widetilde \tau \Phi(\bx)$ satisfies the condition (\ref{t2}).
For (\ref{t3}), we observe that
\[
\mrm{grad} ( T^{j', h'}_{i, j, h} ) = \mrm{grad} (S^{j, j - (i+1) + 1 - h}_{i, j', j'-i +1 - h'}),
\]
and $(j', h') \neq (j, h+1)$ if and only if $(j, j-i-h) \neq (j', j'-i+1-h')$.
Thanks to these observations and the (\ref{s3})-(\ref{s4}) of $\Phi (\bx)$, it leads to
\begin{align*}
\begin{split}
& ^{\widetilde \tau} T^{j', h'}_{i, j, h} = - (S^{j, j- (i+1) +1 - h}_{i, j', j'-i+1-h'})^{\natural} =0, 
%\quad \mbox{if}  \ \mrm{grad} (\ ^{\widetilde \tau} T^{j', h'}_{i, j, h} ) <0,\\
 %
%& ^{\widetilde \tau} T^{j', h'}_{i, j, h} =- (S^{j, j- (i+1) +1 - h}_{i, j', j'-i+1-h'})^{\natural} =0,
%\quad \mbox{if} \ \mrm{grad} (\ ^{\widetilde \tau} T^{j', h'}_{i, j, h} ) = 0, (j', h') \neq (j, h+1).
\end{split}
\end{align*}
if either $\mrm{grad} (\ \! ^{\widetilde \tau} T^{j', h'}_{i, j, h} ) <0$ or 
$\mrm{grad} (\ \! ^{\widetilde \tau} T^{j', h'}_{i, j, h} ) = 0$ and  $(j', h') \neq (j, h+1)$.
If $\mrm{grad} (\ \! ^{\widetilde \tau} T^{j', h'}_{i, j, h} ) = 0$ and $(j, j-i-h) = (j', j'-i+1-h')$, then
\[
(S^{j, j- (i+1) +1 - h}_{i, j', j'-i+1-h'})^{\natural} 
=-  (S^{j, j- (i+1) +1 - h}_{i, j', j'-i+1-h'})^* = - (\mrm{id}_{W_j})^* = - \mrm{id}_{W_j},
\]
where the second equality is from the (\ref{s5}) of $\Phi(\bx)$. 
So   we have
\[
^{\widetilde \tau} T^{j', h'}_{i, j, h} =  \mrm{id}_{W_j}, 
\quad
\mbox{if} \ \mrm{grad} (\ ^{\widetilde \tau} T^{j', h'}_{i, j, h} ) = 0, (j', h') = (j, h+1).
\]
By now, it has been shown that the point $\widetilde \tau \Phi(\bx)$ satisfies the conditions (\ref{t1})-(\ref{t5}).
By an entirely similar argument, we can check that the point
$\widetilde \tau \Phi(\bx)$ satisfies the conditions (\ref{s1})-(\ref{s5}), once we observe from (\ref{x-natural}) that 
$(T^{j, j-i +1-h}_{i, j, j-i-h})^{\natural} = (T^{j, j-i +1-h}_{i, j, j-i-h})^*$ for (\ref{s5}). 

Recall that the automorphism $\widetilde \tau$ on $\mathfrak{sl}(W_i')$, compatible with $\widetilde \tau$  
on $\Lambda(\widetilde \bv, \widetilde \bw)$,
is given by $x \mapsto - x^{\natural}$.
By a similar analysis as above, one can check that $^{\widetilde \tau} e_i = e_i$
and $^{\widetilde \tau} f_i = f_i$. This implies that the point $\widetilde \tau \Phi(\bx)$ 
satisfies the last condition (\ref{sl2}), completing the proof.
\end{proof}

By Proposition ~\ref{tau-Phi}, (\ref{tau-comm}) and (\ref{varphi}), we have the following.

\begin{prop}
\label{tau-varphi}
Assume that $\zeta=(\xi, 0)$ with $\xi$ satisfying that $\xi_i >0$ for all $i\in I$.
Then the automorphism $\tau_{\zeta}$ in (\ref{tau}) and the embeddings  $\varphi$ and $\varphi_0$ in (\ref{varphi}) are compatible. Specifically, 
we have the following commutative diagram.
\begin{align}
\begin{split}
%\xymatrixrowsep{.1in}
%\xymatrixcolsep{.1in}
\xymatrix{
& \M_{-\zeta} (\bv, \bw)  \ar@{->}[rr]^{\varphi}  \ar@{.>}[dd]^(.3){\pi}
&&  \M_{-\zeta} (\widetilde \bv, \widetilde \bw) 
\ar@{->}[dd]^{\widetilde \pi} \\
\M_{\zeta}(\bv, \bw) \ar@{->}[ur]^{\tau_{\zeta} }  \ar@{->}[rr]^(.7){\varphi}  \ar@{->}[dd]^{\pi}
& & \M_{\zeta} (\widetilde \bv, \widetilde \bw) \ar@{->}[ur]^{\widetilde \tau_{\zeta}}  \ar@{->}[dd]^(.3){\widetilde \pi} & \\
& \M_{0}(\bv, \bw) \ar@{.>}[rr]^(.3){\varphi_0}   &&\M_0(\widetilde \bv, \widetilde \bw)  \\
\M_{0}(\bv, \bw) \ar@{.>}[ur]^{\tau_0} \ar@{->}[rr]^{\varphi_0}  &&\M_0(\widetilde \bv, \widetilde \bw) \ar@{->}[ur]^{\widetilde \tau_0} &
}
\end{split}
\end{align}
\end{prop}

\subsection{Maffei's morphism and reflection functors}

In this section, we show the compatibility of Maffei's morphism and the reflection functors on quiver varieties 
in Section ~\ref{Weyl-action}.

Recall that to a pair $(\bv, \bw)$ and a fixed vertex $i\in I$, we associate a new pair $(\bv', \bw) = (s_i *_{\bw} \bv, \bw)$ in
Section ~\ref{Weyl-action}.
To the same pair, we attach a third  pair $(\widetilde \bv, \widetilde \bw)$ in Section ~\ref{Maffei}.
Now apply the operation in (\ref{Weyl-action}) to $(\widetilde \bv, \widetilde \bw)$, we have 
$( (\widetilde \bv)', \widetilde \bw) = (s_i *_{\widetilde \bw} \widetilde \bv, \widetilde \bw)$, while applying the procedure in
Section ~\ref{Maffei} yields $(\widetilde{\bv}', \widetilde \bw)$.
We now compare $(\widetilde \bv)'$ and $\widetilde \bv'$.
If $j\neq i$, then $(\widetilde \bv)'_j = (\widetilde \bv)_j = (\widetilde \bv')_j$.
With the convention $\widetilde \bv_0 = \widetilde \bw_1$, the two vectors 
$(\widetilde \bv)'$ and $\widetilde \bv'$ coincide at $j=i$ by the following computation.
\[
(\widetilde \bv)'_i = \widetilde \bv_{i-1} + \widetilde \bv_{i+1} - \widetilde \bv_i
= \bv'_i + \sum_{j\geq i+1} (j-i) \bw_j = \widetilde \bv'_i.
\]
As a result, we have
$(\widetilde \bv)' = \widetilde \bv'$.

Similar to Section ~\ref{Weyl-action}, we fix a triple $(V, V', W)$ of dimension vector $(\bv, \bv', \bw)$ and $V_j= V_j' $ for all $j\neq i$.
We define $\widetilde V, \widetilde W$ and $\widetilde V'$ as in Section ~\ref{Maffei}. 
In particular, $\widetilde V'_j = \widetilde V_j = V_j \oplus \oplus_{k\geq j+1, 1\leq h\leq k - j} W^{(h)}_{k}$ if $j \neq i$ and
$\widetilde V'_i = V'_i \oplus \oplus_{k\geq i+1, 1\leq k\leq k-i} W^{(h)}_{k}$.

Recall the diagram (\ref{reflection-raw}) from Section ~\ref{Weyl-action}.
We write $\Lambda^{\xi\text{-}ss} (V, W)$ for $\Lambda^{\xi\text{-}ss}_{\zeta_{\mbb C}}(V, W)$ when $\zeta_{\mbb C}=0$.
With the above preparation, we are ready to state  the compatibility of Maffei's morphism $\Phi$ in Proposition ~\ref{Maffei-Phi} with 
the reflection functor diagram (\ref{reflection-raw}).

\begin{prop}
\label{Phi-Si}
Fix $i \in I$.
Assume that
$\zeta=(\xi, 0)$ and there is $w\in W$ such that $w(\xi)_j >0$ for all $j \in I$ 
(equivalently here is $w\in W$ such that  $w(\xi)_j < 0$, $\forall j\in I$).
The following diagram commutes.
\[
\begin{CD}
\Lambda^{\xi\text{-}ss}(V, W) @<<< F^{ss}(V, V', W) @>>>  \Lambda^{s_i(\xi)\text{-}ss}(V', W) \\
@V\Phi VV @V\Phi_FVV @VV\Phi' V \\
\Lambda^{\xi\text{-}ss}(\widetilde V, \widetilde W) @<<< F^{ss}(\widetilde V, \widetilde V', \widetilde W) @>>>  
\Lambda^{s_i(\xi)\text{-}ss}(\widetilde V', \widetilde  W)
\end{CD}
\]
where the rows are (\ref{reflection-raw}) attached to the triples $(V, V', W)$ and $(\widetilde V, \widetilde V', \widetilde W)$,
the vertical morphisms $\Phi$ and $\Phi'$ on the left and right are the restriction of the morphism  
from Proposition \ref{Maffei-Phi} and the middle one is defined to be
$\Phi_F (\bx, \bx') = (\Phi(\bx), \Phi'(\bx'))$.
\end{prop}

\begin{proof}
We first assume that $\xi_j<0$ for all $j<0$. 
Then the morphism $\Phi$ is well-defined by ~\cite[Lemma 19]{M05}.
In this case, it suffices to show that 
the pair $(\widetilde \bx, \widetilde \bx') \equiv (\Phi(\bx), \Phi'(\bx'))$ satisfies the conditions (\ref{reflection-a})-(\ref{reflection-d}).
In light of ~\cite[Lemmas 28, 30]{M02}, this shows that the maps $\Phi'$ and $\Phi_F$ are well-defined,
which is not obvious since the $\widetilde \bx$ and $\widetilde \bx'$ are defined inductively.

We define a new element $\widetilde \by \in \bM(\widetilde V', \widetilde W)$, as a package of linear maps with respect
to the decompositions of $\widetilde V, \widetilde V'$ and $\widetilde W$, such that
a linear map in $\widetilde \by$ is defined to be its counterpart in $\widetilde \bx'$ if it involves $V_i'$,  
or its counterpart in $\widetilde \bx$ otherwise.

If we can show that the pair $(\widetilde \bx, \widetilde \by) $ satisfies the conditions
(\ref{reflection-a})-(\ref{reflection-d}),  then we have
$\widetilde \by \in \Lambda^{s_i(\xi)\text{-} ss}(\widetilde V', \widetilde W)$ by ~\cite[Lemma 28, Lemma 30]{M02}.
By the definition of $\widetilde \by$ and ~\cite[Lemma 18]{M05}, we see immediately that $\widetilde \by = \widetilde \bx'$,
which implies the well-definedness of $\Phi_F$.

The rest is then to  show that the pair $(\widetilde \bx, \widetilde \by)$ satisfies (\ref{reflection-a})-(\ref{reflection-d}).
Clearly from the definitions, the pair satisfies (\ref{reflection-c}) and (\ref{reflection-d}).
We now prove that the sequence in (\ref{reflection-a}) for the pair $(\widetilde \bx, \widetilde \by)$ is a complex, i.e., $b_i(\widetilde \bx) a_i (\widetilde \by) =0$.
We consider the restriction to 
the subspace $W'_i = \oplus_{j\geq i+1, 1\leq h\leq j-i} W_j^{(h)}$ of $\widetilde V'=V'_i \oplus W'_i$, and we get
\[
b_i(\widetilde \bx) a_i (\widetilde \by) |_{W'_i} = 
b_i (\widetilde \bx) a_i(\widetilde \bx)|_{W'_i} = - \mu_i (\widetilde \bx) |_{W'_i} =0.
\]
We then consider the restriction of $b_i (\widetilde \bx) a_i(\widetilde \by)$ to $V'_i$ in two cases.
The first one is to consider $\pi_{V_i} b_i(\widetilde \bx) a_i (\widetilde \by) |_{V'_i}$ where $\pi_{V_i}$ is the project to the component $V_i$.
Any component in $\pi_{V_i} b_i(\widetilde \bx) a_i (\widetilde \by)|_{V'_i}$ passing through $W_{i+1}'$ is zero. 
The possible nonzero component of the linear map $\pi_{V_i} b_i(\widetilde \bx) a_i (\widetilde \by)|_{V'_i}$
passing through $W'_{i-1}$ is when it factors through its subspace $W_{i}^{(1)}$, 
which is $p_iq_i'$ if we use the notations  $\bx =(x_i, y_i, p_i, q_i)$ and $\bx' =(x'_i, y'_i, p'_i, q'_i)$.
Hence we have  
$$
\pi_{V_i} b_i(\widetilde \bx) a_i (\widetilde \by) |_{V'_i} = b_i (\bx) a_i(\bx') =0.
$$
The second case is to consider the component  $\pi_{W'_i}  b_i(\widetilde \bx) a_i (\widetilde \by) |_{V'_i}$. 
In this case, the possible nonzero component factored through the map  $\pi_{W_{i}'} b_i(\widetilde \bx)|_{W'_{i-1}}$ is
when it factors through the linear map $T_{i-1, j, j-i}^{j, j-i+1}: W_{j}^{(j-i+1)} \to W_j^{(j-i)}$, and it equals to
$S^{V'}_{i-1, j, j-i+1}(\widetilde \bx')$ of $\widetilde \bx'$.
On the other hand, the possible nonzero component factored through the map $\pi_{W_i'} b_i (\widetilde \bx)|_{W'_{i+1}}$ is
when it factors through the linear map $- S^{V}_{i, j, j-i}(\widetilde \bx)$ of $\widetilde \bx$ 
and it equals $ - S^{V}_{i, j, j-i}(\widetilde \bx)  x'_i =- S^{V'}_{i-1, j, j-i+1}(\widetilde \bx')$.
The two sums to zero, and therefore we obtain 
$$
\pi_{W'_i}  b_i(\widetilde \bx) a_i (\widetilde \by) |_{V'_i}=0.
$$
Altogether, it confirms that the sequence in (\ref{reflection-a}) for $(\widetilde \bx, \widetilde \by)$ is a complex.

It remains to show that the map $a_i(\widetilde \by)$ is injective for  the pair
$(\widetilde \bx, \widetilde \by)$ to satisfy (\ref{reflection-a}).
Clearly $a_i(\widetilde \by)|_{V'_i} = a_i(\widetilde \bx')|_{V'_i}$ is injective.
The fact that $S^{j, h}_{i-1, j, h} = \mrm{id}_{W_j}$ implies that $a_i(\widetilde \by) |_{W'_i}$ is injective.
So $a_i(\widetilde \by)$ is injective, and thus (\ref{reflection-a}) holds for $(\widetilde \bx, \widetilde \by)$.

Now we show that the pair $(\widetilde \bx, \widetilde \by)$ satisfies the last condition (\ref{reflection-b}).
We observe from the definition of $\widetilde \by$ that 
$a_i(\widetilde \bx)|_{W'_i} \pi_{W'_i} b_i(\widetilde \bx) - a_i(\widetilde \by)|_{W'_i} \pi_{W'_i} b_i(\widetilde \by) =0$.
So it is reduced to show that 
$a_i(\widetilde \bx)|_{V_i} \pi_{V_i} b_i(\widetilde \bx) - a_i(\widetilde \by)|_{V'_i} \pi_{V'_i} b_i(\widetilde \by)=0$.
By using the (\ref{t1}) and (\ref{s1}) of $\widetilde \bx$ and $\widetilde \by$, we have
\begin{align*}
[a_i(\widetilde \bx)|_{V_i} \pi_{V_i} b_i(\widetilde \bx) - a_i(\widetilde \by)|_{V'_i} \pi_{V'_i} b_i(\widetilde \by)]|_{W_{i+1}'}=0,\\
\pi_{W_{i+1}'} [a_i(\widetilde \bx)|_{V_i} \pi_{V_i} b_i(\widetilde \bx) - a_i(\widetilde \by)|_{V'_i} \pi_{V'_i} b_i(\widetilde \by)]=0.
\end{align*}
Moreover, in light of (\ref{t2}) and (\ref{s2}), it yields
\begin{align*}
\begin{split}
& \pi_{W_{i-1}'} [a_i(\widetilde \bx)|_{V_i} \pi_{V_i} b_i(\widetilde \bx) - a_i(\widetilde \by)|_{V'_i} \pi_{V'_i} b_i(\widetilde \by)]|_{W'_{i-1}}\\
& =
\oplus_{j>  i} S^V_{i, j, j-i} \pi_{V_{i+1}} [a_i(\bx) b_i(\bx) - a_i(\bx') b_i(\bx')]|_{V_{i+1}} \oplus_{j' > i} T_{i, V}^{j', 1}=0.
\end{split}
\end{align*}
The following vanishing results can be obtained in a similar manner.
\begin{align*}
\pi_{U_i}  [a_i(\widetilde \bx)|_{V_i} \pi_{V_i} b_i(\widetilde \bx)  - &  a_i(\widetilde \by)|_{V'_i} \pi_{V'_i} b_i(\widetilde \by)]|_{U_i} 
= a_i (\bx) b_i (\bx) - a_i (\bx') b_i(\bx') =0, \\
\pi_{U_i}  [a_i(\widetilde \bx)|_{V_i} \pi_{V_i} b_i(\widetilde \bx) - & a_i(\widetilde \by)|_{V'_i} \pi_{V'_i} b_i(\widetilde \by)]|_{W'_{i-1}} \\
& =  [ a_i (\bx) b_i (\bx) - a_i (\bx') b_i(\bx')]|_{V_{i+1}} \oplus_{j' > i} T^{j', 1}_{i, V} =0, \\
\pi_{W_{i-1}'}  [a_i(\widetilde \bx)|_{V_i} \pi_{V_i} b_i(\widetilde \bx) - & a_i(\widetilde \by)|_{V'_i} \pi_{V'_i} b_i(\widetilde \by)]|_{U_i} \\
& = \oplus_{j> i} S^V_{i, j, j-i}\pi_{V_{i+1}}  [ a_i (\bx) b_i (\bx) - a_i (\bx') b_i(\bx')] =0.
\end{align*}
These analyses imply that the pair $(\widetilde \bx, \widetilde \by)$ satisfies  (\ref{reflection-b}).
This finishes the proof under the assumption that $\xi_j<0$ for all $j\in I$. 

As a result, we see that the map $\Phi$ is well-defined if  $\xi$ is chosen so that the entries in $s_i (\xi)$ are positive. 
Now the general case is obtained by an induction on the length of $w$.
The proposition is thus proved.
\end{proof}

Let $\widetilde S_i: \M_{\zeta} (\widetilde \bv, \widetilde \bw)  \to \M_{s_i(\zeta)} (s_i* \widetilde \bv, \widetilde \bw) $
be the reflection functor in (\ref{Si}) defined for the pair $(\widetilde \bv, \widetilde \bw)$.  
As a consequence of Proposition ~\ref{Phi-Si}, we have the compatibility of the reflection functor $S_i$ (\ref{Si}) 
and Maffei's morphism $\varphi$ (\ref{varphi}).

\begin{prop}
\label{Si-varphi}
Fix $i \in I$.
Assume that $\zeta=(\xi, 0)$ and  there is an element $w\in W$ such that $w(\xi)_j >0$ for all $j \in I$, or $w(\xi)_j < 0$ for all $j\in I$.
Then the reflection functor $S_i$ in (\ref{Si}) and the imbedding $\varphi$ in (\ref{varphi}) 
are compatible with each other. More precisely, we have the following commutative diagram.
\begin{align}
\begin{CD}
\M_{\zeta}(\bv, \bw)  @>\varphi >> \M_{\zeta} (\widetilde \bv, \widetilde \bw) \\
@VS_i VV @VV \widetilde  S_i V\\
 \M_{s_i(\zeta)} (s_i * \bv, \bw) @>\varphi >> \M_{s_i(\zeta)} (s_i* \widetilde \bv, \widetilde \bw) 
\end{CD}
\end{align}
\end{prop}

\subsection{$\sigma$-Quiver varieties and partial resolutions of  nilpotent Slodowy slices}
\label{s-quiver}

For the pair $(\widetilde \bv, \widetilde \bw)$, let $\widetilde \sigma:  \M_{\zeta} (\widetilde \bv, \widetilde \bw) \to  \M_{-w_0(\zeta)} (w_0* \widetilde \bv, \widetilde \bw) $ denote
the isomorphism defined by (\ref{sigma}). 
By combining Propositions ~\ref{tau-varphi} and \ref{Si-varphi}, we obtain the compatibility of the isomorphism $\sigma$ (\ref{sigma})
and the immersion $\varphi$ (\ref{varphi}).

\begin{prop}
\label{prop:sigma-varphi}
Assume that the parameter $\zeta=0$ or  $\zeta=(\xi,0)$ satisfies that $\xi_i >0$ for all $i\in I$. 
Then the isomorphism $\sigma=\sigma_{\zeta, w_0}$ for $\w=w_0$ in (\ref{sigma}) and the imbedding $\varphi$ in (\ref{varphi}) 
are compatible, so that we have the following commutative diagram.
\begin{align}
\begin{CD}
\M_{\zeta}(\bv, \bw) @>\varphi>>  \M_{\zeta} (\widetilde \bv, \widetilde \bw) \\
@V\sigma VV @VV\widetilde \sigma V\\
 \M_{-w_0(\zeta)} (w_0 * \bv, \bw) @>\varphi >>  \M_{-w_0(\zeta)} (w_0* \widetilde \bv, \widetilde \bw) 
\end{CD}
\end{align}
\end{prop}

By an abuse of notations, let  $\widetilde \sigma$ be the isomorphism on $T^*\mathscr F_{\widetilde \bv, \widetilde \bw}$ 
defined by (\ref{sigma1}) with respect to the form on $\widetilde W$ in (\ref{tild-V-form}), 
which is compatible with the isomorphism on $\M_{\xi}(\widetilde \bv, \widetilde \bw)$ 
under the same notation by Theorem~\ref{sigma-sigma-1}. 
So we have 
\begin{cor}
\label{cor:sigma-varphi}
Under the setting of Proposition~\ref{prop:sigma-varphi}, the isomorphism $\sigma$ is compatible with the
isomorphism $\widetilde \sigma$ on $T^*\mathscr F_{\widetilde \bv, \widetilde \bw}$ via $\phi \varphi$. 
\end{cor}

Recall the varieties $\widetilde{\mathcal S}_{\mu', \lambda}$ and $\mathcal S_{\mu', \lambda}$ from (\ref{phi-comm}).
 We have the following analogue for classical groups  of the Nakajima-Maffei theorem.

\begin{thm}
\label{i-Nakajima-Maffei}
Assume that  $\zeta=(\xi, 0)$ with  $\xi_i >0$ for all $i\in I$ and $\theta(\xi) = \xi$. 
Assume further  $w_0 * \bv = \bv$. 
The compositions $\phi \varphi$ and $\phi_0 \varphi_0$ of morphisms from (\ref{varphi}) and (\ref{phi-comm})  
yield isomorphisms: ($\sigma=\sigma_{\zeta, w_0}$ for $\w=w_0$)
\begin{align}
\label{sigma-varphi}
\fS_{\zeta} (\bv, \bw) \cong \widetilde{\mathcal S}_{\mu', \lambda}^{\widetilde \sigma}
\quad \mbox{and} \quad
\fS_{1}(\bv, \bw) \cong  \mathcal S_{\mu', \lambda}^{\widetilde \sigma}.
\end{align}
In particular, we have the following commutative diagram
\begin{align}
\begin{CD}
\fS_{\zeta}(\bv, \bw) @>\phi\varphi >> \widetilde{\mathcal S}_{\mu', \lambda}^{\widetilde \sigma}\\
@V\pi^{\sigma} VV @VV\Pi V\\
\fS_1(\bv, \bw) @>\phi_0 \varphi_0 >>  \mathcal S_{\mu', \lambda}^{\widetilde \sigma}
\end{CD}
\end{align}
\end{thm}

\begin{proof}
The isomorphisms 
are consequences of 
%Proposition ~\ref{prop:sigma-varphi} 
Corollary~\ref{cor:sigma-varphi}
and Theorem ~\ref{Nakajima-Maffei}.
\end{proof}

We now derive some specific results from Theorem ~\ref{i-Nakajima-Maffei} 
assuming that  $W$ is a formed space with alternating sign $\delta_{\bw}$. 
Recall from Lemma ~\ref{tV-alternate} that if $\delta_{\bw}$ alternates, 
then $\widetilde W$ is a $\delta$-form with $\delta = \delta_{\bw,1}$.
Precisely, we write 
$\widetilde{\mathcal S}_{\mu', \lambda}^{\mathfrak o_{\widetilde \bw}}$ for 
$\widetilde{\mathcal S}_{\mu', \lambda}^{\widetilde \sigma}$ if the form associated to $\widetilde \sigma$ is a symmetric form.
Similarly there are notations 
$\mathcal S_{\mu', \lambda}^{\mathfrak o_{\widetilde \bw}}$, 
$\widetilde{\mathcal S}_{\mu', \lambda}^{\mathfrak{sp}_{\widetilde \bw}}$
and
$\mathcal S_{\mu', \lambda}^{\mathfrak{sp}_{\widetilde \bw}}$.

\begin{cor}
\label{cor:i-Nakajima-Maffei}
If $\delta_{\bw}$ alternates, i.e., $\delta_{\bw,i} \delta_{\bw,i+1}=-1$ for all $1\leq i\leq n-1$,  then the isomorphisms in (\ref{sigma-varphi}) read as follows.
\begin{align}
\label{sigma-varphi-O}
\fS_{\zeta} (\bv, \bw) & \cong 
\widetilde{\mathcal S}_{\mu', \lambda}^{\mathfrak o_{\widetilde \bw}},  \
%\Pi^{-1}(\mathcal S_{e_0}) \cap T^* \mathscr F_{\widetilde \bv, \widetilde \bw}^{\mrm{O}_{\widetilde \bw}}, \
\fS_{1}(\bv, \bw) \cong  \mathcal S_{\mu', \lambda}^{\mathfrak o_{\widetilde \bw}},
%\overline{\mathcal O}_{\mu'} \cap  \mathcal S_{e_0} \cap \mathfrak o_{\widetilde \bw}, 
&& \mbox{if} \ \delta_{\bw,i} =(-1)^{i+1}.\\
\fS_{\zeta} (\bv, \bw) & \cong 
\widetilde{\mathcal S}_{\mu', \lambda}^{\mathfrak{sp}_{\widetilde \bw}}, \
%\Pi^{-1}(\mathcal S_{e_0}) \cap T^* \mathscr F_{\widetilde \bv, \widetilde \bw}^{\mrm{Sp}_{\widetilde \bw}}, \
\fS_{1}(\bv, \bw) \cong  \mathcal S_{\mu',\lambda}^{\mathfrak{sp}_{\widetilde \bw}}, 
%\overline{\mathcal O}_{\mu'} \cap  \mathcal S_{e_0} \cap \mathfrak {sp}_{\widetilde \bw}, 
&& \mbox{if} \ \delta_{\bw,i} =(-1)^i.
\end{align}
\end{cor}

\subsection{Rectangular symmetry for classical groups}

By combining  the rectangular symmetry for general linear groups in Section~\ref{rect} 
and Theorem~\ref{i-Nakajima-Maffei},
we can obtain a rectangular symmetry for classical groups. 
We repeat the process in Section~\ref{rect}, we have 
an immersion
\[
\fS_{\zeta}(\bv, \bw) \to T^*\mathscr F_{\widehat \bv, \widehat \bw}^{\widehat{\sigma}},
\]
where the pair $(\widehat \bv, \widehat \bw)$ 
is in (\ref{hat-bw})
and $\widehat \sigma$ is the automorphism 
induced from a form  of the vector space of dimension $\widehat \bw_1$ defined similar to (\ref{tild-V-form}).
There is a similar result as Theorem ~\ref{i-Nakajima-Maffei} in describing the new immersion via the intersection
$\mathcal S_{\widehat \mu', \widehat{\lambda}}^{\widehat \sigma} =
\overline{\mathcal O}^{\widehat \sigma}_{\widehat \mu'} \cap \mathcal S^{\widehat \sigma}_{\widehat e_0}$ and its 
partial Springer resolution 
$\widetilde{\mathcal S}_{\widehat \mu', \widehat{\lambda}}^{\widehat \sigma} =\Pi_{\widehat \bv, \widehat \bw}^{-1}(\mathcal S_{\widehat e_0})^{\widehat \sigma}$.
Thus, we have the following  counterpart of Proposition~\ref{prop:rect}.

\begin{thm}
\label{Sp-O}
%$(i^{\bw_i})_{1\leq i\leq n} \in \mathcal P_{-1}(\widetilde \bw_1)$. 
Let $(\widetilde \bv, \widetilde \bw)$ and $(\widehat \bv, \widehat \bw)$ be the pairs defined by
(\ref{tild-bw}) and (\ref{hat-bw}) such that associated compositions $\mu$ and $\widehat \mu$ satisfy (\ref{mu-hat}) (see also (\ref{mu'-hat})).
Then the following diagram is commutative with  isomorphic horizontal maps, 
which sends $e_0$ of Jordan type $(i^{\bw_i})_{1\leq i\leq n}$ to $\widehat e_0$  of Jordan type $(i^{\bw_{\theta(i)}})_{1\leq i\leq n}$ in the base.  
\begin{align}
\label{Sp-O-diag}
\begin{CD}
\widetilde{\mathcal S}^{\sigma}_{\mu', \lambda}
@>\cong>> 
\widetilde{\mathcal S}^{\widehat \sigma}_{\widehat \mu', \widehat \lambda}\\
@V \Pi_{\widetilde \bv, \widetilde \bw} VV @VV\Pi_{\widehat \bv, \widehat \bw}V\\
\mathcal S_{\mu', \lambda}^{\sigma}
@>\cong >>
\mathcal S_{\widehat \mu', \widehat \lambda}^{\widehat \sigma}
\end{CD}
\end{align}
\end{thm}

\begin{rem}
\label{rem:Sp-O}
Perhaps the most important case of 
Theorem ~\ref{Sp-O} and Corollary ~\ref{cor:i-Nakajima-Maffei} is the rectangular symmetry 
between geometries of $\mrm{Sp}_{2w}$ and $\mrm O_{2w'}$
for various $w$ and $w'$ and  respective Lie algebras $\mathfrak{sp}_{2w}$ and $\mathfrak o_{2w'}$. 
Specifically,  assume that $n$ is even and $\delta_{\bw}$ alternates with  $\delta_{\bw,i} = (-1)^i$,
then Theorem ~\ref{Sp-O} (\ref{Sp-O-diag}) yields the following commutative diagram.
\begin{align}
\label{Sp-O-special}
\begin{CD}
\widetilde{\mathcal S}_{\mu', \lambda}^{\mathfrak{sp}_{\widetilde \bw}} 
@>\cong>> 
\widetilde{\mathcal S}_{\widehat \mu', \widetilde{\lambda}}^{\mathfrak o_{\widehat \bw}} \\
@V \Pi_{\widetilde \bv, \widetilde \bw} VV @VV\Pi_{\widehat \bv, \widehat \bw}V\\
\mathcal S_{\mu', \lambda}^{\mathfrak{sp}_{\widetilde \bw}}
%\overline{\mathcal O}_{\mu'} \cap \mathcal S_{e_0} \cap \mathfrak{sp}_{\widetilde \bw} 
@>\cong >>
\mathcal S_{\widehat \mu',  \widehat{\lambda}}^{\mathfrak o_{\widehat \bw}}
%\overline{\mathcal O}_{\widehat \mu'} \cap \mathcal S_{\widehat e_0} \cap \mathfrak{o}_{\widehat \bw}.
\end{CD}
\end{align}

Further, the associated Springer fibers $\mathscr F^{\mathfrak{sp}_{\widetilde \bw}}_{\widetilde \bv, \widetilde \bw; e_0}$ 
and  
$\mathscr F^{\mathfrak{o}_{\widehat \bw}}_{\widehat \bv, \widehat \bw; \widehat e_0}$ 
of $\Pi_{\widetilde \bv, \widetilde \bw}$ and $\Pi_{\widehat \bv, \widehat \bw}$
are isomorphic:
\begin{align}
\label{i-fiber}
\mathscr F^{\mathfrak{sp}_{\widetilde \bw}}_{\widetilde \bv, \widetilde \bw; e_0}
\cong 
\mathscr F^{\mathfrak{o}_{\widehat \bw}}_{\widehat \bv, \widehat \bw; \widehat e_0}.
\end{align}

In the case when $n$ is even and  $\delta_{\bw}$ alternates with $\delta_{\bw, i}=(-1)^{i+1}$,
%$\delta_1(\bw) =1 $ and $\delta_n(\bw)=-1$, 
one has a similar diagram
with the pair $(\mathfrak{sp}_{\widetilde \bw}, \mathfrak{o}_{\widehat\bw})$ replaced by 
$(\mathfrak{o}_{\widetilde \bw}, \mathfrak{sp}_{\widehat\bw})$.
In a similar manner for $n$ being odd,  
one gets similar diagrams with the pair $(\mathfrak{sp}_{\widetilde \bw}, \mathfrak{o}_{\widehat\bw})$ replaced by
either $(\mathfrak{sp}_{\widetilde \bw}, \mathfrak{sp}_{\widehat\bw})$ or $(\mathfrak{o}_{\widetilde \bw}, \mathfrak{o}_{\widehat\bw})$.
\end{rem}

Now we single out a special pair of $(\bv, \bw)$ for (\ref{Sp-O-special}) and (\ref{i-fiber}) in the following example and relate it to
the works of Henderson-Licata~\cite{HL14} and Wilbert~\cite{W15} (see also ~\cite{ES12}).

\begin{example}
\label{2-row}
Fix an integer $k$ such that $0< k\leq n-k$. 
Define $\bw^\dagger$ by $\bw^\dagger_i = \delta_{i, k} + \delta_{i, n-k}$. 
%We define $\delta_i(\bw^0)= (-1)^i$ for each $i$.
Recall that $n$ is even and set $r=n/2$ for convenience. Let  
$\bv^\dagger$ be a vector defined as follows.
\[
\begin{tabular}{   | c |c | c | c | c | c| c| c | c | c | c | c | c | c| c|} 

 \hline 
$i$ & 1  & 2& $\cdots$  & k-1 & k  & $\cdots$ & r  & r+1& $\cdots$ & n-k & n-k+1 & $\cdots$ & n-1 & n \\
\hline 
$\bv^\dagger_i$  & 1 & 2&$\cdots$  & k-1 & k  & $\cdots$ & k & k+1 & $\cdots$ & k+1 & k & $\cdots$& 2 & 1 \\
 \hline
 \end{tabular}
\]
With respect to the pair $(\bv^\dagger, \bw^\dagger)$, the data in (\ref{Sp-O-special}) read
\begin{align*}
& \widetilde \bv^\dagger= (n-1,  \cdots, r+1,  r, r, r-1, \cdots, 1), && \widetilde \bw^\dagger_i= \delta_{i, 1}n.\\
& \widehat \bv^\dagger = (n+1, \cdots, r+3, r+2, r, r-1, \cdots, 1), && \widehat \bw^\dagger_i = \delta_{i, 1} (n+2). 
\end{align*} 
Note that $\widetilde \bv^\dagger$ has an extra $r$, while $r+1$ is missing from $\widehat \bv^\dagger$. 
In particular, we have
\begin{align*}
& \mu' = n^1, & & e_0 \in \mathcal O_{k^1(n-k)^1} \cap \mathfrak{sp}_n.\\
&  \widehat \mu' = 1^1 (n+1)^1, && \widehat e_0 \in \mathcal O_{(k+1)^1 (n+1-k)^1} \cap \mathfrak o_{n+2}.
\end{align*}
The bottom row of (\ref{Sp-O-special}) reads as the following, which is  Corollary 5.2 in~\cite{HL14}. 
\begin{align}
\label{Sp-O-special-2}
\mathcal S_{n^1, k^1(n-k)^1}^{\mathfrak{sp}_n} 
%\overline{\mathcal O}_{n^1} \cap \mathcal S_{k^1 (n-k)^1} \cap \mathfrak{sp}_{n}
\cong 
\mathcal S_{1^1 (n+1)^1, (k+1)^1 (n+1-k)^1}^{\mathfrak o_{n+2}}.
%\overline{\mathcal O}_{1^1 (n+1)^1} \cap \mathcal S_{(k+1)^1(n+1-k)^1} \cap \mathfrak{o}_{n+2},
\end{align}
Note that both sides in (\ref{Sp-O-special-2}) are empty unless $k$ is even or $k=n-k$.

Observe the $\mathscr F^{\mathfrak{sp}_n}_{\widetilde \bv^\dagger, \widetilde \bw^\dagger}$ is
 the complete flag variety of $\mrm{Sp}_n$.
Hence the left hand side of (\ref{i-fiber}) is the Springer fiber, say $\mathscr B^{\mathfrak{sp}_n}_{e_0}$, of $e_0$. 
In light of the fact that  the complete flag variety of $\mrm O_2$ consists of two points, 
the $\mathscr F^{\mathfrak{o}_{n+2}}_{\widehat \bv^\dagger, \widehat \bw^\dagger}$ 
is isomorphic to a connected component,
say $\mathscr B^{\mathfrak{so}_{n+2}}$,
of the complete flag variety of $\mrm{O}_{n+2}$.
So we get 
$T^* \mathscr B^{\mathfrak{so}_{n+2}} \cong 
T^*\mathscr F^{\mathfrak o_{n+2}}_{\widehat \bv^\dagger, \widehat \bw^\dagger}$.
So the right hand side of  the (\ref{i-fiber}) for $(\bv^\dagger, \bw^\dagger)$ is exactly the Springer fiber 
$\mathscr B^{\mathfrak{so}_{n+2}}_{\widehat e_0}$ of $\widehat e_0$. 
Thus, the equality (\ref{i-fiber}) is transformed into  the following isomorphism, which is 
Theorem B in~\cite{W15}.
\begin{align}
\label{Sp-O-special-3}
\mathscr B^{\mathfrak{sp}_n}_{e_0}
\cong 
\mathscr B^{\mathfrak{so}_{n+2}}_{\widehat e_0}.
\end{align}

Finally, the top row of (\ref{Sp-O-special}) implies the following isomorphism of the Springer resolutions of the nilpotent Slodowy slices in (\ref{Sp-O-special-2}), which proves a conjecture in~\cite[1.3]{HL14}.
\begin{align}
\label{Sp-O-special-4}
\widetilde{\mathcal S}_{n^1, k^1(n-k)^1}^{\mathfrak{sp}_n} 
\cong 
\widetilde{\mathcal S}_{1^1 (n+1)^1, (k+1)^1 (n+1-k)^1}^{\mathfrak o_{n+2}}.
\end{align}
The isomorphism (\ref{Sp-O-special-4}) together with~\cite[Theorem 1.2]{HL14} 
implies a conjecture by McGerty and Lusztig on the relationship between
type D Nakajima varieties and Slodowy varieties in~\cite[1.3]{HL14}.
\end{example}

\subsection{Column/Row removal reductions for classical groups}

Now we investigate the classical counterpart of the geometric column/row removal reductions in Section~\ref{col}.
Recall from Section~\ref{col}, we have the  isomorphism $\M_{\zeta}(\bv, \bw) \cong \M_{\zeta} (\breve \bv, \breve \bw)'$.
If  $V$ and $W$ are formed spaces with signs $\tilde \delta_{\bv}$ and  $\delta_{\bw}$) respectively, 
then the associated vector space $\breve V$ of dimension $\breve \bv$
(resp. $\breve W$) naturally inherit one from $\tilde \delta_{\bv}$ (resp. $\delta_{\bw}$).
So we have an automorphism $\sigma'$ on  $\M_{\zeta}(\breve \bv, \breve \bw)'$. 
In particular, we have the following  geometric incarnation of ~\cite[Proposition 13.5]{KP82}.

\begin{prop}
\label{i-prop:col}
There is $\mathcal S_{\mu', \lambda}^{\widetilde \sigma} \cong \mathcal S_{\breve \mu', \breve \lambda}^{\widetilde \sigma'}$ where $\widetilde \sigma$ is from (\ref{s-quiver}) and $\widetilde \sigma'$ is defined similarly.
\end{prop}

Since the definition of $\sigma'$ involves the longest Weyl group element of the Dynkin diagram of type $A_{n+1}$,
it is not immediately clear how to compare $ \widetilde{\mathcal S}_{\mu', \lambda}^{\widetilde \sigma}$ and 
$\widetilde{\mathcal S}_{\breve \mu', \breve \lambda}^{\widetilde \sigma'}$. In general, the two varieties are not isomorphic. 
However, in the case when $(\mu', \lambda)$ and $(\breve \mu', \breve \lambda)$ satisfy
the conditions in Theorem~\ref{Sp-O}, they are isomorphic. 
Example~\ref{2-row} is such a case. 
Similarly, we have the following counterpart of Proposition~\ref{prop:row}.

\begin{prop}
\label{i-prop:row}
Suppose that the pair $(\ddot \mu', \ddot \lambda)$ is defined by (\ref{row-addition}). 
There is an isomorphism $\mathcal S_{\mu', \lambda}^{\widetilde \sigma} \cong \mathcal S_{\ddot \mu', \ddot \lambda}^{\ddot \sigma}$ where $\widetilde \sigma$ is from (\ref{s-quiver}) and $\ddot \sigma$ is defined similarly.
\end{prop}

By Propositions~\ref{i-prop:col} and~\ref{i-prop:row}, one has a geometric version of Theorem 12.3 in~\cite{KP82}.

\section{Fixed-points and categorical quotients}
\label{fixed-cat}

In this section, we consider quiver varieties 
$ \M_0(\bv^0,\bw^0)^{\tau}$
of a general Dynkin graph  
for those pairs of formed spaces ($V^0, W^0$) of dimension vectors
$(\bv^0, \bw^0)$ and signs  $\tilde \delta_{\bv^0} $ and $\delta_{\bw^0}$ are chosen to be alternating, i.e.,
$\tilde \delta_{\bv^0, i} \delta_{\bw^0,i} =-1$ for all $i\in I$ and $\tilde \delta_{\bv^0,\i(h)} \tilde \delta_{\bv^0,\o (h)} =-1$ for all $h\in H$. 
In Remark~\ref{|a|=2}, we will consider $\M_0(\bv^0,\bw^0)^{a\tau}$. 
We show that there is a closed immersion from 
Kraft-Procesi-Nakajima's construction~\cite{KP82, N94} via categorical quotients to $\sigma$-quiver varietes 
$\M_0(\bv,\bw)^{a\tau}$.

\subsection{Polynomial invariants on $\bM (\bv^0, \bw^0)^{\tau}$}

Recall the automorphism $\tau$ on $\bM (\bv^0, \bw^0)$ from (\ref{tau-raw}) and $\G_{\bv^0}$ from (\ref{Gv}). 
Let $\G^{\tau}_{\bv^0} =\{ g\in \G_{\bv^0} | g_i g_i^* =1\}$ and 
$\bM^{\tau} =\bM(\bv^0, \bw^0)^{\tau}$ be the variety of $\tau$-fixed points in $\bM(\bv^0, \bw^0)$.
We are interested in finding a set of generators for  the algebra $R^{\G^{\tau}_{\bv^0}}$ 
of $\G^{\tau}_{\bv^0}$-invariant  regular functions on $\bM^{\tau}$.
Following Lusztig, we consider the following elements in $R^{\G^{\tau}_{\bv^0}}$.
A sequence $h_1, \cdots, h_s$ of arrows in $H$ is called a path if $\i (h_i) = \o(h_{i+1})$ for all $1\leq i\leq s-1$.
It is called a cycle if it further satisfies $\i (h_s) = \o (h_1)$.
For a cycle $h_1, \cdots, h_s$ in $H$, we define a $\G^{\tau}_{\bv^0}$-invariant function $\mrm{tr}_{h_1,\cdots, h_s}$ 
on $\bM^{\tau}$ by 
\begin{align}
\label{Tr}
\mrm{tr}_{h_1,\cdots, h_s} (\bx) = \mrm{trace} (x_{h_s} x_{h_{s-1}} \cdots x_{h_1}), \quad \forall \bx \in \bM^{\tau}.
\end{align}
For any path $h_1,\cdots, h_s \in H$ and a linear form $\chi$ on $\Hom (W^0_{\o(h_1)}, W^0_{\i (h_s)})$, we define
a $\G^{\tau}_{\bv}$-invariant function $\chi_{h_1,\cdots, h_s}$ on $\bM^{\tau}$ by
\begin{align}
\label{f-chi}
\chi_{h_1, \cdots, h_s} (\bx) = \chi (q_{\i (h_s)} x_{h_s} x_{h_{s-1}} \cdots x_{h_1} p_{\o(h_1)}),\quad \forall \bx \in \bM^{\tau}.
\end{align}

The following theorem is an analogue for classical groups of ~\cite[Theorem 1.3]{L98}. 

\begin{thm}
\label{R-G-inv}
Assume that the signs $\tilde \delta_{\bv^0} $ and $\delta_{\bw^0}$ alternate.
The algebra $R^{\G^{\tau}_{\bv^0}}$ is generated by the functions of the forms  (\ref{Tr}) and (\ref{f-chi}).
In particular, the algebra of $\G^{\tau}_{\bv^0}$-invariant regular functions on
$\Lambda(\bv^0, \bw^0)^{\tau}$ in (\ref{Lambda-tau-raw}) is generated by the restriction of the functions (\ref{Tr}) and (\ref{f-chi}) to $\Lambda(\bv^0, \bw^0)^{\tau}$.
\end{thm}

The remaining part of this section is devoted to the proof of Theorem~\ref{R-G-inv}.
The proof is modeled on that of ~\cite[Theorem 1.3]{L98} with slight modifications.
Instead of  the results on tensor invariants for general linear groups, 
we need a similar result on tensor invariants for classical groups as follows.
Let $E$ be a $\delta$-formed space with the form $(-, -)_E$ and 
$G(E)$ be the group of isometries with respect to the form $(-,-)_E$.
If $n$ is even and $x=\{ (i_1, j_1), \cdots, (i_{n/2}, j_{n/2}) \}$ is a set of ordered pairs 
such that $\{ i_1, j_1, \cdots, i_{n/2}, j_{n/2}\}=\{1,\cdots, n\}$, we define 
the following $G(E)$-invariant linear forms on $T= E^{\otimes n}$ by 
\begin{align}
\label{f_x}
f_x ( e_1\otimes \cdots \otimes e_n) = \prod_{k =1}^{n/2} (e_{i_k}, e_{j_k}), 
\quad \forall e_1\otimes \cdots \otimes e_n \in E^{\otimes n}.
\end{align}

\begin{prop} [\cite{W}]
\label{W-classical}
The space of $G(E)$-invariant linear forms on the tensor space $T$
is zero when $n$ is odd, and is spanned by the forms $f_x$ (\ref{f_x}) for various $x$ when $n$ is even.
\end{prop}

Now we begin to prove Theorem ~\ref{R-G-inv}.
For simplicity we write $V$ and $W$ for $V^0$ and $W^0$ respectively in the proof. 
Recall the function $\ve: H\to \{ \pm 1\}$ from Section ~\ref{lavw}.
Let $\Omega = \ve^{-1}(1)$ and we set
\[
\bM_{\Omega} =\bM_{\Omega} (\bv^0, \bw^0) = \oplus_{h\in \Omega} \Hom (V_{\o (h)}, V_{\i (h)} ) \oplus \oplus_{i\in I} \Hom (V_i, W_i).
\]
Since the parameters $\tilde \delta_{\bv^0}$ and $\delta_{\bw^0}$ alternate, there is an isomorphism
\[
\bM^{\tau} \cong \bM_{\Omega}  
\]
given by projection. 
After fixing a basis $\mathcal B_i$ for $W_i$ and identifying $V_i$ and $V_i^*$ via the forms, 
we have 
\begin{align}
\bM_{\Omega}  \cong 
\otimes_{h\in \Omega} V_{\o (h)} \otimes V_{\i (h)} \oplus \oplus_{i\in I, b\in \mathcal B_i} V_{i, b},
\end{align}
where $V_{i, b}$ is a copy of $V_i$ indexed by $b$.

Following Lustig, it is enough to show that 
the space of $\G_{\bv^0}^{\tau}$-invariant regular functions on $\bM^{\tau}$
of homogenous degree $n$ is spanned by various products of functions of the form (\ref{Tr}) and (\ref{f-chi}).
Thanks to ~\cite[Lemma 1.4]{L98}, it is reduced to show that this is also the case  for 
the space of $\G^{\tau}_{\bv^0}$-invariant linear forms on $(\bM^{\tau})^{\otimes n}$.
To this end, it is further reduced to study the $\G^{\tau}_{\bv^0}$-invariant linear forms on the 
tensor space $T = E_1\otimes \cdots \otimes E_n$ where $E_i$ is either $V_{\o (h)} \otimes V_{\i (h)}$
or $V_{i, b}$.
Write $T =\otimes_{i\in I} E^i$ where $E^i$ is the tensor product of all $V_i$ in $T$.
In light of  ~\cite[Lemma 1.5]{L98}, 
the $\G^{\tau}_{\bv^0}$-invariant linear forms on $T$ are the tensor products of $G(V_i)$-invariant linear forms on $E^i$.
If $T$ can be decomposed as the tensor product of components of 
the following forms,
\begin{align*}
\begin{split}
V_{\o (h_1)} \otimes V_{\i (h_1)}^{\otimes 2} 
\otimes V_{\i (h_2)}^{\otimes 2} 
\otimes \cdots \otimes V_{\i  (h_{s-1})}^{\otimes 2} \otimes V_{\i (h_{s})},   \mbox{where}\ &  \mbox{$h_1\cdots   h_r$ is a cycle in $H$},\\
V_{\o (h_1), b} \otimes V_{\o (h_1)} \otimes V_{\i (h_1)}^{\otimes 2} 
\otimes \cdots \otimes V_{\i (h_{s-1})}^{\otimes 2} \otimes V_{\i (h_{s})}
\otimes V_{\i(h_s), b'}, &   \\
 \hfill \mbox{where} \ b \in \mathcal B_{\o(h_1)}, b' \in \mathcal B_{\i (h_s)}, & \  h_1\cdots h_r \ \mbox{is a path in $H$},
\end{split}
\end{align*}
then by applying Proposition ~\ref{W-classical} the  space of 
$\G^{\tau}_{\bv^0}$-invariant linear forms on $T$ is spanned by
the tensor products of the $f_x$ in (\ref{f_x}) for various $x$. 
The latters  in turn are
products of various functions in (\ref{Tr}) and (\ref{f-chi}).
Now following the proof of ~\cite[Theorem 1.3]{L98}
we see that the space of  $\G^{\tau}_{\bv^0}$-invariant linear forms on $T$ is spanned by
products of linear forms in (\ref{Tr}) and (\ref{f-chi}).
Theorem ~\ref{R-G-inv} is thus proved.

\subsection{The closed immersion $\iota$}

Recall $\Lambda(\bv^0, \bw^0)=\Lambda_{\zeta_{\mbb C}} (\bv^0, \bw^0)$ from (\ref{Lambda}) with $\zeta_{\mbb C} =0$. 
We can consider the categorical quotient $\Lambda(\bv^0, \bw^0)^{\tau}//\G^{\tau}_{\bv^0}$. 
Unlike $\M_0(\bv^0, \bw^0)^{\tau}$, the variety $\Lambda(\bv^0, \bw^0)^{\tau}//\G^{\tau}_{\bv^0}$ depends on the forms associated to  $\bv^0$.
For example, when $\Gamma =\mrm A_1$ and the forms on $\bv^0, \bw^0$ do not alternate, $\Lambda(\bv^0, \bw^0)^{\tau}//\G^{\tau}_{\bv^0}=\{\mrm{pt}\}$,  otherwise it is isomorphic
to the determinantal variety in $\mathfrak g(\bw^0)$ of endomorphisms of rank $\leq \dim \bv^0$. 
For the latter fact, we refer the reader to ~\cite[Theorem 1.2]{KP82}.

By the universality of categorical quotient, there is  a morphism 
$$
\Lambda(\bv^0, \bw^0)^{\tau}// \G^{\tau}_{\bv^0}\to  \Lambda (\bv^0, \bw^0)//\G_{\bv^0} =\M_0(\bv^0, \bw^0),
$$
which factors through $\M_0(\bv^0, \bw^0)^{\tau}$ so that we have a morphism of varieties:
\begin{align}
\label{iota}
\iota: \Lambda(\bv^0, \bw^0)^{\tau}// \G^{\tau}_{\bv^0}\to \M_0(\bv^0, \bw^0)^{\tau}.
\end{align}

\begin{prop}
\label{prop:iota-1}
The morphism $\iota$  in (\ref{iota})  is a  closed immersion for an arbitrary graph. 
\end{prop}

\begin{proof}
It is enough to show that 
the induced map 
\[
\mbb C [\Lambda (\bv^0, \bw^0)]^{\G_{\bv^0}} \to \mbb C[\Lambda(\bv^0, \bw^0)^{\tau}]^{\G^{\tau}_{\bv^0}}
\]
of the inclusion $\Lambda(\bv^0, \bw^0)^{\tau} \to \Lambda(\bv^0, \bw^0)$ is  surjective. 
But this is the case by ~\cite[Theorem 1.3]{L98} and Theorem ~\ref{R-G-inv}.
The proposition is thus proved.
\end{proof}

For the remaining part of this section, we assume that $\Gamma$ is of type $\mrm A_n$.
When $\bw^0_i=0$ for all $i\geq 2$, 
the variety  $\Lambda(\bv^0,  \bw^0)^{\tau}// \G^{\tau}_{ \bv^0}$ is  studied by Kraft-Procesi in ~\cite{KP82}.
The generalization to arbitrary $\bw^0$
is mentioned by Nakajima implicitly in ~\cite[Remark 8.5.4]{N94} 
and explicitly in ~\cite[Appendix A(ii)]{N15}.
See also ~\cite{K90}.
Now we shall sharpen the previous result in type $A_n$. 
By Proposition ~\ref{tau-Phi}, we have a closed immersion.
\[
\Lambda (\bv^0, \bw^0)^{\tau} \overset{\Phi}{\longrightarrow} \Lambda(\widetilde \bv^0, \widetilde \bw^0)^{\widetilde \tau}.
\]
There is a natural imbedding $\G_{\bv^0} \to \G_{\widetilde \bv^0}$ with respect to the decomposition (\ref{t-V-W}), which
restricts to an imbedding $\G^{\tau}_{\bv^0} \to \G^{\widetilde \tau}_{\widetilde \bv^0}$. 
This induces a morphism of varieties
\begin{align}
\label{varphi'}
\Lambda (\bv^0, \bw^0)^{\tau}// \G^{\tau}_{\bv^0} 
\overset{\varphi_0'}{\longrightarrow} 
\Lambda(\widetilde \bv^0, \widetilde \bw^0)^{\widetilde \tau}//
\G^{\widetilde \tau}_{\widetilde \bv^0}.
\end{align}

Putting (\ref{iota}) and (\ref{varphi'}) together yields the following commutative diagram.
\begin{align}
\label{iota-varphi}
\begin{CD}
\Lambda (\bv^0, \bw^0)^{\tau}// \G^{\tau}_{\bv^0} @> \varphi'_0 >> \Lambda(\widetilde \bv^0, \widetilde \bw^0)^{\widetilde \tau}//
\G^{\widetilde \tau}_{\widetilde \bv^0}\\
@V \iota VV @VV\tilde \iota V \\
\M_0(\bv^0, \bw^0)^{\tau} @>\varphi_0>> \M_0(\widetilde \bv^0, \widetilde \bw^0)^{\tau},
\end{CD}
\end{align}
where $\iota$ and $\tilde \iota$ are the morphisms defined in (\ref{iota}).

\begin{prop} 
\label{prop:iota}
When the graph is of Dynkin type $\mrm A_n$ and the signs $\tilde \delta_{\bv^0} , \delta_{\bw^0}$ alternate, 
the morphism $\varphi'_0$ in (\ref{varphi'})  is a closed immersion and $ \tilde \iota$  in (\ref{iota-varphi}) is an isomorphism.
\end{prop}

\begin{proof}
We have a commutative diagram
\[
\xymatrix{
\Lambda(\widetilde \bv^0, \widetilde \bw^0)^{\widetilde \tau}// \G^{\widetilde \tau}_{\widetilde \bv^0} 
\ar@{->}[rr]^{\tilde \iota} \ar@{->}[dr] & &  \M_0(\widetilde \bv^0, \widetilde \bw^0)^{\tau} \ar@{->}[dl] \\
& \mathfrak{gl}(\widetilde \bw^0)^{\widetilde \sigma} 
%\overline{\mathcal O}_{\mu'}^{\widetilde \sigma} 
&
}
\]
where the morphism on the right is from (\ref{Sp-O-diag}) and the one on the left is defined in a similar way.
Both morphisms are  closed immersions with the same image 
by Theorem ~\ref{i-Nakajima-Maffei} and (a slightly general version of) ~\cite[Theorem 5.3]{KP82}, 
which implies that $\tilde \iota$ is isomorphic. 

Since $\varphi_0$ and $\iota$ are closed immersions, so is $\varphi'_0$ by using  the commutative diagram (\ref{iota-varphi}).
The proposition is thus proved.
\end{proof}

\begin{rem}
In light of ~\cite{K90, N15} and Theorem ~\ref{i-Nakajima-Maffei}, 
it is expected that $\iota$ in (\ref{iota}) is an isomorphism for a Dynkin graph of type $\mrm A_n$.
We conjecture that this holds for any graph.
\end{rem}

Most results in this section can be extended to a more general situation where the isomorphism $a$ in Section~\ref{dia} is involved in a straightforward manner. 
We end this section with a remark on the connection with ~\cite[(Ai), (Aiii), (Aiv)]{N15}, which is grown
out from a discussion with Professor H. Nakajima. 

\begin{rem}
\label{|a|=2}
(1). Consider the Dynkin diagram of type $\mrm A_{2n+1}$. Then there is a closed-immersion similar to (\ref{iota}):
\begin{align*}
\label{iota-a}
\iota': \Lambda(\bv^0, \bw^0)^{a \tau}// \G^{a\tau}_{\bv^0}\to \M_0(\bv^0, \bw^0)^{a\tau}.
\end{align*}
%where $\sigma$ is defined using $a$ and $\tau$. 
The domain of $\iota'$ is an $S^1$-equivariant instantons moduli space on $\mbb R^4$  in ~\cite[(Aiii)]{N15}.
Specifically, if the form on $\bw^0_{i}$ is an orthogonal form for all $i$, then the domain of $\iota'$ is exactly the $\mrm{SO}(r)$-instantons in {\it loc. cit.}, Figure 7. 
(Note that $\bw^0_n$ corresponds to $w_0$ in Figure 7 in {\it loc. cit.})
The orthogonal/symplectic forms in {\it loc. cit.} are defined over $V_i \oplus V_{-i}$ and $W_i\oplus W_{-i}$, similar to~\cite{E09}. In our setting, 
we  assign to  $V_n$ a symplectic form and each
$W_i$ and $V_i$ for $i\neq n$ an orthogonal form $(- | -)$, set $W_i=W_{2n-i}$, $V_i=V_{2n-i}$. From these data,  we can obtain  orthogonal/symplectic forms used in {\it loc. cit} on $W_i\oplus W_{-i}$ or $V_i\oplus V_{2n-i}$
by the rule
$[ (u_1, u_2), (w_1, w_2)] = (u_1|w_2) \pm (u_2| w_1)$,
where the choice of $+$ leads to an orthogonal form and the choice of $-$ leads to a symplectic form as desired. 
Under this setting, the domain of $\iota'$ is exactly the instanton moduli space given in Figure 7, {\it loc. cit.}
Note that in this setting, the orders of $a$ and $\tau$ are $4$, while their composition $a\tau$ has order $2$. 

(2) If our graph allows loops, the arguments in this section still work through, with a minor modification in the proof of Theorem~\ref{R-G-inv}.  
In particular, when the graph is a Jordan quiver, i.e., a vertex with two arrows, then we have a closed-immersion
\begin{align*}
%\label{iota-a2}
\iota'': \Lambda(\bv^0, \bw^0)^{- a \tau}// \G^{a\tau}_{\bv^0}\to \M_0(\bv^0, \bw^0)^{a\tau}.
\end{align*}
where $a$ is induced by the obvious involution on the Jordan quiver.
The domain of $\iota''$ is an $\mrm{SO}/\mrm{Sp}$ instantons moduli space on $\mbb R^4$ in~\cite[(Ai)]{N15}. 
(See~\cite{Choy} for further details.)

(3) Let $H$ be a finite subgroup in $\mrm{SU}(2)$. By taking the $H$-equivariant parts in $\iota''$, one obtains a similar closed immersion  whose domain is exactly
the $\mrm{SO}/\mrm{Sp}$ instantons moduli space on $\mbb R^4/H$, which is discussed in~\cite[(Aiv)]{N15}. 

Composing $a\tau$ or $\tau$  with the reflection functor $S_{w_0}$, it also gives rise to the $\mrm{SO}/\mrm{Sp}$ instantons moduli space on ALE spaces
if $\M_0(\bv^0, \bw^0)$ is replaced by $\M_{\zeta}^{\mrm{reg}}(\bv^0, \bw^0)$ for $\zeta$ generic.
In particular, if the McKay diagram of $\H$ is of type $\mrm D^{(1)}_{2n}$, $\mrm E^{(1)}_7$ or $\mrm E^{(1)}_8$, one uses $\tau S_{w_0}$, and
$a\tau S_{w_0}$ is used for the remaining cases $\mrm A^{(1)}_n$, $\mrm D^{(1)}_{2n+1}$ and $\mrm E^{(1)}_6$. 
This is known to Nakajima, see~\cite{N18}, and implicitly given in~\cite[Sect. 9]{N03}.
\end{rem}

\section{Quiver varieties and symmetric spaces}
\label{sQV+QSS}

In this section, we study fixed-point subvarieties of Nakajima varieties under an anti-symplectic automorphism.
In type A case, we identify  them with the symmetric space of a given symmetric pair of type AI/AII.

\subsection{The anti-symplectic automorphism $\hat \tau_{\zeta}$}

Similar to $\tau_{\zeta}$, we define a simpler automorphism
\[
\hat \tau: \bM (\bv, \bw) \to \bM (\bv, \bw), \bx= (x_h, p_i, q_i) \mapsto ({}^{\hat \tau} x_h, {}^{\hat\tau} p_i, {}^{\hat\tau} q_i),
\]
where 
$
{}^{\hat\tau} x_h = x^*_{\bar h},\ {}^{\hat \tau} p_i = q_i^* , \ {}^{\hat \tau} q_i = p^*_i, \quad \forall h\in H, i\in I.
$
The $\hat\tau_{\zeta}$ only differs from $\tau_{\zeta}$ by a minus sign at $x_h$ for $h\in \ve^{-1}(-1)$ and $p_i$. 
Despite this minor perturbation, the new automorphism behaves quite differently from $\tau_{\zeta}$, as we shall see in the following 
and yet proofs  are always in parallel 
with the old ones with minor modifications, which often involve removals of minus signs.
It is easy to see 
$
\mu({}^{\hat \tau}\bx) = \mu(\bx)^*.
$
So it induces an isomorphism on $\M_{\zeta}(\bv, \bw)$:
\[
\hat\tau_{\zeta} : \M_{(\xi, \zeta_{\mbb C})}(\bv, \bw) \to \M_{(-\xi, \zeta_{\mbb C})}(\bv, \bw).
\]
It is also clear that $\hat \tau_{\zeta}$ is independent of the choices of forms on $V$ by the same argument for the similar property of $\tau_{\zeta}$. 
In contrast with its symplectic analogue $\tau_{\zeta}$, the $\hat \tau_{\zeta}$ is anti-symplectic, that is, 
$$
\omega ({}^{\hat\tau} \bx, {}^{\hat\tau}\bx') = - \omega (\bx, \bx'),
$$
which can be verified by definition.  Now we determine the order of $\hat \tau_{\zeta}$.

\begin{prop}
If the forms on $W$  are uniform, i.e., $\delta_{\bw,i}=\delta_{\bw, j}$ for all $i, j\in I$,
then the $\hat\tau_{\zeta}$ is involutive: $\hat\tau_{\zeta}^2=1$.
In general, if $W$ is a formed space with sign $\delta_{\bw}$, then $\hat\tau_{\zeta}^4=1$.
\end{prop}

\begin{proof}
The proof follows the same line as that of Proposition 3.3.2 with the
observation that
$\hat \tau^2([\bx]) = [(x_h, \delta_{\bw,i} p_i, \delta_{\bw,i} q_i)]=[\bx]$,
where the last equality is given by the action of the element $(\delta_{\bw, i} \mrm{id}_{V_i})_{i\in I}\in \G_{\bv}$.
The above observation indicates that $\hat\tau_{\zeta}^4=1$. Proposition is thus proved.
\end{proof}

It is clear that the isomorphism $\hat \tau_{\zeta}$ commutes with the isomorphisms $a$ and $S_{\w}$.

\begin{lem}
One has
$S_i \hat \tau_{\zeta} = \hat \tau_{s_i \zeta} S_i$
and $a \hat \tau_{\zeta} = \hat \tau_{a \zeta} a$.
\end{lem}

\subsection{The  $\hat \sigma$-quiver varieties}

Similar to $\sigma$, we consider the following isomorphism
\begin{align}
\hat \sigma : = a S_{\w} \hat \tau_{\zeta}: \M_{(\xi, \zeta_{\mbb C})}(\bv, \bw)  \to 
\M_{(- a \w \xi, a \w \zeta_{\mbb C})} (a\w*\bv,a \bw), \quad \forall \w\in \mathcal W.
\end{align}

The $\hat \sigma$-$quiver$ $variety$ is defined to be 
$$
\fP_{\zeta}(\bv, \bw) \equiv \M_{\zeta}(\bv, \bw)^{\hat \sigma},
$$ 
whenever $a \w \zeta_{\mbb C} = \zeta_{\mbb C}$, $- a \w \xi =\xi $
and $a \w*\bv= \bv$. By summing over all $\bv$, we have 
$$
\fP_{\zeta}(\bw) \equiv
\M_{\zeta}(\bw)^{\hat \sigma}.
$$
It is clear that $\fP_{\zeta}(\bv, \bw)$, and hence $\fP_{\zeta}(\bw)$, is independent of the choice of the form on $V$, due to the same property on $\tau_{\zeta}$.
Since the $a$ and $S_{\w}$ are symplectomorphisms and $\hat \tau_{\zeta}$ is anti-symplectic, the $\hat \sigma$ is anti-symplectic.
Summing up, we have

\begin{prop}
$\fP_{\zeta}(\bv, \bw)$ is a fixed-point subvariety  of $\M_{\zeta}(\bv, \bw)$ under an anti-symplectic automorphism.
Its definition is independent of the choice of the form on $V$. 
If $\zeta$ is generic, then $\fP_{\zeta}(\bv, \bw)$ is smooth if it is nonempty and $\hat\sigma$ is of finite order. 
If $W$ is a formed space with sign $\delta_{\bw}$ 
and $\w$ is of finite order, then the order of $\hat \sigma$ is a divisor of l.c.m.$\{4, |\w|, |a|\}$.
If further $\delta_{\bw}$ is uniform and $a^2=\w^2 =1$, then $\hat\sigma^2=1$.
\end{prop}

Just like $\sigma$-quiver varieties, the $\hat\sigma$-quiver varieties include original quiver varieties. 
By a general property of anti-symplectic involution, we have

\begin{prop}
If $|\hat \sigma| =2$, i.e., $\hat \sigma$ is anti-involutive, then  $\fP_{\zeta}(\bv, \bw)$ is a Lagrangian subvariety
of $\M_{\zeta}(\bv, \bw)$. In particular, 
the dimension of $\fP_{\zeta}(\bv, \bw)$, if nonempty,  is half of the dimension of $\M_{\zeta}(\bv, \bw)$. 
\end{prop}

Via restriction, there is a proper map
\begin{align}
\pi^{\hat \sigma}: \fP_{\zeta}(\bv, \bw) \to \fP_1(\bv, \bw),
\end{align}
where $\fP_1(\bv,\bw)$ is defined in the same way as $\fS_1(\bv,\bw)$ in (\ref{S0-def}).

Arguing in a similar way as the $\sigma$ case, it yields

\begin{prop}
The map $\pi^{\hat \sigma}$ is $\G_{\bw}^{\sigma}$-equivariant and if $\zeta_{\mbb C}=0$, it is $\G^{\sigma}_{\bw}\times \mbb C^{\times}$-equivariant. 
\end{prop}

Recall the fixed-point subgroup $\mathcal W^{\w, a}$. 
For any $x\in \mathcal W^{\w, a}$, the original reflection functor induces an action on the $\hat \sigma$-quiver varieties
\[
S_x^{\hat\sigma} : \fP_{\zeta}(\bv, \bw) \to \fP_{x \zeta}(x * \bv, \bw). 
\]
Further, the group $\mathcal W^{\w, a}$ acts on the cohomology group $\H^*(\fP_{\zeta}(\bv, \bw), \mbb Z)$ 
when $\bw - \mathbf C \bv=0$.

\subsection{$\hat \sigma$-quiver varieties of type $\mrm A$}

Recall the setting from  Section 6.1. We define
\begin{align}
\mathfrak p \equiv \mathfrak p(W) = \{ x\in \End (W) | x = x^* \}. 
\end{align}
This is called a $symmetric$ $space$ with respect to the symmetric pair $(\mathfrak{gl}(W), \mathfrak g(W))$.
Let $\N(\mathfrak p)$ be the variety of nilpotent elements in $\mathfrak p$. 

Consider the automorphism $\hat \sigma_1$ on $T^*\mathscr F_{\bv,\bw}$ defined by $(x, F) \mapsto (x^*, F^{\perp})$, 
and the fixed point subvariety $(T^* \mathscr F_{\bv, \bw})^{\hat \sigma_1}$.
Let

\[
\hat \Pi: (T^* \mathscr F_{\bv, \bw})^{\hat \sigma_1} \to  \N(\mathfrak p).
\]
be the first projection.

Retain the setting from Section 6.2. In particular, $\bw_i=0$ for all $i\geq 2$. 

\begin{prop}
Let $a=1$ and  $\w= w_0$. 
Then the $\hat \sigma$ gets identified with the automorphism $\hat \sigma_1$. 
If further $w_0* \bv = \bv$, then there is a commutative diagram 
$$
\begin{CD}
\fP_{\zeta}(\bv, \bw) @>\cong >> (T^*\mathscr F_{\bv, \bw})^{\hat \sigma_1}\\
@V\pi^{\hat \sigma} VV @VV\hat \Pi V\\
\fP_{1}(\bv, \bw) @>>> \N(\mathfrak p)
\end{CD} 
$$ 
\end{prop}

\begin{proof}
The proof is  the same as that of Theorem 6.2.1 with  minus signs removed.
\end{proof}

\begin{rem}
The above identification implies that $\pi^{\hat \sigma}$ is not semismall in general. 
For example, when the form on $W$ is symplectic and $(\bv, \bw)=(1, 2)$, 
the map  $\hat \Pi$ is the projection from the projective line $\mbb P^1$ to  a point. 
\end{rem}

Now we discuss the $\hat \sigma$ counterpart of the results in Section 8. 
Recall that it is assumed that forms on $V$ and $W$ are $\delta$-forms.
We define a non-degenerate bilinear form $\{ - | - \}$ on $\widetilde V_i$ by
\begin{align}
\left \{ (v_i, w^{(h)}_j)_{j\geq i+h} \middle|\:   (v'_i, u^{(h)}_{j})_{j\geq i + h} \right \}_{\widetilde V_i} 
=(v_i, v'_i)_{V_i} + \sum_{j\geq i+h} \left ( w^{(h)}_j, u^{(j-i+1-h)}_j \right )_{W_j} 
\end{align}
where $v, v'\in V_i$ and $w^{(h)}_j, u^{(h)}_j\in W^{(h)}_j$ such that $1\leq h\leq j-i$.

\begin{lem}
If the forms on $W$ are uniform $\delta$-forms, then so is the form on $\widetilde W_1$.
If further $V$ and $W$ are uniform $\delta$-form, then so is the form on $\widetilde V_i$, similar to $\tilde \tau$ in Proposition 8.1.2.  
\end{lem}

With the form $\{-|-\}$, one can define the automorphism $\hat \tau_{\{-|-\}}$ on $\Lambda(\widetilde V, \widetilde W)$. 

\begin{prop}
The following diagram is commutative.
\[
\begin{CD}
\Lambda(V, W) @>\Phi>> \Lambda(\widetilde V, \widetilde W) \\
@V\hat \tau VV @VV\hat \tau_{\{-|-\}} V\\
\Lambda(V, W) @>\Phi>> \Lambda(\widetilde V, \widetilde W) \\
\end{CD}
\]
\end{prop}

\begin{proof}
The proof is the same as that of Proposition 8.1.2 with minus signs removed at appropriate places.
\end{proof}

Following the line of arguments in Section 8, we reach at the identification of $\hat \sigma$-quiver varieties with nilpotent Slodowy slices
in symmetric space $\mathfrak p$, a counterpart of Theorem 8.3.2. 
Let $\hat \sigma_{\{-|-\}}$ denote the automorphism on $T^* \mathscr F_{\widetilde \bv, \widetilde \bw}$ defined with respect to the form $\{-|-\}$.

\begin{thm}
Assume that $\zeta = (\xi, 0)$ with $\xi_i>0$ for all $i\in I$ and $\theta \zeta = \zeta$.
Assume also that $w_0* \bv = \bv$. Then there is a commutative diagram
\[
\begin{CD}
\fP_{\zeta}(\bv, \bw) @>\cong >> \widetilde{\mathcal S}^{\hat \sigma_{\{-|-\}}}_{\mu', \lambda}\\
@V \pi^{\hat \sigma} VV @V\hat \Pi VV\\
\fP_1 (\bv, \bw) @>>> \mathcal S^{\hat \sigma_{\{-|-\}}}_{\mu', \lambda}
\end{CD}
\]
\end{thm}

\begin{rem}
(1) One still has a counterpart of Corollary 8.3.3, with $\delta_{\bw}$ alternates replaced by $\delta_{\bw}$ being uniform.

(2) One still has a rectangular symmetry similar to Theorem 8.4.1. 
Note that the $\hat \sigma$ therein is not the same as the $\hat\sigma$ in this section.

(3) One still has the column/row removal reduction in the symmetric space setting, similar to Propositions~\ref{i-prop:col},~\ref{i-prop:row}. 
This is a refinement of  results in~\cite{O86} (see also~\cite{O91}).

(4) Via the Kostant-Sekiguchi-Vergne correpondence~\cite{S87},~\cite{Ve95} 
and the works of Barbasch-Sepanski~\cite[Theorem 2.3]{BS98} and Chen-Nadler~\cite{CN18}, 
we see that results in the preceding remarks can be transported, at least diffeomorphically,
onto the nilpotent Slowdowy slices of the associated real groups.
\end{rem}

Finally, we return to study the relationship between fixed-points subvarieties  and categorical quotients. 
We assume that forms on $V$ and $W$ are uniform. We can consider the fixed-point subvariety $\bM(\bv, \bw)^{\hat \tau}$ under $\hat \tau$.
We can define the $\G^{\tau}_{\bv}$-invariant functions $\mrm{tr}_{h_1, \cdots, h_s}(-)$ and $\chi_{h_1, \cdots, h_s}(-)$
in exactly the same manner as (111) and (112). 
Then using a similar argument as the proof of Theorem 9.1.1, we have

\begin{prop}
The algebra of $\G^{\tau}_{\bv}$-invariant regular functions on $\bM(\bv, \bw)^{\hat \tau}$ is 
generated by the functions $\mrm{tr}_{h_1,\cdots, h_s}(-)$ and $\chi_{h_1,\cdots, h_s}(-)$ for various paths $h_1,\cdots, h_s$.
\end{prop}

From the above proposition, we have 

\begin{prop}
There is a closed immersion $\hat \iota: \Lambda(\bv, \bw)^{\hat \tau}// \G^{\tau}_{\bv} \to \fP_1(\bv, \bw)$ with $a=1$.
\end{prop}

\end{document}